\documentclass[11pt]{amsart}
\usepackage[pdftex,margin=1.1in, marginpar=.8in]{geometry}

\usepackage{amsmath}
\usepackage{amsfonts}
\usepackage{amssymb}
\usepackage{amsthm}

\usepackage[colorlinks=true,urlcolor=blue,linkcolor=blue,citecolor=blue,pagebackref,breaklinks]{hyperref}
\usepackage{cite}

\usepackage{mathrsfs}
\usepackage{bm}
\usepackage{dsfont}
\usepackage{upgreek}
\usepackage{xcolor}
\usepackage{bbold}

\usepackage{graphicx}
\usepackage[all,cmtip]{xy}
\usepackage{blkarray}
\usepackage{array}
\usepackage{todonotes}
\usepackage{enumitem}  

\setlist[enumerate,1]{label = (\arabic*),itemsep=.5em, topsep=.5em}

\usepackage{leftidx}
\usepackage{mathdots}
\usepackage{arydshln}
\usepackage{stmaryrd}  
\SetSymbolFont{stmry}{bold}{U}{stmry}{m}{n}

\newenvironment{psm}
{\left(\begin{smallmatrix}}
{\end{smallmatrix}\right)}

\newenvironment{bsm}
{\left[\begin{smallmatrix}}
{\end{smallmatrix}\right]}

\numberwithin{equation}{subsection}

\newtheorem{thm}{Theorem}[subsection]
\newtheorem{prop}[thm]{Proposition}

\theoremstyle{remark}
\newtheorem{rmk}[thm]{Remark}

\newcommand{\mr}[1]{\mathrm{#1}}

\newcommand{\bA}{\mathbb{A}}

\newcommand{\bC}{\mathbb{C}}

\newcommand{\bH}{\mathbb{H}}

\newcommand{\bQ}{\mathbb{Q}}
\newcommand{\bR}{\mathbb{R}}
\newcommand{\bS}{\mathbb{S}}

\newcommand{\bU}{\mathbb{U}}

\newcommand{\bZ}{\mathbb{Z}}

\newcommand{\bba}{\mathbb{a}}
\newcommand{\bbb}{\mathbb{b}}
\newcommand{\bbc}{\mathbb{c}}
\newcommand{\bbd}{\mathbb{d}}

\newcommand{\cA}{\mathcal{A}}
\newcommand{\cB}{\mathcal{B}}

\newcommand{\cD}{\mathcal{D}}
\newcommand{\cE}{\mathcal{E}}
\newcommand{\cF}{\mathcal{F}}

\newcommand{\cH}{\mathcal{H}}
\newcommand{\cI}{\mathcal{I}}

\newcommand{\cK}{\mathcal{K}}
\newcommand{\cL}{\mathcal{L}}
\newcommand{\cM}{\mathcal{M}}
\newcommand{\cN}{\mathcal{N}}
\newcommand{\cO}{\mathcal{O}}

\newcommand{\cS}{\mathcal{S}}

\newcommand{\cU}{\mathcal{U}}

\newcommand{\cW}{\mathcal{W}}

\newcommand{\fa}{\mathfrak{a}}
\newcommand{\fb}{\mathfrak{b}}
\newcommand{\fc}{\mathfrak{c}}
\newcommand{\fd}{\mathfrak{d}}

\newcommand{\fp}{\mathfrak{p}}

\newcommand{\fv}{\mathfrak{v}}
\newcommand{\fw}{\mathfrak{w}}
\newcommand{\fx}{\mathfrak{x}}
\newcommand{\fy}{\mathfrak{y}}
\newcommand{\fz}{\mathfrak{z}}

\newcommand{\fA}{\mathfrak{A}}
\newcommand{\fB}{\mathfrak{B}}
\newcommand{\fC}{\mathfrak{C}}
\newcommand{\fD}{\mathfrak{D}}

\newcommand{\fX}{\mathfrak{X}}

\newcommand{\sL}{\mathscr{L}}

\DeclareMathAlphabet{\mathpzc}{OT1}{pzc}{m}{it}

\newcommand{\pzM}{\mathpzc{M}}

\newcommand{\bfP}{\mathbf{P}}

\newcommand{\bbeta}{\bm{\beta}}

\DeclareMathOperator{\GL}{GL}
\DeclareMathOperator{\GSp}{GSp}

\DeclareMathOperator{\GU}{GU}

\DeclareMathOperator{\U}{U}

\DeclareMathOperator{\SL}{SL}
\DeclareMathOperator{\Sp}{Sp}
\DeclareMathOperator{\SO}{SO}
\DeclareMathOperator{\SU}{SU}

\newcommand{\cont}{\mathrm{cont}}

\newcommand{\im}{\mathrm{Im}}

\newcommand{\Nm}{\mathrm{Nm}}
\newcommand{\ord}{\mathrm{ord}}

\newcommand{\sgn}{\mathrm{sgn}}

\newcommand{\Tr}{\mathrm{Tr}}
\newcommand{\triv}{\mathrm{triv}}

\newcommand{\val}{\mathrm{val}}
\newcommand{\vol}{\mathrm{vol}}

\newcommand{\Sieg}{\mr{Sieg}}
\newcommand{\Kling}{\mr{Kling}}

\DeclareMathOperator{\Her}{Her}
\DeclareMathOperator{\Hom}{Hom}
\DeclareMathOperator{\Ind}{Ind}

\DeclareMathOperator{\Lie}{Lie}

\DeclareMathOperator{\Sym}{Sym}

\newcommand{\adic}{\text{-}\mathrm{adic}}
\newcommand{\qexp}{q\text{-}\mathrm{exp}}

\newcommand{\lhra}{\ensuremath{\lhook\joinrel\longrightarrow}}
\newcommand{\lra}{\longrightarrow}
\newcommand{\ra}{\rightarrow}

\newcommand{\ol}{\overline}

\newcommand{\llb}{\llbracket}
\newcommand{\rrb}{\rrbracket}

\newcommand{\bid}{\mathbf{1}}

\newdir{d}{{\subset}}   

\newcommand{\genlegendre}[4]{%
  \genfrac{(}{)}{}{#1}{#3}{#4}%
  \if\relax\detokenize{#2}\relax\else_{\!#2}\fi
}

\makeatletter
\def\l@section{\@tocline{1}{0pt}{1pc}{}{}}
\def\l@subsection{\@tocline{2}{0pt}{1pc}{4.6em}{}}
\def\l@subsubsection{\@tocline{3}{0pt}{1pc}{7.6em}{}}
\renewcommand{\tocsection}[3]{%
  \indentlabel{\@ifnotempty{#2}{\makebox[2.3em][l]{%
    \ignorespaces#1 #2.\hfill}}}#3}
\renewcommand{\tocsubsection}[3]{%
  \indentlabel{\@ifnotempty{#2}{\hspace*{2.3em}\makebox[2.3em][l]{%
    \ignorespaces#1 #2.\hfill}}}#3}
\renewcommand{\tocsubsubsection}[3]{%
  \indentlabel{\@ifnotempty{#2}{\hspace*{4.6em}\makebox[3em][l]{%
    \ignorespaces#1 #2.\hfill}}}#3}
\makeatother

\newcommand{\bc}{\textnormal{big-cell}}
\newcommand{\sym}{\rm{sym}}
\newcommand{\alt}{\rm{alt}}

\newcommand{\mf}{f}
\newcommand{\dsec}{{\tt f}}
\newcommand{\sph}{\mr{sph}}
\newcommand{\Bes}{\cB}
\newcommand{\bes}{B}
\newcommand{\Whi}{\cW}
\newcommand{\whi}{W}
\newcommand{\wal}{I}
\newcommand{\alphaS}{\alpha_{\mathbb{S}}}
\newcommand{\alphabS}{\bar{\alpha}_{\mathbb{S}}}
\newcommand{\Schw}{\Phi}
\newcommand{\sPi}{\scriptscriptstyle{\Pi}}
\newcommand{\sLambda}{\scriptscriptstyle{\Lambda}}
\newcommand{\disc}{\mathrm{disc}}

\newcommand{\bdot}{\boldsymbol{\cdot}}

\DeclareMathOperator*{\motimes}{\text{\raisebox{0.25ex}{\scalebox{0.8}{$\bigotimes$}}}}

\newcommand{\dR}{\mathrm{dR}}

\title{$p$-adic $L$-functions for $\GSp(4)\times\GL(2)$}
\author{Zheng Liu}
\address{University of California, Santa Barbara, , CA, United States}
\email{\href{mailto:zliu@math.ucsb.edu}{zliu@math.ucsb.edu}}
\date{}

\begin{document}

\begin{abstract}
We construct $p$-adic $L$-functions interpolating the critical values of the degree eight $L$-functions of $\GSp(4)\times\GL(2)$ for cuspidal automorphic representations generated by $p$-ordinary Siegel modular forms of genus two and $p$-ordinary modular forms.

\end{abstract}

\maketitle
\tableofcontents 
\numberwithin{equation}{subsection}

\section{Introduction}
Let $\Pi$ be a cuspidal automorphic representation of $\GSp(4,\bA_\bQ)$ and $\pi$ be a cuspidal automorphic representation of $\GL(2,\bA_\bQ)$.   Fix an odd prime $p$. The goal of this article is to construct a $p$-adic interpolation of the critical values of the degree eight $L$-function for $\Pi\times\pi$ twisted by finite order Dirichlet characters $\chi$. By using the integral representation for $L(s,\Pi\times\pi\times\chi)$ discovered by Furusawa \cite{Furusawa}, we prove the following theorem:

\begin{thm}\label{thm:intro}
Assume that $\Pi$ is tempered and $\Pi$ (resp. $\pi$) contains a $p$-ordinary holomorphic cuspidal Siegel modular forms $\varphi_\ord$ of weight $(l_1,l_2)$ (resp. $p$-ordinary cuspidal modular forms $f_\ord$ of weight $l$) with algebraic Fourier coefficients, where $l_1,l_2,l$ satisfy 
\begin{align*}
   &\min\{-l_1+l_2+l,\,l_1+l_2-l\}\geq 3.
\end{align*} 
(See \S\ref{sec:ord-at-p} for the definition of being $p$-ordinary.) Also assume that $\varphi_\ord$ (resp. $f_\ord$) is fixed by the right translation action of $\begin{bmatrix}\bid_2&\Sym_2(\bZ_v)\\0&\bid_2\end{bmatrix}$ (resp. $\begin{bmatrix}1&\bZ_v\\0&1\end{bmatrix}$) for all finite places $v$ of $\bQ$. 

Fix a choice of the following auxiliary data:
\begin{enumerate}
\item[--] $\bS=\begin{bmatrix}\bba&\frac{\bbb}{2}\\\frac{\bbb}{2}&\bbc\end{bmatrix}\in\Sym^*_2(\bZ)_{>0}$,
\item[--] a unitary Hecke character $\Lambda$ of $\cK=\bQ(\sqrt{-\det\bS})$,
\item[--] a finite set $S$ of places of $\bQ$ containing $p,\infty$ such that  for all $v\notin S$, $\varphi_\ord,f_\ord, \Lambda$ are spherical at $v$, $\bbc\in\bZ^\times_v$, $4\det\bS=\disc(\cK_v/\bQ_v)$,
\item[--] an open subgroup $U^p$ of $\hat{\bZ}^{p,\times}$ containing $\prod_{v\notin S} \bZ^\times_v$.
\end{enumerate}
\begin{enumerate}[leftmargin=1.5em,label=(\roman*)]
\item There exists a $p$-adic measure $\cL^S_{\Pi,\pi}\in\pzM eas\left(\bQ^\times\backslash \bA^\times_{\bQ,f} /U^p,\bar{\bQ}_p\right)$ satisfying the interpolation property: For all finite order characters $\chi:\bQ^\times\backslash \bA^\times_{\bQ}/U^p\ra\bC^\times$ and integers $k$ such that $s=k$ for $l_1+l_2+l$ even (resp. $s=k+\frac{1}{2}$ for $l_1+l_2+l$ odd) is a critical point for $L(s,\tilde{\Pi}\times\tilde{\pi}\times\chi)$, {\it i.e.}
\begin{equation}\label{eq:k-crt}
  \left\{\begin{array}{ll} -\frac{\min\{-l_1+l_2+l,\,l_1+l_2-l\}}{2}+2 \leq k\leq \frac{\min\{-l_1+l_2+l,\,l_1+l_2-l\}}{2}-1,&\text{$l_1+l_2+l$ even},\\[2ex]
  -\frac{\min\{-l_1+l_2+l,\,l_1+l_2-l\}}{2}+2 \leq k+\frac{1}{2}\leq \frac{\min\{-l_1+l_2+l,\,l_1+l_2-l\}}{2}-1,&\text{$l_1+l_2+l$ odd},\end{array}\right.
\end{equation}
we have
\begin{align*}
   \cL^S_{\Pi,\pi}\left((\chi|\bdot|^k)_{p_\adic}\right)
   = &\,\frac{\bes^\dagger_{\bS,\Lambda}\left(\varphi_{\ord}\right)
   \,\whi_\bbc(f_\ord)}{\bfP(\varphi_\ord,\varphi_\ord)\,\bfP(f_\ord,f_\ord)} \cdot I_\infty(k,\Pi_\infty,\pi_\infty,\Lambda_\infty)\\[1ex]
     &\times  
   \left\{\begin{array}{ll}
   E_p\left(k,\tilde{\Pi}\times\tilde{\pi}\times\chi\right)
   \cdot L^S\left(k,\tilde{\Pi}\times\tilde{\pi}\times\chi\right),&\text{$l_1+l_2+l$ even},\\[2ex]
   E_p\left(k+\frac{1}{2},\tilde{\Pi}\times\tilde{\pi}\times\chi\right)
   \cdot L^S\left(k+\frac{1}{2},\tilde{\Pi}\times\tilde{\pi}\times\chi\right),&\text{$l_1+l_2+l$ odd}.
   \end{array}\right.
\end{align*}
Here $\tilde{\Pi}$ (resp. $\tilde{\pi}$) denotes the contragredient representation of $\Pi$ (resp. $\pi$), $E_p(s,\tilde{\Pi}\times\tilde{\pi}\times\chi)$ is the modified Euler factor at $p$ for $p$-adic interpolation defined in \cite{CoPerrin,Coates} (see \eqref{eq:Ep} for an explicit formula in our case here), and $I_\infty(k,\Pi_\infty,\pi_\infty,\Lambda_\infty)$ is an archimedean integral \eqref{eq:azI}. $\bes^\dagger_{\bS,\Lambda}(\varphi_\ord)$ is the Bessel period indexed by $(\bS,\Lambda)$ with a modification at $p$ defined in \eqref{eq:bes-dagger},  $\whi_\bbc(f_\ord)$ is  the Whittaker period indexed by $\bbc$ defined in \eqref{eq:Wcf}, and $\bfP(\varphi_\ord,\varphi_\ord)$, $\bfP(f_\ord,f_\ord)$ are defined in terms of Petersson inner product in \eqref{eq:Pet}\eqref{eq:Pet2}.

\item When $l_1=l_2=l$, for $\chi,k$ as above, we have the more precise interpolation formula:
\begin{equation*}
\begin{aligned}
   &\cL^S_{\Pi,\pi}\left((\chi|\bdot|^k)_{p_\adic}\right)
   = \,\bbc\sqrt{\det\bS}\, 2^{-6}\cdot \frac{\bes^\dagger_{\bS,\Lambda}\left(\varphi_{\ord}\right)
   \,\whi_\bbc(f_\ord)}{ \bfP(\varphi_\ord,\varphi_\ord)\,\bfP(f_\ord,f_\ord)} \\[1ex]
   &\times 
   \left\{\begin{array}{ll}
   E_\infty\left(k,\tilde{\Pi}\times\tilde{\pi}\times\chi\right)
   E_p\left(k,\tilde{\Pi}\times\tilde{\pi}\times\chi\right)
   \cdot  L^S\left(k,\tilde{\Pi}\times\tilde{\pi}\times\chi\right), &\text{$l_1+l_2+l$ even},\\[2ex] 
   i\,E_\infty\left(k+\frac{1}{2},\tilde{\Pi}\times\tilde{\pi}\times\chi\right)
   E_p\left(k+\frac{1}{2},\tilde{\Pi}\times\tilde{\pi}\times\chi\right)
   \cdot  L^S\left(k+\frac{1}{2},\tilde{\Pi}\times\tilde{\pi}\times\chi\right), &\text{$l_1+l_2+l$ odd},
   \end{array}\right.
\end{aligned}
\end{equation*}
where $E_\infty(s,\tilde{\Pi}\times\tilde{\pi}\times\chi)$) is the modified Euler factor at $\infty$ for $p$-adic interpolation defined in \cite{CoPerrin,Coates} (see \eqref{eq:Einf} for an explicit formula for our case here).

\item If $\Pi_p$ does not belong to IIa or IVa in the classification in \cite{RSLocal}, then there exists a choice of $(\bS,\Lambda)$ and $f_\ord$ such that $\bes^\dagger_{\bS,\Lambda}\left(\varphi_{\ord}\right)
\whi_\bbc(f_\ord)\neq 0$.
\end{enumerate}
\end{thm}

\begin{rmk}
The construction of $p$-adic $L$-functions for $\GSp(4)\times\GL(2)$ has also been studied in \cite{Agarwal,LPSZ,LoZe,LoRi}. The region of the weights in the above theorem is the same as those considered in \cite{Agarwal,LoRi}, and is different from those in \cite{LPSZ,LoZe}. The cyclotomic variable is not included in the constructions in \cite{Agarwal,LoRi,LoZe}.  

The construction in \cite{Agarwal} uses the same integral representation as our construction, but other than missing the cyclotomic variable, the interpolation formula also contains a factor of local zeta integrals at $p$, which is not computed or shown to be nonzero. 

\cite{LPSZ,LoZe,LoRi} uses a different automorphic integral studied \cite{Novo,Soudry}. This integral involves Eisenstein series on $\GL(2)$ with Whittaker models of representations of $\GSp(4)$, appearing in the local zeta integrals,  while the automorphic integral we use involves Klingen Eisenstein series on $\GU(2,2)$ with Bessel models of representations of $\GSp(4)$ appearing in the local zeta integrals. In some sense, the automorphic integral we use here is computationally more complicated than the one used in \cite{LPSZ,LoZe,LoRi}. On the other hand, because of the appearance of Whittaker models, the construction in {\it loc.cit} needs to use non-holomorphic Siegel modular forms and higher coherent cohomologies for $\GSp(4)$. Our construction uses holomorphic Siegel modular forms and can bring technical convenience for potential arithmetic applications (for example studying congruences and Euler systems). 
\end{rmk}

\begin{rmk}
Furusawa's integral has been used for studying the algebraicity of critical values of $L(s,\Pi\times\pi)$ in \cite{Furusawa,PiSch,Saha,Pitale,Saha-pullback} with assumptions mostly resulting from the difficulty of choosing test vectors and calculating the corresponding local zeta integrals. The main difference between our approach for handling the local zeta integrals is that by using Garrett's generalization of the doubling method \cite{GaKl}, we reduce the study of the local zeta integral in \cite{Furusawa} to another local integral \eqref{eq:Zv-def} involving degenerate principal series on $\GU(3,3)$, which turns out advantageous for studying the local zeta integral for our purpose. (Garrett's generalization of the doubling method is also used in \cite{Saha-pullback}, but instead of reducing to studying the local integral \eqref{eq:Zv-def}, the author computes the local integrals for Garrett's generalization of the doubling method and shows that the corresponding local sections for the Klingen Eisenstein series agrees with the ones used in \cite{PiSch,Saha} to compute local integrals in \cite{Furusawa}.) 
\end{rmk}

Our main focus in this paper is to study, from the perspective of $p$-adic interpolations, the local zeta integrals in Furusawa's formula at finite places with minimal assumptions on $\Pi_v,\pi_v$. (Although the methods and calculations here are ready for use to construct four-variable $p$-adic $L$-functions for Hida families on $\GSp(4)$ and $\GL(2)$, we do not purse this generality here to avoid adding extra length to this paper.) For the construction of $p$-adic $L$-functions, we address the following two aspects: 
\begin{itemize}[left= 10pt]
\item Choose nice sections at $p$ for $p$-adic interpolations and compute the corresponding local zeta integrals in a uniform way for all $\Pi_p,\pi_p$ appearing as local components of automorphic representations of $\GSp(4,\bA_\bQ)$, $\GL(2,\bA_\bQ)$ which are ordinary at $p$, without assuming conditions on $(\bS,\Lambda)$ at $p$.

\item Show that at ramified finite places, a convenient section can be chosen such that the local zeta integral is computable for all $\Pi_v,\pi_v$, without assuming conditions on $(\bS,\Lambda)$. 
\end{itemize}
When computing local zeta integrals at $p$ for $p$-adic interpolations, there are often two types of simplifications. One is to assume that the local representations are unramified, and the other is to assume that certain characters are sufficiently ramified. The factors coming out from the computations for these two types have apparently different shapes,  but both can be interpreted as products of $\gamma$-factors as explained in the work of Coates and Perrin-Riou \cite{CoPerrin,Coates}. Thus one expects that a nice strategy for the computation handles all the cases uniformly by exploiting the feature of $\bU_p$-eigenvectors with nonzero eigenvalues. The local zeta integrals in Furusawa's formula at ramified finite places have been studied in \cite{PiSch,Saha,Pitale} with assumptions on $\Pi_v,\pi_v$ as well as on $(\bS,\Lambda)$, and applied to extend the results on algebraicity of critical values of $L(s,\Pi\times\pi)$ in \cite{Furusawa}. Their assumptions on $(\bS,\Lambda)$ causes the assumption on the existence of a nonzero so-called fundamental Fourier coefficient (as defined in \cite[(0.1)]{Furusawa}) for a form in $\Pi$ of a specific level. This existence has been proved in \cite{Saha-FundFC,SaSch} for scalar-weight forms of square-free level. Our choice of test sections allows us to bypass this issue.

\vspace{1em}
\noindent{\bf Notation.} We fix an odd prime number $p$.

We use $v$ to denote a place of $\bQ$. We fix the additive character 
\[
   \psi_{\bA_\bQ}=\bigotimes_v\psi_v:\bQ\backslash\bA\ra\bC^\times,
   \qquad \psi_v(x)=\left\{\begin{array}{ll}e^{-2\pi i\{x\}_v},& v\neq\infty\\e^{2\pi ix},&v=\infty\end{array},\right.
\] 
where $\{x\}_v$ is the fractional part of $x$. 

For an imaginary quadratic field $\cK$ and a finite place $v$ of $\bQ$, we write 
\begin{align*}
   \cK_v&=\cK\otimes_\bQ\bQ_v,
   &\cO_{\cK,v}&=\cO_\cK\otimes_\bZ\bZ_v.
\end{align*}
We will use $\Lambda,\Upsilon,\Xi$ to denote Hecke characters of $\cK^\times\backslash\bA^\times_\cK$. We denote by $\Lambda^c,\Upsilon^c,\Xi^c$ their compositions with the complex conjugation of $\cK/\bQ$, and write $\Lambda_{\bQ},\Upsilon_\bQ,\Xi_\bQ$ to mean their restrictions to $\bQ^\times\backslash\bA^\times_\bQ$, and $\Lambda_v,\Upsilon_v,\Xi_v$ to mean their restrictions to $\cK\otimes\bQ_v$. For any algebra $R$ and $\fz\in \cO_\cK\otimes_\bZ R$, denote by $\bar{\fz}$ the element obtained from $\fz$ by applying the complex conjugation of $\cK/\bQ$ to the first factor.

For a positive integer $n$, define the algebraic group $\GSp(2n)$ over $\bZ$ as
\begin{align*}
   \GSp(2n,R)=\left\{g\in\GL(2n,R):\ltrans{g}\begin{bmatrix}0&\bid_n\\-\bid_n&0\end{bmatrix}g=\nu_g\begin{bmatrix}0&\bid_n\\-\bid_n&0\end{bmatrix}, \,\nu_g\in R^\times\right\}
\end{align*}
for all $\bZ$-algebra $R$. Given an imaginary quadratic field $\cK$, define the algebraic group $\GU(n,n)$ over $\bZ$ as
\begin{align*}
   \GU(n,n)(R)=\left\{g\in\GL(2n,\cO_\cK\otimes R):\ltrans{\bar{g}}\begin{bmatrix}0&\bid_n\\-\bid_n&0\end{bmatrix}g=\nu_g\begin{bmatrix}0&\bid_n\\-\bid_n&0\end{bmatrix}, \,\nu_g\in R^\times\right\}.
\end{align*}
For a $\bZ$-algebra $R$, we put
\begin{align*}
    \Her_n(\cO_\cK\otimes_\bZ R)&=\{\bbeta\in M_n(\cO_\cK\otimes_\bZ R):\ltrans{\bar{\bbeta}}=\bbeta\},
    &\Sym_n(R)&=\{\beta\in M_n(R):\ltrans{\beta}=\beta\},
\end{align*}
and if $R$ is an integral domain, we put
\[
   \Sym^*_n(R)=\left\{\begin{bsm}\beta_{11}&\frac{\beta_{12}}{2}&\cdots&\frac{\beta_{1n}}{2}\\ \frac{\beta{12}}{2}&\beta_{22}&\cdots&\frac{\beta_{2n}}{2}\\ \vdots&\vdots & \ddots&\vdots\\ \frac{\beta_{1n}}{2}&\frac{\beta_{2n}}{2}&\cdots&\beta_{nn}\end{bsm}\in \Sym_n(\mr{Frac}(R)):\beta_{ij}\in R \text{ for all $i,j$}\right\}.
\]

For an algebraic group $G$ over $\bQ$, we write $[G]$ to denote $G(\bQ)Z_G(\bQ)\backslash G(\bA)$. 

Given a character $\theta:\bQ^\times_v\ra \bC^\times$ or a character $\Theta:\cK^\times_p\ra\bC^\times$, we let 
\begin{align}
  \label{eq:char-circ} \theta^\circ&=\mathds{1}_{\bZ^\times_p}\cdot \theta,
   &\Theta^\circ&=\mathds{1}_{\cO^\times_{\cK,p}}\cdot \Theta.
\end{align}


\vspace{1em}
\noindent{\bf Acknowledgment.} I would like to thank Ming-Lun Hsieh for suggesting considering Furusawa's automorphic integral for $\GSp(4)\times\GL(2)$, and Bin Xu, Hang Xue for useful discussions. During the preparation of this paper, the author was partially supported by the NSF grant DMS-2001527, and by the NSF grant DMS-1440140 while the author was in residence at MSRI during the spring of 2023 for the program ``Algebraic Cycles, L-Values, and Euler Systems''.

\section{Furusawa's formula and the doubling method}
Let $\Pi$ be an irreducible cuspidal automorphic representation of $\GSp(4,\bA_\bQ)$ and $\pi$ be an irreducible cuspidal automorphic representation of $\GL(2,\bA_\bQ)$. Denote by $\tilde{\Pi},\tilde{\pi}$ their contragredient representations. Let $\chi$  be a finite order character of $\bQ^\times\backslash\bA^\times_\bQ$. By combining Furusawa's formula for the degree eight $L$-functions for $\GSp(4)\times\GL(2)$ \cite{Furusawa} and Garrett's generalization of the doubling method for Klingen Eisenstein series on unitary groups \cite{GaKl}, we express $L(s,\tilde{\Pi}\times\tilde{\pi}\times\chi)$ as an integral involving a form in $\Pi$, a form in $\pi$ and a Siegel Eisenstein series on $\GU(3,3)$. Then we make some preliminary calculations of the local zeta integrals, which will be used for proving the interpolation property of the later constructed $p$-adic $L$-function.

\subsection{The integral representation for $L$-functions for $\GSp(4)\times\GL(2)$}

First, we introduce some preliminaries: Bessel models for $\GSp(4)$, forms on $\GL(2)$ and $\GU(1,1)$ (mostly following Section 1 in \cite{Furusawa}), and Siegel Eisenstein series on $\GU(3,3)$. Then we prove the formula in Theorem~\ref{thm:Furusawa} which we will use for constructing $p$-adic $L$-functions.

\subsubsection{Bessel models for $\GSp(4)$}
Take 
\begin{align*}
  \bS=\begin{bmatrix}\bba&\frac{\bbb}{2}\\\frac{\bbb}{2}&\bbc\end{bmatrix}\in\Sym_2(\bQ)
\end{align*}
and put $\cK=\bQ(\sqrt{-\det \bS})$. Fix a square root $\sqrt{\bbb^2-4\bba\bbc}\in\cK$ and  put
\[
   \alphaS=\frac{\bbb+\sqrt{\bbb^2-4\bba\bbc}}{2\bbc}.
\] 
Let 
\begin{align*}
   T_\bS&=\left\{A\in\GL(2):\ltrans{A}\bS A=\det A\cdot\bS\right\}.
\end{align*}
We identify it with the subgroup 
\begin{align*}
   \left\{\begin{bmatrix}A\\&\det A\cdot \ltrans{A}^{-1}\end{bmatrix}:A\in T_\bS\right\}\subset \GSp(4),
\end{align*}
which we will also denote by $T_{\bS}$. For $\fz\in \cK\otimes_\bQ R$ with $R$ any $\bQ$-algebra, let
\begin{align*}
      \imath_\bS\left(\fz\right)
   =\begin{bmatrix}1&1\\-\alphabS&-\alphaS\end{bmatrix}\begin{bmatrix}\fz\\&\bar{\fz}\end{bmatrix}\begin{bmatrix}1&1\\-\alphabS&-\alphaS\end{bmatrix}^{-1}
   =\begin{bmatrix}\alphaS&1\\ \alphabS&1\end{bmatrix}^{-1} \begin{bmatrix}\fz\\&\bar{\fz}\end{bmatrix}\begin{bmatrix}\alphaS&1\\ \alphabS&1\end{bmatrix} \in M_2(R).
\end{align*} 
Then we have the isomorphism
\begin{align*}
   &\mr{Res}_{\cK/\bQ}\GL(1)\stackrel{\sim}{\lra} T_\bS,
   &&\fz\longmapsto \begin{bmatrix}\imath_\bS(\fz)\\&\ltrans{\imath_\bS(\bar{\fz})}\end{bmatrix}.
\end{align*}
Denote by $U^\Sieg_{\GSp(4)}$ the unipotent subgroup of $\GSp(4)$ consisting of $\begin{bmatrix}\bid_2&X\\&\bid_2\end{bmatrix}$ with $\ltrans{X}=X$. Put
\begin{align*}
  R_\bS&=T_\bS U^{\Sieg}_{\GSp(4)}.
\end{align*}
Given a Hecke character $\Lambda:\cK^\times\backslash\bA^\times_\cK\ra\bC^\times$, define the character $\Lambda_\bS:R(\bA_\bQ)\ra\bC^\times$ as
\begin{align*}
   \Lambda_\bS\left(\begin{bmatrix}\imath_\bS(\fz)\\&\ltrans{\,\ol{\imath_\bS(\bar{\fz})}}\end{bmatrix}\begin{bmatrix}\bid_2&X\\&\bid_2\end{bmatrix}\right)=\Lambda(\fz)\psi_{\bA_\bQ}(\Tr\, \bS X).
\end{align*}

Suppose that $\varphi$ be a cuspidal automorphic form on $\GSp(4)$ with central character equal to $\Lambda_\bQ$. The global {\it Bessel period} for $\varphi$ with respect to $(\bS,\Lambda)$ (and the fixed additive character $\psi_{\bA_\bQ}$) is defined to be
\begin{align*}
   \bes_{\bS,\Lambda}(\varphi)=\int_{[R_\bS]}\Lambda^{-1}_\bS(r)\cdot\varphi(r)\,dr.
\end{align*}
We also define the function $\Bes_{\bS,\Lambda}(\varphi)$ on $\GSp(4,\bA_\bQ)$ as
\begin{align*}
   \Bes_{\bS,\Lambda}(\varphi)(g)=\bes_{\bS,\Lambda}(g\cdot \varphi)=\int_{[R_\bS]}\Lambda^{-1}_\bS(r)\cdot\varphi(rg)\,dr.
\end{align*}
For an irreducible cuspidal automorphic representation $\Pi$ of $\GSp(4,\bA_\bQ)$ with central character $\omega_\Pi=\Lambda_\bQ$, the subspace of functions on $\GSp(4,\bA)$ spanned by $\{\Bes_{\bS,\Lambda}(\varphi):\varphi\in\Pi\}$ is called the space of the global {\it Bessel model} of type $(\bS,\psi,\Lambda)$ for $\Pi$.

\vspace{1em}
Fix an isomorphism $\Pi\cong \bigotimes_v\Pi_v$. By \cite{NoPS},
\[
   \dim \Hom_{R_\bS(\bQ_v)}\left(\Pi_v,\Lambda_{\bS,v}\right)\leq 1.
\]
Suppose that $\Pi$ has a nontrivial global Bessel model. Then $\Hom_{R(\bQ_v)}\left(\Pi_v,\Lambda_{\bS,v}\right)$ is one-dimensional for all $v$. From it, we pick $\bes^{\Pi_v}_{\bS,\Lambda_{v}}$ and denote by $\Bes^{\Pi_v}_{\bS,\Lambda_v}(\varphi_v)$ the function on $\GSp(4,\bQ_v)$ sending $g_v$ to $\bes^{\Pi_v}_{\bS,\Lambda_v}(g_v\cdot\varphi_v)$. Then there exists nonzero $C_{\bS,\Lambda,\Pi}\in\bC$ such that for all $\varphi\in\Pi$ with image $\bigotimes_v\varphi_v$ under our fixed isomorphism $\Pi\cong\bigotimes_v\Pi_v$, we have 
\begin{equation}\label{eq:Bvarphi}
   \Bes_{\bS,\Lambda}(\varphi)(g)=C_{\bS,\Lambda,\Pi}\cdot \prod_v \Bes^{\Pi_v}_{\bS,\Lambda_v}(\varphi_v)(g_v).
\end{equation}

\subsubsection{Automorphic forms on $\GL(2)$ and $\GU(1,1)$}
We have the isomorphism
\begin{align*}
   \left(\mr{Res}_{\cK/\bQ}\GL(1)\right)\times\GL(2)/\{(a,a^{-1}\cdot\bid_2):a\in\GL(1)\}&\stackrel{\sim}{\lra} \GU(1,1)\\
   (\fa,g)&\longmapsto \fa g.
\end{align*}
Given a modular form $f$ on $\GL(2)$ and a Hecke character $\Upsilon:\cK^\times\backslash\bA^\times_\cK\ra\bC^\times$ with $\Upsilon_\bQ$ equal to the central character of $f$, we can define the form $f^\Upsilon$ on $\GU(1,1)$ as
\begin{align*}
   \mf^\Upsilon(\fa g)&=\Upsilon(\fa)f(g), &\fa\in\bA^\times_\cK,\,g\in\GL(2,\bA_\bQ).
\end{align*}
For an irreducible cuspidal automorphic representation $\pi$ of $\GL(2,\bA_\bQ)$ with central character $\omega_\pi=\Upsilon_\bQ$,  $\pi^\Upsilon=\{f^\Upsilon:f\in\pi\}$ is an irreducible cuspidal automorphic representation of $\GU(1,1)$.

Given $\bbc\in\bQ$, we have the global {\it Whittaker period} for $f\in\pi$:
\begin{align}
   \label{eq:Wcf}\whi_{\bbc}\left(\mf\right)&=\int_{\bQ\backslash\bA_\bQ} \mf\left(\begin{bmatrix}1&x\\&1\end{bmatrix} \right)\cdot\psi_{\bA_\bQ}(-\bbc x)\,dx.
\end{align}
We define the function $\Whi_{\bbc}\left(\mf\right)$ on $\GL(2,\bA_\bQ)$ as 
\begin{align*}
   \Whi_{\bbc}\left(\mf\right)(g)&=\whi_{\bbc}\left(g\cdot \mf\right)=\int_{\bQ\backslash\bA_\bQ} \mf\left(\begin{bmatrix}1&x\\&1\end{bmatrix} g\right)\cdot\psi_{\bA_\bQ}(-\bbc x)\,dx.
\end{align*}
Fix an isomorphism $\pi\cong\bigotimes_v\pi_v$, and for each place $v$, pick a nonzero Whittaker functional
\[
   \whi^{\pi_v}_{\bbc}\in \Hom_{U_{\GL(2)}(\bQ_v)}\left(\pi_v,\psi_{\bbc,v}\right),
\]
where $U_{\GL(2)}\subset \GL(2)$ is the standard unipotent subgroup, and $\psi_{\bbc,v}$ is the character of $U_{\GL(2)}(\bQ_v)$ sending $\begin{bmatrix}1&x\\&1\end{bmatrix}$ to $\psi_v(\bbc x)$. For $f_v\in\pi_v$, denote by $\Whi^{\pi_v}_{\bbc}(f_v)$ the function on $\GL(2,\bQ_v)$ sending $g$ to $\whi^{\pi_v}_\bbc(g\cdot f_v)$. The existence and uniqueness of the Whittaker models for the $\pi_v$'s tell us that there exists a nonzero $C_{\bbc,\pi}\in\bC$ such that for all $f\in\pi$ corresponding to $\bigotimes_v f_v\in\bigotimes_v\pi_v$, 
\begin{equation}\label{eq:Wf1}
   \Whi_\bbc(\mf)(g)=C_{\bbc,\pi}\cdot \prod_{v} \Whi^{\pi_v}_\bbc(f_v)(g_v).
\end{equation}

Let
\begin{align*}
   \Whi_{\bbc}\left(\mf^{\Upsilon}\right)(g)&=\int_{\bQ\backslash\bA_\bQ} \mf^\Upsilon\left(\begin{bmatrix}1&x\\&1\end{bmatrix} g\right)\cdot\psi_{\bA_\bQ}(-\bbc x)\,dx,
   &g\in\GU(1,1)(\bA_\bQ),
\end{align*}
and 
\begin{align*}
   \Whi^{\pi_v,\Upsilon_v}_\bbc(f_v)(\fa g)&=\Upsilon_v(\fa)\cdot \Whi^{\pi_v}_\bbc(f_v)(g),
   &(\fa,g)\in\cK_v\times\GL(2,\bQ_v).
\end{align*}
Then \eqref{eq:Wf1} gives
\begin{equation}\label{eq:Wf2}
   \Whi_\bbc(\mf^\Upsilon)(g)=C_{\bbc,\pi}\cdot \prod_{v} \Whi^{\pi_v,\Upsilon_v}_\bbc(f_v)(g_v).
\end{equation}

\subsubsection{The Siegel Eisenstein series on $\GU(3,3)$}
Let $Q_{\GU(3,3)}\subset \GU(3,3))$ be the standard Siegel parabolic subgroup consisting of elements whose lower-left $3$-by-$3$ blocks are $0$. Let $\Xi:\cK^\times\backslash\bA^\times_\cK\ra\bC^\times$ be a Hecke character, and $I_v(s,\chi,\Xi)$ be the degenerate principal series on $\GU(3,3)(\bQ_v)$ consisting of smooth functions $\dsec_v(s,\chi,\Xi):\GU(3,3)(\bQ_v)\ra\bC$ such that 
\begin{align*}
   \dsec_v(s,\chi,\Xi)\left(\begin{bmatrix}\fA&\fB\\0&\fD\end{bmatrix}g\right)=\Xi_v(\det \fA)\chi_v(\det \fA\fD^{-1})|\det \fA\fD^{-1}|^{s+\frac{3}{2}}_v\,\dsec_v(s,\chi,\Xi)(g)
\end{align*}
for all $g\in\GU(3,3)(\bQ_v)$ and $\begin{bmatrix}\fA&\fB\\0&\fD\end{bmatrix}\in Q_{\GU(3,3)}(\bQ_v)$. If $v$ is a finite place where $\chi_v,\Xi_v$ are unramified, we denote by $\dsec^\sph_v(s,\chi,\Xi)$ the unique section in $I_v(s,\chi,\Xi)$ invariant under the right translation by $\GU(3,3)(\bZ_v)$ and taking value $1$ at $\bid_6$.

The {\it Siegel Eisenstein series} associated to a section $\dsec(s,\chi,\Xi)$ inside $I(s,\chi,\Xi)=\bigotimes'_v I_v(s,\chi,\Xi)$, the restricted tensor product with respect to $\dsec^\sph_v(s,\chi,\Xi)$, is defined as
\[
    E^{\Sieg}(g;\dsec(s,\chi,\Xi))=\sum_{\gamma\in Q_{\GU(3,3)}(\bQ)\backslash \GU(3,3)(\bQ)}\dsec(s,\chi,\Xi)(\gamma g).
\]
(The sum converges for $\mr{Re}(s)\gg 0$ and has a meromorphic continuation to $s\in\bC$.)

\subsubsection{The integral representation}
Let
\[
   \GSp(4)\times_{\GL(1)}\GU(1,1)=\left\{(g,h)\in \GSp(4)\times\GU(1,1):\nu_g=\nu_h\right\}.
\]
We fix an embedding
\begin{equation}\label{eq:gp-ebd}
\begin{aligned}
   \imath:\GSp(4)\times_{\GL(1)}\GU(1,1)&\lra \GU(3,3)\\
   \begin{bmatrix}A&B\\C&D\end{bmatrix}\times\begin{bmatrix}\fa&\fb\\\fc&\fd\end{bmatrix}&\longmapsto 
   \begin{bmatrix}A&&B\\&\fa&&\fb\\C&&D\\&\fc&&\fd\end{bmatrix}.
\end{aligned}
\end{equation}

For a positive integer $n$, let
\begin{equation}\label{eq:dnv}
	d_{n,v}\big(s,\Xi(\chi\circ\Nm)\big)=\prod_{j=1}^n L_v\left(2s+j,\Xi_\bQ\chi^2\eta^{n-j}_{\cK/\bQ}\right),
\end{equation}
with $\eta_{\cK/\bQ}$ the quadratic character of $\bQ^\times\backslash\bA^\times_\bQ$ associated to the extension $\cK/\bQ$.

\begin{thm}\label{thm:Furusawa} 
Assume that $\Xi\Lambda^c\Upsilon^c=\triv$. 
\begin{enumerate}[leftmargin=1.5em,label=(\roman*)]
\item 
\begin{equation}\label{eq:factorization}
\begin{aligned}
   &\int_{[\GSp(4)\times_{\GL(1)}\GU(1,1)]} E^{\Sieg}\big(\imath(g,h);\dsec(s,\chi,\Xi)\big)\cdot \varphi(g) \cdot \mf^\Upsilon(h)\,\Xi^{-1}(\det h) \,dh\,dg\\
   =&\,C_{\bS,\Lambda,\Pi}\cdot C_{\bbc,\pi}\cdot \prod_v Z_v\Big(\dsec_v(s,\chi,\Xi),\Bes^{\Pi_v}_{\bS,\Lambda_v}(\varphi_v), \Whi^{\pi_v,\Upsilon_v}_\bbc(f_v)\Big)
\end{aligned}
\end{equation}
with 
\begin{equation}\label{eq:Zv-def}
\begin{aligned}
   &Z_v\Big(\dsec_v(s,\chi,\Xi),,\Bes^{\Pi_v}_{\bS,\Lambda_v}(\varphi_v), \Whi^{\pi_v,\Upsilon_v}_\bbc(f_v)\Big)\\
   =&\,\int_{\big(R'_\bS\backslash\GSp(4)\times_{\GL(1)}\GU(1,1)\big)(\bQ_v)} \dsec_v(s,\chi,\Xi)\left(\cS^{-1}\imath(\eta_\bS\, g,h)\right)\cdot \Bes^{\Pi_v}_{\bS,\Lambda_v}(\varphi_v)(g)\\
   &\hspace{10em}\times \Whi^{\pi_v,\Upsilon_v}_\bbc(f_v)\left(\begin{bmatrix}0&1\\-1&0\end{bmatrix}h\right)\Xi^{-1}_v(\det h)\,dhdg,
\end{aligned}
\end{equation}
where $R'_\bS \subset \GSp(4)\times_{\GL(1)}\GU(1,1)$ is the subgroup
\[
   \left\{\left(\begin{bmatrix}\imath_\bS(\fz)\\&\ltrans{\imath_\bS(\bar{\fz})}\end{bmatrix}\begin{bmatrix}\bid_2&X\\&\bid_2\end{bmatrix},\,\fz\cdot\bid_2\right):\fz\in \mr{Res}_{\cK/\bQ}\GL(1),\,X\in\Sym_2\right\},
\]
and
\begin{align*}
   \eta_\bS&=\begin{bmatrix}
  1\\ \alphaS&1\\&&1&-\alphabS\\&&&1
  \end{bmatrix},
   &\cS&=\begin{bmatrix}1\\&1\\&&1\\&&&1\\&&1&&1\\&1&&&&1\end{bmatrix}.
\end{align*}

\item Suppose that
\begin{enumerate}
\item[-] $\pi,\Pi,\Upsilon,\Xi,\chi$ are all unramified  and $\varphi_v, f_v$ are spherical at $v$,
\item[-] $\bS=\begin{bmatrix}\bba&\frac{\bbb}{2}\\\frac{\bbb}{2}&\bbc\end{bmatrix}$ belongs to $M_2(\bZ_v)$ with $\bbc\in\bZ^\times_v$ and $\bbb^2-4\bba\bbc=\mr{disc}(\cK_v/\bQ_v)$.
\end{enumerate}
Then
\begin{equation}\label{eq:Z-unram}
\begin{aligned}
   &Z_v\Big(\dsec^\sph_v(s,\chi,\Xi),,\Bes^{\Pi_v}_{\bS,\Lambda_v}(\varphi_v), \Whi^{\pi_v,\Upsilon_v}_\bbc(f_v)\Big)\\
   =&\,d_{3,v}\left(s+\frac{1}{2},\Xi(\chi\circ\Nm)\right)^{-1}L_v\left(s+\frac{1}{2},\tilde{\Pi}\times\tilde{\pi}\times\chi\right)
   \cdot \bes^{\Pi_v}_{\bS,\Lambda_v}(\varphi_v) \, \whi^{\pi_v,\Upsilon_v}_{\bbc}(f_v).
\end{aligned}
\end{equation}
\end{enumerate}

\end{thm}

\begin{proof}
(i) Let $P_{\GU(2,2)}\subset\GU(2,2)$ be the Klingen parabolic subgroup consisting of elements of the form 
\begin{align*}
  &\begin{bmatrix}\fx&\ast&\ast&\ast\\&\fa&\ast&\fb\\&&\nu\bar{\fx}^{-1}\\&\fc&\ast&\fd\end{bmatrix}, 
  &\fx\in\mr{Res}_{\cK/\bQ}\GL(1),\, \begin{bmatrix}\fa&\fb\\\fc&\fd\end{bmatrix}\in\GU(1,1) \text{ with similitude $\nu$}.
\end{align*}
Its Levi subgroup is isomorphic to $(\mr{Res}_{\cK/\bQ}\GL(1))\times\GU(1,1)$. It follows from the (generalized) doubling method \cite{GaKl} that, for $t\in \GU(1,1)(\bA_\bQ)$ with $\nu(t)=\nu(g)$,
\begin{align*}
   &\int_{[\U(1,1)]} E^\Sieg\big(\imath(g,h_1t);\dsec(s,\chi,\Xi)\big)\cdot \mf^{\Upsilon}(h_1t)\,\Xi^{-1}(\det h_1t)\,dh
\end{align*}
equals 
\[
    E^\Kling\left(g;\Phi_{\dsec(s,\chi,\Xi),\mf^{\Upsilon}}\right)=\sum_{\gamma\in P_{\GSp(4)}(\bQ)\backslash \GSp(4,\bQ)} \Phi_{\dsec(s,\chi,\Xi),\mf^{\Upsilon}}(\gamma g),
\]
the Klingen Eisenstein series on $\GU(2,2)$  associated to the section 
\[
   \Phi_{\dsec(s,\chi,\Xi),\mf^{\Upsilon}}\in\mr{Ind}_{P_{\GSp(4)}}^{\GSp(4)}\left((\Xi\boxtimes \pi^\Upsilon)\otimes(\chi|\cdot|^s_{\bA_\bQ}\circ\Nm\boxtimes \chi|\cdot|^s_{\bA_\bQ}\circ\nu^{-1})\right)
\]
given as
\begin{equation}\label{eq:Kling-sec}
   \Phi_{\dsec(s,\chi,\Xi),\mf^{\Upsilon}}(g)=\int_{\U(1,1)(\bA_\bQ)}\dsec(s,\chi,\Xi)\left(\cS^{-1}\imath(g,h_1t)\right)\mf^{\Upsilon}(h_1 t)\,\Xi^{-1}(\det h_1t)\,dh_1.
\end{equation}
Therefore, we have
\begin{align*}
   \text{LHS of \eqref{eq:factorization}}=\int_{[\GSp(4)]} E^\Kling\left(g;\Phi_{\dsec(s,\chi,\Xi),\mf^{\Upsilon}}\right)\cdot \varphi(g)\,dg.
\end{align*}
Applying \cite[(2.4) Theorem]{Furusawa} to the right hand side gives
{\small
\begin{align*}
   \text{LHS of \eqref{eq:factorization}}&=\int_{\big(R_\bS\backslash \GSp(4)\big)(\bA_\bQ)} \int_{\bQ\backslash\bA_\bQ}\Phi_{\dsec(s,\chi,\Xi),\mf^{\Upsilon}}
   \left(\begin{bmatrix}1\\&1&&x\\&&1\\&&&1\end{bmatrix}\eta_\bS\, g\right)\psi(\bbc x)\,dx  \cdot \Bes_{\bS,\Lambda}(\varphi)(g)\,dg
\end{align*}}
Plugging in the formula \eqref{eq:Kling-sec} and using that $\cS^{-1}\left(\begin{bsm}1\\&1&&x\\&&1\\&&&1\end{bsm},\begin{bsm}1\\-x&1\end{bsm}\right)\cS\in Q_{\GU(3,3)}$, we obtain
\begin{align*}
   \text{LHS of \eqref{eq:factorization}}=&\int_{\left(R'_\bS\backslash \GSp(4)\times_{\GL(1)}\GU(1,1)\right)(\bA_\bQ)} \dsec(s,\chi,\Xi)\left(\cS^{-1}\imath(\eta_\bS\, g,h\right)
   \cdot \Bes_{\bS,\Lambda}(\varphi)(g)\\
   &\hspace{12em}\times\int_{\bQ\backslash\bA_\bQ} \mf^\Upsilon\left(\begin{bmatrix}1\\x&1\end{bmatrix}h\right)\psi(\bbc x)\,dx\,dhdg\\
   =&\int_{\left(R'_\bS\backslash \GSp(4)\times_{\GL(1)}\GU(1,1)\right)(\bA_\bQ)} \dsec(s,\chi,\Xi)\left(\cS^{-1}\imath(\eta_\bS \,g,h\right)\\
   &\hspace{12em}\times \Bes_{\bS,\Lambda}(\varphi)(g)
   \cdot \Whi_{\bbc}(\mf^\Upsilon)\left(\begin{bmatrix}&1\\-1 \end{bmatrix} h\right)\,dhdg
\end{align*}
Plugging \eqref{eq:Bvarphi}\eqref{eq:Wf2} into it proves the factorization \eqref{eq:factorization}.

\vspace{.5em}

(ii) When everything is spherical and $\bS$ satisfies the condition in (ii), it follows from \cite[Theorem on p.105]{Furusawa} that
\begin{equation}\label{eq:Furu-unram}
\begin{aligned}
   &Z_v\Big(\dsec^\sph_v(s,\chi,\Xi),\Bes^{\Pi_v}_{\bS,\Lambda_v}(\varphi_v), \Whi^{\pi_v,\Upsilon_v}_\bbc(f_v)\Big)\\
  =&\,\Phi_{\dsec_v(s,\chi,\Xi),\mf^\Upsilon}(\bid_4)\cdot \frac{L_v\left(s+\frac{1}{2},\tilde{\Pi}\times\tilde{\pi}\times\chi\right)}{L_v(2s+1,\Xi_\bQ\chi^2)\,L_v\left(s+1,\mr{BC}(\pi)\times\Upsilon^{-1}\Xi(\chi\circ\Nm)\right)}.
\end{aligned}
\end{equation}
From the definition of section $\Phi_{\dsec_v(s,\chi,\Xi),\mf^\Upsilon}$, it is easy to see that its valuation at $\bid_4$ is computed by the local zeta integral of the standard doubling method. The formula in \cite[Proposition 3]{LapidRallis} gives us
\begin{equation}\label{eq:Phi-at-1}
\begin{aligned}
   &\Phi_{\dsec_v(s,\chi,\Xi),\mf^\Upsilon}(\bid_4)\\
   =&\,d_{2,v}\left(s+\frac{1}{2},\Xi(\chi\circ\Nm)\right)^{-1} L_v\left(s+1,\mr{BC}(\pi)\times\Upsilon^{-1}\Xi(\chi\circ\Nm)\right)\\
   =&\,d_{3,v}\left(s+\frac{1}{2},\Xi(\chi\circ\Nm)\right)^{-1}
   \,L_v(2s+1,\Xi_\bQ\chi^2)\,L_v\left(s+1,\mr{BC}(\pi)\times\Upsilon^{-1}\Xi(\chi\circ\Nm)\right)
\end{aligned}
\end{equation}
Combining \eqref{eq:Furu-unram} and \eqref{eq:Phi-at-1} proves \eqref{eq:Z-unram}.
\end{proof}

\subsection{Classical scalar-weight sections for Siegel Eisenstein series and a special case of archimedean zeta integrals}\label{sec:sw-dsec}
We assume that $\chi_\infty$, $\Xi_\infty$ are both trivial on $\bR_{>0}$. Suppose that $\Xi$ has $\infty$-type $\left(\frac{r}{2},-\frac{r}{2}\right)$. For an integer $t\equiv r\mod 2$, inside the degenerate principal series $I_\infty(s,\chi,\Xi)$, we have the classical section of scalar weight $t$:
\begin{equation}\label{eq:sw-dsec}
\begin{aligned}
   &\dsec^t_{\infty}(s,\chi,\Xi)\left(g=\begin{bmatrix}\fA&\fB\\ \fC&\fD\end{bmatrix}\right)\\
   =&\,\chi_\infty(\nu_g)\nu^{\frac{3}{2}(r-t)}_g|\nu_g|^{\frac{3}{2}(s+\frac{3}{2}-r)}(\det g)^{\frac{r+t}{2}}\det(\fC i+\fD)^{-t} |\det(\fC i+\fD)|^{-2s-3+t}.
\end{aligned}
\end{equation}

We can explicitly compute the zeta integrals in \eqref{eq:Zv-def} for $v=\infty$ in the special case where $\dsec_\infty(s,\chi,\Xi)=\dsec^t_\infty(s,\chi,\Xi)$, $\Pi_\infty\cong \cD_{t,t}$ and $\pi_\infty\cong \cD_t$ for some integer $t\geq 3$. Here $\cD_{t,t}$ (resp. $D_t$) denotes the holomorphic discrete series of $\GSp(4,\bR)$ (resp. of $\GL(2,\bR)$) of scalar weight $t$ with trivial central character (resp. central character $\sgn^t$).

\begin{prop}\label{eq:Zlinf}
Suppose that $\varphi_\infty\in\cD_{t,t}$ (resp. $f_\infty\in\cD_t$) belongs to the $K_\infty$-type of scalar weight $-t$. Then
\begin{align*}
    &Z_\infty\left(\dsec^t_{\infty}(s,\chi,\Xi),\Bes^{\cD_{t,t}}_{\bS,\triv}(\varphi_\infty),\Whi^{\cD_t,\Xi^{-c}_\infty}_{\bbc}(f_\infty)\right)\\
    =&\,2^{-s-\frac{3t}{2}-\frac{3}{2}}i^{-t}\pi^2(\pi\bbc)^{-s-\frac{3t}{2}+\frac{3}{2}}|\alphaS-\alphabS|^{-2s-t}
    \cdot\frac{\Gamma\left(s+\frac{3t-3}{2}\right)\Gamma\left(s+\frac{t-1}{2}\right)}{\Gamma\left(s+\frac{t+3}{2}\right)}
    \cdot\bes^{\cD_{t,t}}_{\bS,\triv}(\varphi_\infty)\, \whi^{\cD_t}_{\bbc}(f_\infty).
\end{align*}
\end{prop}

\begin{proof}
We have
\begin{align*}
   Z_\infty\left(\dsec^t_{\infty}(s,\chi,\Xi),\Bes^{\cD_{t,t}}_{\bS,\triv}(\varphi_\infty),\Whi^{\cD_t,\Xi^{-c}_\infty}_{\bbc}(f_\infty)\right)
   =\int_{(R_\bS\backslash\GSp(4))(\bR)} F^t_\infty(\eta_\bS\,g;s,\chi,\Xi)\cdot\Bes^{\cD_{t,t}}_{\bS,\triv}(\varphi_\infty)(g)\,dg
\end{align*}
with
\begin{equation}\label{eq:Fl-infty}
\begin{aligned}
   F^t_\infty(\eta_\bS\,g;s,\chi,\Xi)=\int_{\U(1,1)(\bR)} &\dsec^t_\infty(s,\chi,\Xi)\left(\cS^{-1}\imath(\eta_\bS\, g,h_1h)\right)\\
   &\times  \Whi^{\cD_t,\Xi^{-c}_\infty}_{\bbc}(f_\infty)\left(\begin{bmatrix}&1\\-1\end{bmatrix} h_1 h\right) \cdot \Xi^{-1}_\infty(\det h_1 h)\,dh_1,
\end{aligned}
\end{equation}
where $h\in \GU(1,1)(\bR)$ with $\nu_h=\nu_g$. Let
\[
   A_0=\begin{bmatrix}\frac{2}{\rule{0pt}{.8em}\sqrt{-(\alphaS-\alphabS)^2}}&0\\-\frac{\alphaS+\alphabS}{\rule{0pt}{.8em}\sqrt{-(\alphaS-\alphabS)^2}}&1\end{bmatrix}\in\GL(2,\bR)
\]
Then by the Cartan decomposition for $\GL(2,\bR)$, we have
\[
   \GSp(4,\bR)=R_\bS(\bR)\left\{\begin{bmatrix}\lambda A_0\begin{bmatrix}a\\&a^{-1}\end{bmatrix}\\ &\ltrans{A}^{-1}_0\begin{bmatrix}a^{-1}\\&a\end{bmatrix}\end{bmatrix}:\lambda\in \bR^\times, a\in\bR_{\geq 1} \right\} K_\infty
\] 
Since the right translation of $K_\infty$ on $\varphi_\infty$ (resp. $F^t_\infty(-;s,\chi,\Xi)$) is by the scalar weight $-t$ (resp. $t$),  the integral $Z_\infty$ can be rewritten as
\begin{equation}\label{eq:Z-lambdaa}
\begin{aligned}
   &Z_\infty\left(\dsec^t_{\infty}(s,\chi,\Xi),\Bes^{\cD_{t,t}}_{\bS,\triv}(\varphi_\infty),\Whi^{\cD_t,\Xi^{-c}_\infty}_{\bbc}(f_\infty)\right)\\
   =&\int_{\bR^\times} \int_1^\infty 
   F^t_\infty\left(\eta_\bS\begin{bsm}\lambda A_0\begin{bsm}a\\&a^{-1}\end{bsm}\\ &A^{-1}_0\begin{bsm}a^{-1}\\&a\end{bsm}\end{bsm};s,\chi,\Xi\right) \Bes^{\cD_{t,t}}_{\bS,\triv}(\varphi_\infty)\left(\begin{bsm}\lambda A_0\begin{bsm}a\\&a^{-1}\end{bsm}\\ &A^{-1}_0\begin{bsm}a^{-1}\\&a\end{bsm}\end{bsm};s,\chi,\Xi\right)\\
   &\hspace{20em} \times\pi(a^2-a^{-2})(\lambda\det A_0)^{-3}\,d^\times a\,d^\times\lambda.
\end{aligned}
\end{equation}
By \cite[(4.3.4)]{Furusawa},
\begin{equation}\label{eq:B-infty}
\begin{aligned}
    &\Bes^{\cD_{t,t}}_{\bS,\triv}(\varphi_\infty)\left(\begin{bsm}\lambda A_0\begin{bsm}a\\&a^{-1}\end{bsm}\\ &\ltrans{A}^{-1}_0\begin{bsm}a^{-1}\\&a\end{bsm}\end{bsm};s,\chi,\Xi\right)\\
    =&\,\bes^{\cD_{t,t}}_{\bS,\triv}(\varphi_\infty)\cdot \lambda^t\left(\frac{2}{\sqrt{-(\alphaS-\alphabS)^2}}\right)^t\psi_\infty\left(i\bbc\lambda(a^2+a^{-2})\right).
\end{aligned}
\end{equation}
We also need a formula for $F^t_\infty$. Any section $\dsec_\infty(g;s,\chi,\Xi)\in I_\infty(s,\chi,\Xi)$ satisfies the property that for all $\begin{bmatrix}\fa&\fb\\ \fc&\fd\end{bmatrix}\in\GU(1,1)(\bR)$,
\begin{align*}
   &\dsec_\infty(s,\chi,\Xi)\left(\cS^{-1}\left(\begin{bmatrix}\fx&\ast&\ast&\ast\\&\fa&\ast&\fb\\&&\nu\bar{\fx}^{-1}\\&\fc&\ast&\fd\end{bmatrix}g,\begin{bmatrix}\fd&\fc\\ \fb&\fa\end{bmatrix}h\right)\right)\\
   =&\,\Xi_\infty(\fx(\fa\fd-\fb\fc))\,\chi_\infty(\nu^{-1}\fx\bar{\fx})|\nu^{-1}\fx\bar{\fx}|^{s+\frac{3}{2}}_\infty
   \cdot \dsec_\infty(s,\chi,\Xi)\left(\cS^{-1}\imath\left(g,h\right)\right)
\end{align*}
Plugging this into \eqref{eq:Fl-infty}, we get
\begin{align*}
   F^t_\infty\left(\begin{bmatrix}\fx&\ast&\ast&\ast\\&\fa&\ast&\fb\\&&\nu\bar{\fx}^{-1}\\&\fc&\ast&\fd\end{bmatrix}\right)
   =&\,\Xi_\infty(\fx)\,\chi_\infty(\nu^{-1}\fx\bar{\fx})|\nu^{-1}\fx\bar{\fx}|^{s+\frac{3}{2}}_\infty\cdot \int_{\U(1,1)(\bR)} \dsec^t_\infty(s,\chi,\Xi)\left(\cS^{-1}\imath(\bid_4,h_1)\right) \\
   &\hspace{6em}\times\Whi^{\cD_t,\Xi^{-c}_\infty}_{\bbc}(f_\infty)\left(\begin{bmatrix}\fb&\fa\\-\fd&-\fc\end{bmatrix} h_1\right) \cdot \Xi^{-1}_\infty(\det h_1 )\,dh_1.
\end{align*}
We know that $h_1\mapsto \dsec^t_\infty(s,\chi,\Xi)\left(\cS^{-1}\imath(\bid_4,h_1)\right)$ is a multiple of the matrix coefficient $\left<h_1 w_t,w^\vee_{-t}\right>$ with $w_t$ a vector inside the lowest $K_\infty$-type of the discrete series of $\U(1,1)(\bR)$ for which $e^{ix}\begin{bsm}\cos\theta&\sin\theta\\-\cos\theta&\sin\theta\end{bsm}$ acts on the lowest $K_\infty$-type by $e^{itx}e^{it\theta}$. Thus,
\begin{equation}\label{eq:Fl-infty2}
\begin{aligned}
    F^t_\infty\left(\begin{bmatrix}\fx&\ast&\ast&\ast\\&\fa&\ast&\fb\\&&\nu\bar{\fx}^{-1}\\&\fc&\ast&\fd\end{bmatrix}\right)
   = &\,\Xi_\infty(\fx)\,\chi_\infty(\nu^{-1}\fx\bar{\fx})|\nu^{-1}\fx\bar{\fx}|^{s+\frac{3}{2}}_\infty
   \cdot C_t(s,\chi_\infty,\Xi_\infty) \cdot \Whi^{\cD_t,\Xi^{-c}_\infty}_{\bbc}(f_\infty)\begin{pmatrix}\fb&\fa\\-\fd&-\fc\end{pmatrix},
\end{aligned}
\end{equation}
with $C_t(s,\chi,\Xi)$ computed by
\begin{align*}
   C_t(s,\chi_\infty,\Xi_\infty)&=\int_{\U(1,1)(\bR)} \dsec^t_\infty(s,\chi,\Xi)\left(\cS^{-1}\imath(\bid_4,h_1)\right) 
   \left<h_1 w_{-t},w^{\vee}_{-t}\right>\cdot\Xi^{-1}_\infty(\det h_1)\,dh_1\\
   &=\int_{\SL(2,\bR)} \dsec^t_\infty(s,\chi,\Xi)\left(\cS^{-1}\imath(\bid_4,h^\circ_1)\right) 
   \left<h^\circ_1 w_{-t},w^{\vee}_{-t}\right>\,dh_1,
\end{align*}
where $h^\circ_1\in \SL(2,\bR)$ such that $h_1h^{\circ-1}_1\in \U(1)(\bR)$, and $\left<h_1 w_{-t},w^{\vee}_{-t}\right>$, $\left<h^\circ_1 w_{-t},w^{\vee}_{-t}\right>$ denote matrix coefficients for a vector $w_{-t}$ inside the $K_\infty$-type of weight $-t$ in $(\cD_t)^\vee$. Combining \eqref{eq:Fl-infty2} with
\begin{align*}
   \begin{bmatrix}1&0\\\alphaS&1\end{bmatrix} A_0\begin{bmatrix}a\\&a^{-1}\end{bmatrix}=
   \begin{bmatrix}\frac{2a}{\rule{0pt}{.8em}\sqrt{-(\alphaS-\alphabS)^2}}&\\ &a^{-1}\end{bmatrix}
   \begin{bmatrix}\frac{1}{\sqrt{1+a^4}}&\frac{-\sqrt{-1}a^2}{\sqrt{1+a^4}}\\[.5em]0&\begin{smallmatrix}\sqrt{1+a^4}\end{smallmatrix}\end{bmatrix}
   \underbrace{\begin{bmatrix}\frac{1}{\sqrt{1+a^4}}&\frac{\sqrt{-1}a^2}{\sqrt{1+a^4}}\\[.5em]\frac{\sqrt{-1}a^2}{\sqrt{1+a^4}}&\frac{1}{\sqrt{1+a^4}}\end{bmatrix}}_{\in\U(2,\bR)},
\end{align*}
we see that
\begin{align*}
   F^t_\infty\left(\begin{bsm}\lambda A_0\begin{bsm}a\\&a^{-1}\end{bsm}\\ &A^{-1}_0\begin{bsm}a^{-1}\\&a\end{bsm}\end{bsm};s,\chi,\Xi\right)&=\chi_\infty(\lambda)\left\vert\frac{4\lambda}{(a^2+a^{-2})(\alphaS-\alphabS)^2}\right\vert^{s+\frac{3}{2}}\\
   &\hspace{-2em}\times C_t(s,\chi_\infty,\Xi_\infty)\cdot \Whi^{\cD_t}_{\bbc}(f_\infty)\begin{pmatrix}&\lambda\sqrt{a^2+a^{-2}}\\-\sqrt{a^2+a^{-2}}^{-1} \end{pmatrix}.
\end{align*}
For $\Whi^{\cD_t}_{\bbc}(f_\infty)$, we have the formula
\begin{equation}\label{eq:Wfinfty}
\begin{aligned}
  &\Whi^{\cD_t}_{\bbc}(f_\infty)\left(h=\begin{bmatrix}a&b\\c&d\end{bmatrix}\right)\\
   = &\,\whi^{\cD_t}_{\bbc}(f_\infty)
   \left\{\begin{array}{ll} (\det h)^{t/2} (-ci+d)^{\,-t}\psi_\infty\left(\bbc (ai-b)(-ci+d)^{-1}\right), &\det h>0\\
   0,&\det h<0.
   \end{array}\right.
\end{aligned}
\end{equation}
It follows that
\begin{align*}
   F^t_\infty\left(\begin{bsm}\lambda A_0\begin{bsm}a\\&a^{-1}\end{bsm}\\ &A^{-1}_0\begin{bsm}a^{-1}\\&a\end{bsm}\end{bsm};s,\chi,\Xi\right)
   &= C_t(s,\chi_\infty,\Xi_\infty)\cdot \whi^{\cD_t}_{\bbc}(f_\infty)\\
   &\hspace{-7em}\times\mathds{1}_{\bR_{>0}}(\lambda)\left\vert\frac{4\lambda}{(a^2+a^{-2})(\alphaS-\alphabS)^2}\right\vert^{s+\frac{3}{2}}
   \lambda^{t/2}i^{-t}(a^2+a^{-2})^{t/2}\,\psi_\infty\left(i\bbc\lambda(a^2+a^{-2})\right)
\end{align*}
Plugging this and \eqref{eq:B-infty} into \eqref{eq:Z-lambdaa} gives us
\begin{equation}\label{eq:CZinfty}
\begin{aligned}
   &Z_\infty\left(\dsec^t_{\infty}(s,\chi,\Xi),\Bes^{\cD_{t,t}}_{\bS,\triv}(\varphi_\infty),\Whi^{\cD_t,\Xi^{-c}_\infty}_{\bbc}(f_\infty)\right)\\
   =&\,C_t(s,\chi_\infty,\Xi_\infty)\cdot \bes^{\cD_{t,t}}_{\bS,\triv}(\varphi_\infty)\, \whi^{\cD_t}_{\bbc}(f_\infty)
   \cdot i^{-t}2^{2s+t}|\alphaS-\alphabS|^{-2s-t} \\
   &\times \pi\int_0^\infty\int^\infty_1
   \lambda^{s+\frac{3t}{2}-\frac{5}{2}} (a^2+a^{-2})^{-s+\frac{t}{2}-\frac{3}{2}} \psi_\infty\left(i\bbc\lambda(a^2+a^{-2})\right) (a-a^{-3}) \,da\,d\lambda\\
   =&\,C_t(s,\chi_\infty,\Xi_\infty)\cdot \bes^{\cD_{t,t}}_{\bS,\triv}(\varphi_\infty)\, \whi^{\cD_t}_{\bbc}(f_\infty)
   \cdot i^{-t}2^{2s+t}\pi(2\pi\bbc)^{-s-\frac{3t}{2}+\frac{3}{2}}|\alphaS-\alphabS|^{-2s-t}\frac{\Gamma\left(s+\frac{3t-3}{2}\right)}{2s+t-1}.
\end{aligned}
\end{equation}
The factor $C_t(s,\chi_\infty,\Xi_\infty)$ can be computed easily as follows. For $h^\circ_1=\begin{bmatrix}a&b\\c&d\end{bmatrix}\in \SL(2,\bR)$,
\begin{align*}
    &\dsec^t_\infty(s,\chi,\Xi)\left(\cS^{-1}\imath(\bid_4,h^\circ_1)\right) 
   \left<h^\circ_1 w_{-t},w^{\vee}_{-t}\right>\,dh_1\\
   =&\,(a+d+(c-b)i)^{-t}|a+d+(c-b)i|^{-2s-3+t}\cdot (a+d-(c-b)i)^{-t}\\
   =&\,|a+d+(c-b)i|^{-2s-3-t},
\end{align*}
which are invariant under the left and right translations by $\SO(2,\bR)$. Applying the Cartan decomposition for $\SL(2,\bR)$, we have
\begin{equation}\label{eq:C-dou}
\begin{aligned}
   C_t(s,\chi_\infty,\Xi_\infty)
   &=\pi\int_1^\infty  (x+x^{-1})^{-2s-3-t} (x^2-x^{-2})x^{-1}dx\\
   &=2^{-2s-t-1}\pi(2s+1+t)^{-1}.
\end{aligned}
\end{equation}
Combining \eqref{eq:CZinfty} and \eqref{eq:C-dou} proves the proposition.
\end{proof}

\subsection{Big-cell sections for Siegel Eisenstein series}\label{sec:bigcell}

For our purpose of constructing $p$-adic $L$-functions, at the place $p$ and finite places where the conditions in (ii) of Theorem~\ref{thm:Furusawa} are not satisfied, we use the big-cell sections as the test sections for Siegel Eisenstein series.

Given a Schwartz function $\Schw_v$ on $\Her_3(\cK_v)$, we define the big-cell section $\dsec^{\bc}_{v,\Schw_v}(s,\chi,\Xi)\in I_v(s,\chi,\Xi)$ as
\begin{equation}\label{eq:bigcell-sec}
\begin{aligned}
   &\dsec^{\bc}_{v,\Schw_v}(s,\chi,\Xi)\left(g=\begin{bmatrix}\fA&\fB\\ \fC&\fD\end{bmatrix}\right)\\
   =&\, \Xi\left(\det(\nu_g\ltrans{\bar{\fC}}^{-1}\right)\,
   \chi\left(\det(\nu_g \ltrans{\bar{\fC}}^{-1}\fC^{-1})\right)\,
   \left|\det(\nu_g \ltrans{\bar{\fC}}^{-1}\fC^{-1})\right|^{s+\frac{3}{2}}_v
   \cdot \Schw_v(\fC^{-1}\fD).
\end{aligned}
\end{equation} 
From its definition, it is easy to see that the section  $\dsec^{\bc}_{v,\Schw_v}(s,\chi,\Xi)$ is supported on the big Bruhat cell $Q_{\GU(3,3)}(\bQ_v)\begin{bmatrix}&-\bid_3\\ \bid_3\end{bmatrix} Q_{\GU(3,3)}(\bQ_v)$.

\vspace{.5em}

In the next three subsections, we compute local zeta integrals for big-cell sections associated to certain  Schwartz functions $\Schw_v$ of the form
\begin{equation}\label{eq:Schw-block}
\begin{aligned}
   \Schw_v\begin{pmatrix}w_{11}&\bar{\fw}_{21}&\bar{\fw}_{31}\\ \fw_{21}&w_{22}& \bar{\fw}_{32}\\ \fw_{31}&\fw_{32}&w_{33}\end{pmatrix}
   =&\, \Schw_{v,1,\sym}\begin{pmatrix}w_{11}&\frac{\fw_{21}+\bar{\fw}_{21}}{2}\\ \frac{\fw_{21}+\bar{\fw}_{21}}{2}&w_{22}\end{pmatrix}
   \,   \Schw_{v,2}(w_{33})\\
   &\times\Schw_{v,1,\alt}\left(\frac{\fw_{21}-\bar{\fw}_{21}}{\alphaS-\alphabS}\right)
   \,\Schw_{v,0}\left(\begin{bmatrix}\alphaS&1\\ \alphabS&1\end{bmatrix}^{-1}\begin{bmatrix}\fw_{31}&\fw_{32}\\ \bar{\fw}_{31}&\bar{\fw}_{32}\end{bmatrix}\right),
\end{aligned}
\end{equation}
with $\Schw_{v,1,\sym}$ a Schwartz function on $\Sym_2(\bQ_v)$, $\Schw_{v,1,\alt}$, $\Schw_{v,2}$ Schwartz functions on $\bQ^\times_v$, and $\Schw_{v,0}$  a Schwartz function on $M_{2,2}(\bQ_p)$.

\subsection{Local zeta integrals for big-cell sections I}
Given an integer $m\geq 1$, we define the following level groups:
\begin{equation}\label{eq:Ks}
\begin{aligned}
   K_{\GSp(4),v}(\varpi^m_v)&=\left\{g\in\GSp(4,\bZ_v):g\equiv \bid_4\mod \varpi^m_v\right\},\\
    K^1_{\GSp(4),v}(\varpi^m_v)&=\left\{g\in\GSp(4,\bZ_v):g\equiv \begin{bsm}1&\ast&\ast&\ast\\ &1&\ast&\ast\\&&1\\&&\ast&1\end{bsm}\mod \varpi^m_v\right\},\\
   K'_{\GSp(4),v}(\varpi^m_v)&=\left\{g\in\GSp(4,\bZ_v):g\equiv \begin{bsm}\bid_2&\ast\\ &\bid_2\end{bsm}\mod \varpi^m_v\right\},
\end{aligned}
\end{equation}
and for $G=\GL(2)$ or $\GU(1,1)$,
\begin{equation}\label{eq:Ks2}
\begin{aligned}
   K_{G,v}(\varpi^{m}_v)&=\left\{g\in G(\bZ_v):g\equiv \bid_2\mod \varpi^m_v\right\},\\
   K^1_{G,v}(\varpi^{m}_v)&=\left\{g\in G(\bZ_v):g\equiv \begin{bsm}1&\ast\\ &1\end{bsm}\mod \varpi^m_v\right\}.
\end{aligned}
\end{equation}
We define a Schwartz function $\Phi_v$ with \begin{equation}\label{eq:vol-sec}
\begin{aligned}
   \Schw_{v,1,\sym}&=\mathds{1}_{\Sym_2(\bZ_v)},
   &\Schw_{v,1,\alt}&=\mathds{1}_{4\bbc^2\varpi^{2m}_v\bZ_v},\\
   \Schw_{v,0}&=\mathds{1}_{K_{\GL(2),v}(\varpi^{m}_v)},
   &\Schw_{v,2}&=\cF^{-1}\mathds{1}_{-\bbc(1+\varpi^{m}_v\bZ_v)}.
\end{aligned}
\end{equation}
The big-cell sections associated to this type of $\Phi_v$ will be used at finite places other than $p$ where the input for the local zeta integrals has ramification.

\begin{prop}\label{prop:Z-vol}
Suppose that $\bS\in\Sym^*_2(\bZ_v)$, and take an integer $m\geq 1$ such that $\Bes^{\Pi_v}_{\bS,\Lambda_v}(\varphi_v)$ (resp. $\Whi^{\pi_v}_\bbc(f_v)$) is fixed by $K'_{\GSp(4),v}(\varpi^{m})$ (resp. $K^1_{\GL(2),v}(\varpi^{m})$) and $\chi_v$, $\Xi_{\bQ,v}$ are trivial on $1+\varpi^m_v\bZ_v$. Let $\Schw_v$ be the Schwartz function on $\Her_3(\cK_v)$ given in \eqref{eq:vol-sec}. Then
\begin{align*}
   &Z_v\Big(\dsec^{\bc}_{v,\Schw_v}(s,\chi,\Xi),\Bes^{\Pi_v}_{\bS,\Lambda_v}(\varphi_v), \Whi^{\pi_v,\Upsilon_v}_\bbc(f_v)\Big)\\
   =&\,\frac{|4\bbc^2 \varpi^{7m}_v|_v}{ (1-|\varpi_v|^4_v)(1-|\varpi_v|^2_v)^2}\cdot \bes^{\Pi_v}_{\bS,\Lambda_v}\left(\begin{bmatrix}&-\bid_2\\ \bid_2\end{bmatrix}\cdot \varphi_v\right)
   \,\whi^{\pi_v}_\bbc(f_v).
\end{align*} 
\end{prop}
\begin{proof}
Write 
\begin{align*}
   g&=\begin{bmatrix}A&B\\C&D\end{bmatrix},
   &h&=\begin{bmatrix}\fa&\fb\\ \fc&\fd\end{bmatrix},
\end{align*}
and let 
\[
   \begin{bmatrix}\fA&\fB\\ \fC&\fD\end{bmatrix}
   =\cS^{-1}\imath(\eta_\bS\, g,h).
\] 
Then
\begin{align*}
   \fC&=\begin{bmatrix}
   1&-\alphabS\\&1\\&&1
   \end{bmatrix}
   \begin{bmatrix}
   C&-\begin{bmatrix}\alphabS\\1\end{bmatrix}\fa\\-\begin{bmatrix}\alphaS&1\end{bmatrix}A&\fc
   \end{bmatrix},
   &\fD&=\begin{bmatrix}
   1&-\alphabS\\&1\\&&1
   \end{bmatrix}
   \begin{bmatrix}
   D&-\begin{bmatrix}\alphabS\\1\end{bmatrix}\fb\\-\begin{bmatrix}\alphaS&1\end{bmatrix}B&\fd
   \end{bmatrix}.
\end{align*}
Put
\begin{align*}
   \fc'&=\fc-\fa\cdot\bbc^{-1}\Tr\bS AC^{-1},
   &\fd'&=\fd-\fb\cdot\bbc^{-1}\Tr\bS AC^{-1}.
\end{align*}
Then 
\[
    \det\fC=\fc' \det C,
\]
and
\begin{align*}
   \fC^{-1}\fD=\begin{bmatrix}
   C^{-1}D+\fc^{\prime-1}\fa\nu_g C^{-1}\begin{bmatrix}\alphaS\alphabS&\alphabS\\ \alphaS&1\end{bmatrix}\ltrans{C}^{-1} & C^{-1}\begin{bmatrix}\alphabS\\1\end{bmatrix} \bar{\fc}^{\prime-1}\nu_h\\[1em]
   \fc^{\prime-1}\nu_g\begin{bmatrix}\alphaS&1\end{bmatrix}\ltrans{C}^{-1}&\fc^{\prime-1}\fd'
   \end{bmatrix}.
\end{align*}
By the definition of the big-cell sections in \eqref{eq:bigcell-sec},
\begin{align*}
   &\dsec^{\bc}_{v,\Schw_v}(s,\chi,\Xi)\left(\cS^{-1}\imath(\eta^\bS g,h)\right)\cdot\Xi_v^{-1}(\det h)\\
   =&\,\Xi_v\left(\nu^3\bar{\fc}^{\prime-1}(\det C)^{-1}(\det h)^{-1}\right)\,\chi_v\left(\nu^3 \fc^{\prime-1}\bar{\fc}^{\prime-1}(\det C)^{-2}\right)\, \left|\nu^3 \fc^{\prime-1}\bar{\fc}^{\prime-1}(\det C)^{-2}\right|^{s+\frac{3}{2}}_v\\
   &\times \Schw_v\begin{pmatrix}
   C^{-1}D+\nu\fc^{\prime-1}\fa C^{-1}\begin{bmatrix}\alphaS\alphabS&\alphabS\\ \alphaS&1\end{bmatrix}\ltrans{C}^{-1} & C^{-1}\begin{bmatrix}\alphabS\\1\end{bmatrix} \nu\bar{\fc}^{\prime-1}\\[1em]
   \nu\fc^{\prime-1}\begin{bmatrix}\alphaS&1\end{bmatrix}\ltrans{C}^{-1}&\fc^{\prime-1}\fd'.
   \end{pmatrix}
\end{align*}
Let
\begin{align*}
   T&=\imath_\bS(\fc')^{-1} \nu\ltrans{C}^{-1},
   &t&=\nu\fc^{\prime-1}\bar{\fc}^{\prime-1},
\end{align*}
and
\begin{align*}
   W&=C^{-1}D,
   &w&=\fc^{\prime-1}\fd',
   &u=\nu\bar{\fc}^{\prime}\fa.
\end{align*}
Then
\begin{align*}
   \nu\fc^{\prime-1}\fa C^{-1}\begin{bmatrix}\alphaS\alphabS&\alphabS\\ \alphaS&1\end{bmatrix}\ltrans{C}^{-1}
   &=u\cdot \ltrans{T}\begin{bmatrix}\alphaS\alphabS&\alphabS\\ \alphaS&1\end{bmatrix}T,\\
   \nu\begin{bmatrix}\fc^{\prime-1}\\&\bar{\fc}^{\prime-1}\end{bmatrix}\begin{bmatrix}\alphaS&1\\\alphabS&1\end{bmatrix} \ltrans{C}^{-1}
   &=\begin{bmatrix}\alphaS&1\\\alphabS&1\end{bmatrix} T,
\end{align*}
and
\begin{align*}
   &\dsec^{\bc}_{v,\Schw_v}(s,\chi,\Xi)\left(\cS^{-1}\imath(\eta^\bS g,h)\right)\cdot\Xi^{-1}(\det h)\\
   =&\,\Xi_v\left( \bar{\fc}'\det T\right)\,\chi_v\left(t(\det T)^2\right)\, \left|t(\det T)^2\right|^{s+\frac{3}{2}}_v\\
   &\times\Schw_{v,1,\sym}\left(W+u\cdot \ltrans{T}\begin{bmatrix}\alphaS\alphabS&\frac{\alphaS+\alphabS}{2}\\ \frac{\alphaS+\alphabS}{2}&1\end{bmatrix}T\right)
   \cdot\Schw_{v,1,\alt}\left(u\det T\right)
   \cdot \Schw_{v,0}\left( T\right)
   \cdot \Schw_{v,2}(w).
\end{align*}
Also, we have
\begin{align*}
  g&=\begin{bmatrix}\bid_2&AC^{-1}\\&\bid_2\end{bmatrix}
  \begin{bmatrix}  &-\nu\ltrans{C}^{-1}\\C\end{bmatrix}
  \begin{bmatrix}\bid_2&C^{-1}D\\ &\bid_2\end{bmatrix}\\
  &=\begin{bmatrix}\bid_2&AC^{-1}\\&\bid_2\end{bmatrix} 
   \begin{bmatrix}\imath_\bS\left(\fc'\right)\\&\ltrans{\,\ol{\imath_\bS\left(\bar{\fc}'\right)}}\end{bmatrix}
   \begin{bmatrix}&-T\\ t\, \ltrans{T}^{-1}\end{bmatrix}
   \begin{bmatrix}\bid_2&W\\&\bid_2\end{bmatrix},
\end{align*}
\begin{align*}
  \begin{bmatrix}&1\\-1\end{bmatrix} h
  &=\begin{bmatrix}\fc&\fd\\-\fa&-\fb\end{bmatrix}
  =\begin{bmatrix}1&-\bbc^{-1}\Tr\bS AC^{-1}\\&1\end{bmatrix}\begin{bmatrix}\fc'&\fd'\\ -\fa&-\fb\end{bmatrix}\\
  &=\begin{bmatrix}1&-\bbc^{-1}\Tr\bS AC^{-1}\\&1\end{bmatrix}
  \begin{bmatrix}1&\nu^{-1}\bar{\fc}^{\prime}\fd'\\&1\end{bmatrix}
  \begin{bmatrix}\fc'&\\ &\nu\bar{\fc}^{\prime-1}\end{bmatrix}\begin{bmatrix}1+\nu^{-1}\fa\bar{\fc}'\cdot\fc^{\prime-1}\fd'&(\fc^{\prime-1}\fd')^2\cdot\nu^{-1}\fa\bar{\fc}'\\ -\nu^{-1}\bar{\fc}'\fa&1-\nu^{-1}\fa\bar{\fc}'\cdot\fc^{\prime-1}\fd'\end{bmatrix}\\
  &=\begin{bmatrix}1&-\bbc^{-1}\Tr\bS AC^{-1}\\&1\end{bmatrix}
     \fc'\begin{bmatrix}1\\&t\end{bmatrix}
     \begin{bmatrix}1+uw&w\\ -u&1-uw\end{bmatrix}\\
     &=\begin{bmatrix}1&-\bbc^{-1}\Tr\bS AC^{-1}+t^{-1}w\\&1\end{bmatrix}
     \fc'\begin{bmatrix}1\\&t\end{bmatrix}
     \begin{bmatrix}1+uw&uw^2\\ -u&1-uw\end{bmatrix}.
\end{align*}
Using $W,w,u,T,t$ as the coordinates, we have
\begin{align*}
    \int_{\big(R'_\bS\backslash\GSp(4)\times_{\GL(1)}\GU(1,1)\big)(\bQ_v)} \cdots \,dh\,dg 
   = &\,\frac{1-|\varpi_v|_v}{(1-|\varpi_v|_v^4)(1+|\varpi_v|_v)} \\
   &\hspace{-4em}\times\int_{\GL_2(\bQ_v)\times\bQ^\times_v}|\det T|^{-3}_v |t|^3_v \int_{\Sym_2(\bQ_v)\times\bQ_v\times\bQ_v} \cdots\,dWdw\,du\,\,d^\times T\,d^\times t,
\end{align*}
and
\begin{align*}
   \Bes^{\Pi_v}_{\bS,\Lambda_v}(\varphi_v)(g)\cdot \Whi^{\pi_v,\Upsilon_v}_{\bbc}(f_v)\left( \begin{bmatrix}&1\\-1\end{bmatrix} h\right)
   =&\,\Lambda_v\Upsilon_v(\fc')\cdot\psi_v(\bbc\, t^{-1}w)\\
   &\times \Bes^{\Pi_v}_{\bS,\Lambda_v}(\varphi_v)\left(\begin{bmatrix}&-T\\ t\, \ltrans{T}^{-1}\end{bmatrix}\begin{bmatrix}\bid_2&W\\&\bid_2\end{bmatrix}\right)\\
  &\times \Whi^{\pi_v}_{\bbc}(f_v)\left(\begin{bmatrix}1\\&t\end{bmatrix}
     \begin{bmatrix}1+uw&uw^2\\ -u&1-uw\end{bmatrix}\right).
\end{align*}
It follows that
\begin{align*}
    &Z_v\Big(\dsec^{\bc}_{v,\Schw_v}(s,\chi,\Xi),\Bes^{\Pi_v}_{\bS,\Lambda_v}(\varphi_v), \Whi^{\pi_v,\Upsilon_v}_\bbc(f_v)\Big)\\
    =&\,\frac{1-|\varpi_v|_v}{(1-|\varpi_v|_v^4)(1+|\varpi_v|_v)} \int_{\GL(2,\bQ_v)\times\bQ^\times_v}
    |\det T|^{-3}_v |t|^3_v \cdot \Xi_v\left(\det T\right)\,\chi_v\left(t(\det T)^2\right)\, \left|t(\det T)^2\right|^{s+\frac{3}{2}}_v\\
   &\times\int_{\Sym_2(\bQ_v)\times\bQ_v\times\bQ_v}\Schw_{v,1,\sym}\left(W+u\cdot \ltrans{T}\begin{bmatrix}\alphaS\alphabS&\frac{\alphaS+\alphabS}{2}\\ \frac{\alphaS+\alphabS}{2}&1\end{bmatrix}T\right)
   \cdot\Schw_{v,1,\alt}\left(u\det T\right)
   \cdot \Schw_{v,0}\left(T\right)
   \\
   &\times  \Schw_{v,2}(w)\, \psi_v(\bbc\, t^{-1}w)
   \cdot \Bes^{\Pi_v}_{\bS,\Lambda_v}(\varphi_v)\left(\begin{bmatrix}&-T\\ t\, \ltrans{T}^{-1}\end{bmatrix}\begin{bmatrix}\bid_2&W\\&\bid_2\end{bmatrix}\right)\\
  &\times\Whi^{\pi_v}_{\bbc}(f_v)\left(\begin{bmatrix}1\\&t\end{bmatrix}\begin{bmatrix}1+uw&uw^2\\ -u&1-uw\end{bmatrix}\right)
     \,dW dw\,du\,\,d^\times T\,d^\times t.
\end{align*}
The conditions on $\varphi_v,f_v,\Xi_{\bQ,v},\chi_v$ and choice of $\Schw_v$ in \eqref{eq:vol-sec} implies that
\begin{align*}
   &Z_v\Big(\dsec^{\bc}_{v,\Schw_v}(s,\chi,\Xi),\Bes^{\Pi_v}_{\bS,\Lambda_v}(\varphi_v), \Whi^{\pi_v,\Upsilon_v}_\bbc(f_v)\Big)\\
   =&\,  \frac{1-|\varpi_v|_v}{(1-|\varpi_v|_v^4)(1+|\varpi_v|_v)}\cdot |4\bbc^2\varpi^{2m}_v|_v \cdot \vol\left(K_{\GL(2),v}(\varpi^m_v)\right)\\
   &\times \int_{\bQ^\times} \int_{\bQ_p}\Schw_{v,2}(w)\, \psi_v(\bbc\, t^{-1}w)
   \cdot\Bes^{\Pi_v}_{\bS,\Lambda_v}(\varphi_v)\begin{pmatrix}&-\bid_2\\ t\cdot \bid_2\end{pmatrix}
   \,\Whi^{\pi_v}_\bbc(f_v)\begin{pmatrix}1\\&t\end{pmatrix}
   \,dw\,d^\times t\\
   =&\,\frac{|4\bbc^2 \varpi^{6m}_v|_v}{(1-|\varpi_v|^4_v)(1-|\varpi_v|^2_v)(1+|\varpi_v|_v)}\\
   &\times\int_{\bQ^\times_p} \cF(\cF^{-1}\mathds{1}_{-\bbc(1+\varpi^{m}_v\bZ_v)})(\bbc t^{-1})\cdot\Bes^{\Pi_v}_{\bS,\Lambda_v}(\varphi_v)\begin{pmatrix}&-\bid_2\\ t\cdot \bid_2\end{pmatrix}
   \,\Whi^{\pi_v}_\bbc(f_v)\begin{pmatrix}1\\&t\end{pmatrix}
   \,dt\\
   =&\, \frac{|4\bbc^2 \varpi^{6m}_v|_v}{ (1-|\varpi_v|^4_v)(1-|\varpi_v|^2_v)(1+|\varpi_v|_v)}\cdot \frac{|\varpi^m_v|_v}{1-|\varpi_v|_v}\,\bes^{\Pi_v}_{\bS,\Lambda_v}\left(\begin{bmatrix}&-\bid_2\\ \bid_2\end{bmatrix}\cdot \varphi_v\right)
   \,\whi^{\pi_v}_\bbc(f_v)
\end{align*}

\end{proof}

\subsection{The $\bU_v$-operators}\label{sec:Uv}

Given integers $m_1\geq m_2\geq 0$ (resp. $m_3\geq 0$), we define the $\bU_v$-operator $U^{\GSp(4)}_{v,m_1,m_2}$ (resp. $U^{\GL(2)}_{v,m_3}$) on as
\begin{equation}\label{eq:loc-Uv}
\begin{aligned}
   U^{\GSp(4)}_{v,m_1,m_2}&=\int_{U_{\GSp(4)}(\bZ_v)} \text{action by } u\begin{bmatrix}\varpi^{m_1}_v\\&\varpi^{m_2}_v\\&&\varpi^{-m_1}_v\\&&&\varpi^{-m_2}_v\end{bmatrix}\,du,\\
   U^{\GL(2)}_{v,m_3}&=\int_{U_{\GL(2)(\bZ_v)}} \text{action by } u \begin{bmatrix}\varpi^{m_3}_v\\&\varpi^{-m_3}_v\end{bmatrix} \,du.
\end{aligned}
\end{equation}
It is easy to see that 
\begin{align*}
  U^{\GSp(4)}_{v,m_1,m_2}&=\left(U^{\GSp(4)}_{v,1,0}\right)^{m_1-m_2}\left(U^{\GSp(4)}_{v,1,1}\right)^{m_2},
  &U^{\GL(2)}_{v,m_3}&=\left(U^{\GL(2)}_{v,1}\right)^{m_3}.
\end{align*}


\vspace{.5em}
In the following, given $\varphi_v\in\Pi_v$, $f_v\in\pi_v$, we consider $\varphi_{v,m_1,m_2}$, $f_{v,m_3}$ defined as
\begin{equation}\label{eq:dual-twist}
\begin{aligned}
    \varphi_{v,m_1,m_2}&=\Pi_v\left(\begin{bmatrix}&\bid_2\\ \bid_2\end{bmatrix}\begin{bmatrix}\varpi^{m_1}_v\\&\varpi^{m_2}_v\\&&\varpi^{-m_1}_v\\&&&\varpi^{-m_2}_v\end{bmatrix}\right)\varphi_v,\\
   f_{v,m_3}&=\pi_v\left(\begin{bmatrix}&1\\1\end{bmatrix}\begin{bmatrix}\varpi^{m_3}_v\\&\varpi^{-m_3}_v\end{bmatrix}\right)f_v.
\end{aligned}
\end{equation}
A simple computation gives the following proposition.
\begin{prop}\label{prop:dual-twist}
Suppose that $\varphi_v\in\Pi_v$ (resp. $f_v\in\pi_v$) is an eigenvector for $U^{\GSp(4)}_{v,1,1}$ with eigenvalue $\lambda_v(\varphi_v)$ (resp. eigenvector for  $U^{\GL(2)}_{v,1}$ with eigenvalue $\lambda_v(f_v)$). If the test section $\dsec_v(s,\chi,\Xi)\in I_v(s,\chi,\Xi)$ is invariant under the right translation by $\imath\left(K_{\GSp(4),v}(\varpi^m_v)\times K_{\GL(2),v}(\varpi^m_v)\right)$, then for all $m_1\geq m_2\geq \frac{m}{2}$, $m_3\geq \frac{m}{2}$ and $n\geq 0$,
\begin{align*}
   &Z_v\left(\dsec_v(s,\chi,\Xi),\Bes^{\Pi_v}_{\bS,\Lambda_{v}}(\varphi_{v,,m_1+n,m_2+n}),\Whi^{\pi_v,\Upsilon_v}_\bbc(f_{v,m_3+n})\right)\\
   =&\,\big(\lambda_v(\varphi_v)\lambda_v(f_v)\big)^n
   \cdot Z_v\left(\dsec_v(s,\chi,\Xi),\Bes^{\Pi_v}_{\bS,\Lambda_{v}}(\varphi_{v,m_1,m_2}),\Whi^{\pi_v,\Upsilon_v}_\bbc(f_{v,m_3})\right).
\end{align*}
\end{prop}

\subsection{Local zeta integrals for big-cell sections II}

In this and next subsections, we compute the local zeta integrals in the special case where the test vectors in $\Pi_v$ and $\pi_v$ are of the form in \eqref{eq:dual-twist} with $\varphi_v$ and $f_v$ eigenvectors for $\bU_v$-operators with nonzero eigenvalues. This case will be used for the place $p$ in the construction of $p$-adic $L$-functions.

\begin{prop}\label{Prop:Zbc-1}
With $\varphi_v$, $f_v$, $\lambda_v(\varphi_v),\lambda_v(f_v)$ as in Proposition~\ref{prop:dual-twist}, and the Schwartz function $\Schw_v$ of the form in \eqref{eq:Schw-block}, we have
\begin{equation}\label{eq:Zp-first reduction}
\begin{aligned}
   &Z_v\left(\dsec^\bc_{v,\Schw_v}(s,\chi,\Xi),\Bes^{\Pi_v}_{\bS,\Lambda_{v}}(\varphi_{v,,m_1,m_2}),\Whi^{\pi_v,\Upsilon_v}_\bbc(f_{v,m_3})\right)\\
   =&\,C\cdot \lim_{n\to\infty} \big(\lambda_v(\varphi_v)\lambda_v(f_v)\big)^{-n}\int_{\bQ^\times_v}\int_{\GL(2,\bQ_v)}\int_{\bQ_v}
   |\det T|^{s+\frac{1}{2}}_v |\nu|^{s-\frac{1}{2}}_v |r|^{s-\frac{5}{2}}_v
   \chi_v\left(\nu r\det T\right)
   \Xi_v\left(\det T\right)\\
   &\times\Schw_{v,1,\alt}\left(\nu\right)
   \, \Schw_{v,0}\left( T\right)\cdot\,\psi_v(-\bbc r) \cdot \Bes^{\Pi_v}_{\bS,\Lambda_v}(\varphi_{v,m_1+n,m_2+n})\begin{pmatrix}&-T\\ r^{-1}\nu^{-1}\,\ltrans{T}^{\mr{adj}} \end{pmatrix} \\
   &\times \Whi^{\pi_v}_\bbc(f_{v,m_3+n})\begin{pmatrix}&\nu^{-1}\det T\\-r^{-1}\end{pmatrix}
   \,dr\,d^\times T\,d^\times \nu,
\end{aligned}
\end{equation}
where
\begin{align*}
  C&= \frac{1-|\varpi_v|_v}{(1-|\varpi_v|_v^4)(1+|\varpi_v|_v)}\int_{\Sym_2(\bQ_v)} \Schw_{v,1,\sym}\left(W\right) \,dW \int_{\bQ_v} \Schw_{v,2}\left(w\right) \,dw,
\end{align*}
and the superscript $^{\mr{adj}}$ stands for the adjugate matrix, {\it i.e.} for $X=\begin{bmatrix}x_{11}&x_{12}\\x_{21}&x_{22}\end{bmatrix}$, $X^{\mr{adj}}= \begin{bmatrix}x_{22}&-x_{12}\\-x_{21}&x_{11}\end{bmatrix}$.
\end{prop}
\begin{proof}
It suffices to show \eqref{eq:Zp-first reduction} for $\mr{Re}(s)\gg 0$. We fix a sufficiently large integer $e$ such that
\begin{enumerate}
\item $\Schw_v$ is supported on $\Her_3(\varpi^{-e}_v\cO_{\cK,v})$ and is invariant under the translation by $\Her_3(\varpi^e_v\cO_{\cK,v})$,
\item $\varphi_{v,m_1,m_2}$ is invariant under $K^1_{\GSp(4),v}(\varpi^e_v)$ and $f_{v,m_3}$ is invariant under $K^1_{\GL(2),v}(\varpi^e_v)$.
\end{enumerate}

Write
\begin{equation}\label{eq:gh}
\begin{aligned}
   g&=\begin{bmatrix}
   A&B\\C&D
   \end{bmatrix}\in\GSp(4,\bQ_v),
   &h&=\begin{bmatrix}
   \fa&\fb\\ \fc&\fd
   \end{bmatrix}\in\GU(1,1)(\bQ_v).
\end{aligned}
\end{equation}
Given $n\geq 0$, we partition $\left(\GSp(4)\times_{\GL(1)}\GU(1,1)\right)(\bQ_v)$ as
\[ 
   \left(\GSp(4)\times_{\GL(1)}\GU(1,1)\right)(\bQ_v)=\cU_{1,n}\sqcup\,\cU_{2,n}\sqcup\,\cU_{3,n}
\] 
with  $\cU_{j,n}$ defined as: Given $(g,h)\in  \left(\GSp(4)\times_{\GL(1)}\GU(1,1)\right)(\bQ_v)$,
\begin{align*}
   (g,h)\in \cU_{1,n}&\Longleftrightarrow C^{-1}D\in \Sym_2(\varpi_v^{-2n+e}\bZ_v),\,\fa^{-1}\fb\in \varpi_v^{-2n+e}\bZ_v\\
   (g,h)\in \cU_{2,n}&\Longleftrightarrow C^{-1}D\in \Sym_2(\varpi_v^{-2n+e}\bZ_v),\,\fa^{-1}\fb\notin \varpi^{-2n+e}\bZ_v\\
   (g,h)\in \cU_{3,n}&\Longleftrightarrow C^{-1}D\notin \Sym_2(\varpi_v^{-2n+e}\bZ_v).
\end{align*}
Let
\begin{equation}\label{eq:Ijm}
\begin{aligned}
   I_{j,n}&=\int_{R'_\bS(\bQ_v)\backslash \cU_{j,n}} \dsec^{\bc}_{v,\Schw_v}(s,\chi,\Xi) \left(\cS^{-1}\imath(\eta_\bS\, g,h)\right)\cdot \Bes^{\Pi_v}_{\bS,\Lambda_v}(\varphi_{m_1+n,m_2+n})(g)\\
   &\hspace{10em}\times \Whi^{\pi_v,\Upsilon_v}_{\bbc}(f_{v,m_3+n})\left(\begin{bmatrix}0&1\\-1&0\end{bmatrix}h\right)\Xi^{-1}_v(\det h)\,dhdg.
\end{aligned}
\end{equation}
Then
\[
   Z_v\left(\dsec_v(s,\chi,\Xi),\Bes^{\Pi_v}_{\bS,\Lambda_{v}}(\varphi_{v,m_1+n,m_2+n}),\Whi^{\pi_v,\Upsilon_v}_\bbc(f_{v,m_3+n})\right)
   =I_{1,n}+I_{2,n}+I_{3,n}.
\]
By Proposition~\ref{prop:dual-twist}, \eqref{eq:Zp-first reduction} will follow if we can show
\begin{align*}
   &\lim\limits_{n\to\infty}\big(\lambda_v(\varphi_v)\lambda_v(f_v)\big)^{-n}I_{1,n}=\text{RHS of }\eqref{eq:Zp-first reduction},\\
     &\lim\limits_{n\to\infty}\big(\lambda_v(\varphi_v)\lambda_v(f_v)\big)^{-n}I_{2,n}=\lim\limits_{n\to\infty}\big(\lambda_v(\varphi_v)\lambda_v(f_v)\big)^{-n}I_{3,n}=0.
\end{align*}

\vspace{.5em}

\noindent \underline{Bounding $I_{3,n}$.}\\

Without loss of generality, we assume that $\chi_v,\Xi_v$ and the central characters of $\Pi_v$ and $\pi_v$ are unitary, and $\Phi_v$ is the characteristic function of certain compact open subset of $\Her_3(\cK_v)$. 

With $g,h$ as in \eqref{eq:gh} and writing
\begin{align*}
   \cS^{-1}\imath(\eta_\bS \,g,h)=\begin{bmatrix}
   \fA&\fB\\ \fC&\fD
   \end{bmatrix},
\end{align*}
we have
\begin{align*}
   \fC&=\begin{bmatrix}
   1&-\alphabS\\&1\\&&1
   \end{bmatrix}
   \begin{bmatrix}
   C&-\begin{bmatrix}\alphabS\\1\end{bmatrix}\fa\\-\begin{bmatrix}\alphaS&1\end{bmatrix}A&\fc
   \end{bmatrix},
   &\fD&=\begin{bmatrix}
   1&-\alphabS\\&1\\&&1
   \end{bmatrix}
   \begin{bmatrix}
   D&-\begin{bmatrix}\alphabS\\1\end{bmatrix}\fb\\-\begin{bmatrix}\alphaS&1\end{bmatrix}B&\fd.
   \end{bmatrix}
\end{align*}
Let 
\begin{equation}\label{eq:r}
   r=\fc\fa^{-1}-\bbc^{-1}\Tr\bS AC^{-1}.
\end{equation}
Then
\begin{align*}
   \det\fC=\det C\cdot r\fa,
\end{align*}
and
\begin{align*}
   \fC^{-1}\fD&=\begin{bmatrix}
   C^{-1}D+r^{-1}C^{-1}\begin{bmatrix}\alphabS\\1\end{bmatrix}\begin{bmatrix}\alphaS&1\end{bmatrix}(AC^{-1}D-B)& -C^{-1}\begin{bmatrix}\alphabS\\1\end{bmatrix}r^{-1}(\fc\fa^{-1}\fb-\fd)\\
   \fa^{-1}r^{-1}\begin{bmatrix}\alphaS&1\end{bmatrix}(AC^{-1}B-D)&\fa^{-1}r^{-1}(-\begin{bmatrix}\alphaS&1\end{bmatrix}AC^{-1}\begin{bmatrix}\alphabS\\1\end{bmatrix}\fb+\fd)
   \end{bmatrix}\\
   &=\begin{bmatrix}
   C^{-1}D+r^{-1}\nu_g C^{-1}\begin{bmatrix}\alphaS\alphabS&\alphabS\\\alphaS&1\end{bmatrix}\ltrans{C}^{-1}&C^{-1}\begin{bmatrix}\alphabS\\1\end{bmatrix}r^{-1}\bar{\fa}^{-1}\nu_h\\[1em]
   \fa^{-1}r^{-1}\nu_g\begin{bmatrix}\alphaS&1\end{bmatrix}\ltrans{C}^{-1}&\fa^{-1}\fb+r^{-1}\fa^{-1}\bar{\fa}^{-1}\nu_h.
   \end{bmatrix}
\end{align*}
Plugging these into \eqref{eq:bigcell-sec}, writing
\begin{equation}\label{eq:ghII}
\begin{aligned}
     g&=\begin{bmatrix}A&B\\C&D\end{bmatrix}
   =\begin{bmatrix}\bid_2&AC^{-1}\\0&\bid_2\end{bmatrix}
    \begin{bmatrix}0&-\nu\ltrans{C}^{-1}\\C&0\end{bmatrix}
    \begin{bmatrix}\bid_2&C^{-1}D\\0&\bid_2\end{bmatrix},\\
    \begin{bmatrix}0&1\\-1&0\end{bmatrix} h&=\begin{bmatrix}1&-\fc\fa^{-1}\\0&1\end{bmatrix}\begin{bmatrix}0&\nu\bar{\fa}^{-1}\\-\fa&0\end{bmatrix}\begin{bmatrix}1&\fa^{-1}\fb\\0&1\end{bmatrix},
\end{aligned}
\end{equation} 
and putting 
\begin{equation}\label{eq:Ww}
\begin{aligned}
    W&=C^{-1}D,
    &w&=\fa^{-1}\fb,
\end{aligned}
\end{equation}
we get
{\small
\begin{equation}\label{eq:I3m-1}
\begin{aligned}
   I_{3,n}=&\,
     \frac{1-|\varpi_v|_v}{(1-|\varpi_v|_v^4)(1+|\varpi_v|_v)}
    \int_{T'_\bS(\bQ_v)\backslash (\GL(2,\bQ_v)\times\cK^\times_v)}\int_{\bQ^\times_v}\int_{\bQ_v}\int_{\bQ_v\times (\Sym_2(\bQ_v)-\Sym_2(\varpi^{-2n+e}_v\bZ_v))} \\
    &\times |\nu^3 r^{-2}\fa^{-1}\bar{\fa}^{-1}(\det C)^{-2}|^{s+\frac{3}{2}}_v\cdot \chi_v\left(\nu^3 r^{-2}\fa^{-1}\bar{\fa}^{-1}(\det C)^{-2}\right)\, \Xi_v\left(\nu^2 r^{-1}\fa^{-1}(\det C)^{-1}\right)\\
   &\times\Schw_v\begin{pmatrix}
    W+r^{-1}\nu C^{-1}\begin{bmatrix}\alphaS\alphabS&\alphabS\\ \alphaS&1\end{bmatrix}\ltrans{C}^{-1}&\nu C^{-1}\begin{bmatrix}\alphabS\\1\end{bmatrix}r^{-1}\bar{\fa}^{-1}\\[1em]
   \fa^{-1}r^{-1}\begin{bmatrix}\alphaS&1\end{bmatrix}\nu\ltrans{C}^{-1}&w+r^{-1}\fa^{-1}\bar{\fa}^{-1}\nu
   \end{pmatrix}\\
   &\times \psi_v(-\bbc r)\cdot \Bes^{\Pi_v}_{\bS,\Lambda_v}(\varphi_{v,m_1+n,m_2+n})\left(\begin{bmatrix}0&-\nu\ltrans{C}^{-1}\\C&0\end{bmatrix}
   \begin{bmatrix}\bid_2& W\\0&\bid_2\end{bmatrix}\right)\\
   &\times  \Whi^{\pi_v,\Upsilon_v}_{\bbc}(f_{v,m_3+n})\left(
   \begin{bmatrix}0&\nu \bar{\fa}^{-1}\\ -\fa&0\end{bmatrix}
   \begin{bmatrix}1&w\\0&1\end{bmatrix}
   \right)
   \cdot |\det C|^{3}_v|\fa\bar{\fa}|_v |\nu|^{-4}_v\,dw\,dW\,d r\,d^\times\nu\,d^\times C\,d^\times\fa,
\end{aligned}
\end{equation}}
\hspace{-.5em}where $T'_\bS=\left\{(\imath_\bS(\fz),\fz):\fz\in \mr{Res}_{\cK/\bQ}\GL(1)\right\}\subset \GL(2)\times_{\GL(1)} \mr{Res}_{\cK/\bQ}\GL(1)$. It follows from \cite[Lemma~2-4]{Sugano} that
\begin{equation}\label{eq:Sugano-decomp}
   \GL(2,\bQ_v)\times\cK^\times_v=\bigsqcup\limits_{l\geq -\mu_v} T_\bS(\bQ_p)\left(\begin{bmatrix}\varpi^l_v\\&1\end{bmatrix}\GL(2,\bZ_v)\times \cK^\times_v\right),
\end{equation}
where $\mu_v$ is a non-negative integer defined in equation (2-3) in {\it loc.cit}. (For our purpose of bounding $I_2,I_3$, we do not need to know the exact value of $\mu_v$.) Thus, we can write
\begin{align*}
   \nu\ltrans{C}^{-1}&=\imath_\bS(\fz)\begin{bmatrix}\varpi_v^l\\&1\end{bmatrix}h_0,
   &\fa&=\fz r^{-1}\fy^{-1},
   &&\begin{array}{l}
   \fz\in\cK^\times_v,\,l\geq -\mu_v,\, h_0\in\GL(2,\bZ_v),\\
   r\in\bQ^\times_v,\,\fy\in\cK^\times_v,\end{array}.
\end{align*}
Also, it is not difficult to see (cf. \cite[Lemma 3.5.3]{Furusawa} that there exists a constant $B_1>0$ such that, for any $U\subset \cK^\times_v$,
\[
   \vol\left(T'_\bS(\bQ_v)\backslash T'_\bS(\bQ_v)\left(\begin{bmatrix}\varpi_v^l\\&1\end{bmatrix}\GL(2,\bZ_v)\times U\right)\right)\leq B_1\cdot |\varpi_v|_v^{-l}\cdot \vol(U).
\]
Write
\begin{align*}
   \ltrans{h}^{-1}_0  W h^{-1}_0&=\begin{bmatrix}w_1&w_2\\w_2&w_4\end{bmatrix},
   &&\,w_1,w_2,w_4\in\bQ_v.
\end{align*}
Then from \eqref{eq:I3m-1}, we get
{\small
\begin{equation}\label{eq:I3m-2}
\begin{aligned}
   |I_{3,n}|\leq&\sum_{l\geq -\mu_v} B_1
   \int_{\GL(2,\bZ_v)}\int_{\cK^\times_v}\int_{\bQ^\times_v}\int_{\bQ_v}\int_{(\bQ_v)^3-(\varpi_v^{-2n+e}\bZ_v)^3}\int_{\bQ_v}  
    |\varpi|_v^{2ls}|\fy\bar{\fy}|^{s+\frac{1}{2}}_v |\nu|^{-s+\frac{1}{2}}_v |r|^{-2}_v\\
   &\times \Schw_{v,1,\sym}\left(\ltrans{h}_0\begin{bmatrix}w_1+r^{-1}\nu^{-1}\varpi_v^{2l}\cdot\alphaS\alphabS&w_2+r^{-1}\nu^{-1}\varpi_v^l\cdot\frac{\alphaS+\alphabS}{2}\\w_2+r^{-1}\nu^{-1}\varpi_v^l\cdot\frac{\alphaS+\alphabS}{2}&w_4+r^{-1}\nu^{-1}\end{bmatrix}h_0\right)\\
   &\times \Schw_{v,1,\alt}\left(r^{-1}\nu^{-1}\varpi_v^l(\det h_0)\right)
   \cdot \Schw_{v,0}\left(\begin{bmatrix}\alphaS&1\\\alphabS&1\end{bmatrix}^{-1}\begin{bmatrix}\varpi_v^l \alphaS\fy&\fy\\ \varpi^l_v\alphabS\bar{\fy}&\bar{\fy}\end{bmatrix}h_0\right)
   \cdot \Schw_{v,2}(w+r\nu\fy\bar{\fy})\\
   &\times \left| \psi_v(-\bbc r)\cdot  \Bes^{\Pi_v}_{\bS,\Lambda_v}(\varphi_{v,m_1+n,m_2+n})\left(\begin{bmatrix}&&-\varpi^{l}_v\\&&&-1\\\nu \varpi^{-l}_v\\&\nu\end{bmatrix}
   \begin{bmatrix}1&& w_1&w_2\\&1& w_2& w_4\\&&1\\&&&1\end{bmatrix}
   \begin{bmatrix} h_0&0\\0&\ltrans{h}^{-1}_0\end{bmatrix}\right)\right|\\
   &\times  \left|\Whi^{\pi_v,\Upsilon_v}_{\bbc}(f_{v,m_3+n})\left(
   \begin{bmatrix}0&r\nu \bar{\fy}\\ -r^{-1}\fy^{-1}&0\end{bmatrix}
   \begin{bmatrix}1&w\\0&1\end{bmatrix}
   \right)\right|
   \,dw\,dw_1\,dw_2\,dw_4\,d r\,d^\times\nu\,d^\times\fy\,d^\times h_0
\end{aligned}
\end{equation}}

Let 
\[
   B_2=\ord_v(\mu_v)+\max\left\{-\ord_v(\alphaS),-\ord_v(\alphabS),-\ord_v(\alphaS\alphabS),-\ord_v(\frac{\alphaS+\alphabS}{2}),0\right\}.
\]
Since $h_0\in\GL(2,\bZ_v)$, the nonvanishing of $\Schw_{v,1,\alt}$ and $\Schw_{v,0}$ in \eqref{eq:I3m-2} implies
\begin{equation}\label{eq:Im3-bd1}
\begin{aligned}
   &r^{-1}\nu^{-1}\varpi^l_v\in \varpi^{-e}_v\bZ_v,
   &\hspace{4em}\fy&\in \varpi^{-e-B_2}_v\cO^2_{\cK,v},
\end{aligned}
\end{equation}
which, together with the nonvanishing of $\Schw_{v,1,\sym}$, implies 
\begin{equation}\label{eq:Im3-bd2}
   w_1,w_2\in \varpi^{-e-B_2}_v\bZ_v.
\end{equation} 
The nonvanishing of $\Schw_{v,1,\sym}$ also implies
\begin{equation}\label{eq:Im3-bd9}
   r^{-1}\nu^{-1}\in -w_4+\varpi^{-e}_v\bZ_v.
\end{equation}
For $n\gg 0$, \eqref{eq:Im3-bd2} plus $W \notin \Sym(\varpi^{-2n+e}_v\bZ_v)$ implies that
\begin{equation}\label{eq:Im3-bd8}
   w_4\notin \varpi^{-2n+e}_v\bZ_v.
\end{equation}
It follows from \eqref{eq:Im3-bd1}\eqref{eq:Im3-bd9}\eqref{eq:Im3-bd8} that $r\nu\fy\bar{\fy}\in \varpi^{2n-e}_v \varpi^{-2e-2B_2}_v \bZ_p$. Combining this with the nonvanishing of $\Schw_{v,2}$, we get
\begin{equation}\label{eq:Im3-bd3}
    w\in \varpi^{-e}_v\bZ_p.
\end{equation}
Thus, when  $n\gg 0$ and the product of the Schwartz functions in \eqref{eq:I3m-2} does not vanish, the second condition in our choice of $e$ implies 
\begin{equation}\label{eq:Im3-bd4}
\begin{aligned}
   &\Bes^{\Pi_v}_{\bS,\Lambda_v}(\varphi_{v,m_1+n,m_2+n})\left(\begin{bmatrix}&&-\varpi^{l}_v\\&&&-1\\\nu \varpi^{-l}_v\\&\nu\end{bmatrix}
   \begin{bmatrix}1&& w_1&w_2\\&1&w_2& w_4\\&&1\\&&&1\end{bmatrix}
   \begin{bmatrix} h_0&0\\0&\ltrans{h}^{-1}_0\end{bmatrix}\right)\\
   =&\,\Bes^{\Pi_v}_{\bS,\Lambda_v}(\varphi_{v,m_1+n,m_2+n})\left(\begin{bmatrix}&&-\varpi^{l}_v\\&&&-1\\\nu \varpi^{-l}_v\\&\nu\end{bmatrix}
   \begin{bmatrix}1\\&1&&- r^{-1}\nu^{-1}\\&&1\\&&&1\end{bmatrix}
   \begin{bmatrix} h_0&0\\0&\ltrans{h}^{-1}_0\end{bmatrix}
   \right)\\
   =&\, \Bes^{\Pi_v}_{\bS,\Lambda_v}(\varphi_{v,m_1+n,m_2+n})\left(\begin{bmatrix}1\\&1&&r\\&&1\\&&&1\end{bmatrix}\begin{bmatrix}\varpi^{l}_v\\&r\nu\\&&\nu\varpi^{-l}_v\\&&& r^{-1}\end{bmatrix}
   \begin{bmatrix}&&-1\\&-1\\1\\&&&-1\end{bmatrix}\begin{bmatrix} h_0&0\\0&\ltrans{h}^{-1}_0\end{bmatrix}
   \right)\\
   =&\,\psi_v(\bbc r)\cdot \Bes^{\Pi_v}_{\bS,\Lambda_v}(\varphi_{v,m_1,m_2})\left(\begin{bmatrix}\varpi^{l+n}_v\\&r\nu\varpi^{-n}_v\\&&\nu\varpi^{-l-n}_v\\&&& r^{-1}\varpi^{n}_v\end{bmatrix}
   \begin{bmatrix}&&-1\\&-1\\1\\&&&-1\end{bmatrix}\begin{bmatrix} h_0&0\\0&\ltrans{h}^{-1}_0\end{bmatrix}
   \right)
\end{aligned}
\end{equation}
and 
\begin{equation}\label{eq:Im3-bd5}
\begin{aligned}
   \Whi^{\pi_v,\Upsilon_v}_{\bbc}(f_{v,m_3+n})\left(
   \begin{bmatrix}0&r\nu \bar{\fy}\\- r^{-1}\fy^{-1}&0\end{bmatrix}
   \begin{bmatrix}1&w\\0&1\end{bmatrix}
   \right)
   &=\Whi^{\pi_v,\Upsilon_v}_{\bbc}(f_{v,m_3})\begin{pmatrix}\varpi^{n}_v r\nu\bar{\fy}\\&-\varpi^{-n}_v r^{-1}\fy^{-1} \end{pmatrix}.
\end{aligned}
\end{equation}
We can take $B_3>0$ (depending on $\varphi_v,f_v$ and $m_1,m_2,m_3, e$) such that for all $g_0\in\GSp(4,\bZ_v)$,
{\small
\begin{equation}\label{eq:Im3-bd6} 
\begin{aligned}
  \left\vert \Bes^{\Pi_v}_{\bS,\Lambda_v}(\varphi_{v,m_1,m_2})\left(\begin{bsm}a_1\\&a_2\\&&\nu a^{-1}_1\\&&&\nu a^{-1}_2\end{bsm}g_0\right)\right\vert
   &\left\{\begin{array}{ll}  \leq B_3\,|\nu^{-1} a^2_1|^{-B_3}_v|\nu^{-1} a^2_2|^{-B_3}_v, &\nu^{-1} a^2_1, \nu^{-1} a^2_2\in \varpi^{-e-\val_v(\bbc)}_v\bZ_v, \\[1em]=0,&\text{otherwise,}\end{array}\right.\\
   \left\vert \Whi^{\pi_v,\Upsilon_v}_{\bbc}(f_{v,m_3})\begin{pmatrix}\fa\\&\nu\bar{\fa}^{-1}\end{pmatrix}\right\vert
   &\left\{\begin{array}{ll}\leq B_3\,|\nu^{-1}\fa\bar{\fa}|^{-B_3}_v, &\nu^{-1}\fa\bar{\fa}\in \varpi^{-e-\val_v(\bbc)}\bZ_v,\\[1em] =0, &\text{otherwise}.\end{array}\right.
\end{aligned}
\end{equation}
}
\hspace{-.5em}In order for \eqref{eq:Im3-bd4} to be nonvanishing, 
\begin{equation*}
    r^2\nu\in\varpi^{2n-e-\val_v(\bbc)}\bZ_v,
\end{equation*}
which, together with \eqref{eq:Im3-bd1}, gives
\begin{equation}\label{eq:Im3-bd7}
   \nu^{-1}\in \varpi_v^{2n-2l-3e-\val_v(\bbc)}.
\end{equation}

Now plug \eqref{eq:Im3-bd4} and \eqref{eq:Im3-bd5} into \eqref{eq:I3m-2}, integrate over $w,w_1,w_2,w_4$, apply the bound \eqref{eq:Im3-bd6} and the conditions \eqref{eq:Im3-bd1} \eqref{eq:Im3-bd7}. We get
\begin{align*}
   |I_{3,n}|&\leq   B_1 |\varpi_v|_v^{-4e-2B_2} \cdot 
   \sum_{l\geq -\mu_v}     
   \int_{\cK^\times_v\cap \varpi^{-e-B_2}_v\cO_{\cK,v}} 
   \int_{\bQ^\times_v}\int_{\bQ^\times_v}  
   |\varpi_v|_v^{-l+2ls} |\fy\bar{\fy}|_v^{s+\frac{1}{2}} |\nu|_v^{-s+\frac{1}{2}} |r|^{-1}_v   
   \cdot \mathds{1}_{\varpi^{-e}_v\bZ_v}(r^{-1}\nu^{-1}\varpi^l_v)
    \\
     &\times \mathds{1}_{\varpi^{2n-2l-3e-\val_v(\bbc)}_v\bZ_v}(\nu^{-1}) 
     \cdot  B_3\,|\varpi^{2l+2n}_v\nu^{-1}|^{-B_3}_v |\varpi^{-2n}_v r^2\nu|^{-B_3}_v \cdot B_3\,|\varpi^{2n}_v r^2\nu\fy\bar{\fy}|^{-B_3}_v   
     \,d^\times\fy, \\
    &=B_1B^2_3    |\varpi_v|_v^{-4e-2B_2} 
    \int_{\cK^\times_v\cap \varpi^{-e-B_2}_v\cO_{\cK,v}} |\fy\bar{\fy}|_v^{s+\frac{1}{2}-B_3} d^\times\fy\\
    &\quad\times\sum_{l\geq -\mu_v} |\varpi_v|_v^{-l+2ls-2lB_3}
    \sum_{\val_v(\nu)\leq 2l-2n+3e+\val_v(\bbc)} |\nu|_v^{-s+\frac{1}{2}-B_3}
    \sum_{\val_v(r)\leq l+e-\val_v(\nu)} |r|^{-1-4B_3}_v.
\end{align*}
When $n\gg 0$, $\mr{Re}(s)\gg 0$, a direct computation (using $\sum\limits_{n\geq a} x^n\leq 2x^a$ for all $a\in\bZ$, $0\leq x\leq \frac{1}{2}$) gives that
\begin{align*}
   |I_{n,3}|\leq 16B_1B^2_3 \cdot |\varpi_v|^{\big(2n-5e-2B_2-\val_v(\bbc)\big)s-n(3+6B_3)+B_4}_v,
\end{align*}
where $B_4$ is a constant depending on $e,\bbc,\mu_v,B_2,B_3$, independent of $n,s$. It follows that, when $\mr{Re}(s)\gg 0$, $\lim\limits_{n\to\infty}\big(\lambda_v(\varphi_v)\lambda_v(f_v)\big)^{-n}I_{3,n}=0$.

\vspace{1em}
\noindent\underline{Bounding $I_{2,n}$.}\\

Like bounding $I_{3,n}$,  we can assume that $\chi_v,\Xi_v$ and the central characters of $\Pi_v$ and $\pi_v$ are unitary, and $\Phi_v$ is the characteristic function of certain compact open subset of $\Her_3(\cK_v)$. 

Now $(g,h)\in \cU_{2,n}$. Write $g$ as in \eqref{eq:ghII}, and
\begin{align*}
   \begin{bmatrix}0&1\\-1&0\end{bmatrix} h&
     =\begin{bmatrix}1&-\fd\fb^{-1}\\0 &1\end{bmatrix}
     \begin{bmatrix}0&1\\-1&0\end{bmatrix}
     \begin{bmatrix}0&\fb\\ -\nu\bar{\fb}^{-1}&0 \end{bmatrix}
     \begin{bmatrix}1&0\\\fb^{-1}\fa&1\end{bmatrix}.
\end{align*} 
Then
\begin{align*}
   \Whi^{\pi_v,\Upsilon_v}_\bbc(f_{v,m_3+n})(h)=\psi_v(-\bbc \fd\fb^{-1})\cdot\Whi^{\pi_v,\Upsilon_v}_\bbc(f_{v,m_3})\left(\begin{bmatrix}-\varpi^{-n}_v\nu\bar{\fb}^{-1}\\&-\varpi^n_v\fb\end{bmatrix}\begin{bmatrix}1&0\\\varpi^{-2n}_v\fb^{-1}\fa&1\end{bmatrix}\right)
\end{align*}
Because $C^{-1}D\in\Sym_2(\varpi^{-2m+e}_v\bZ_v)$, by the third condition in our choice of $e$, we have
\begin{align*}
    \Bes^{\Pi_v}_{\bS,\Lambda_v}(\varphi_{v,m_1+n,m_2+n})(g)&=\psi_v(\Tr\bS AC^{-1})\cdot 
    \Bes^{\Pi_v}_{\bS,\Lambda_v}(\varphi_{v,m_1,m_2})\left(\begin{bmatrix} \varpi^n_v \nu\ltrans{C}^{-1}\\&\varpi^{-n}_v C\end{bmatrix}\begin{bmatrix}&-\bid_2\\ \bid_2\end{bmatrix}\right).
\end{align*}
Because $\fb^{-1}\fa\in \varpi_v^{2n+1-e}\bZ_v$, for sufficiently large $n$, we have
\begin{equation}\label{eq:Im2-F}
\begin{aligned}
    \dsec^\bc_{v,\Schw_v}(s,\chi,\Xi)\left(\cS^{-1}\imath(\eta_\bS\, g,h)\right)&=\dsec^\bc_{v,\Schw_v}(s,\chi,\Xi)\left(\cS^{-1}\imath(\eta_\bS \,g,\begin{bmatrix}0&\fb\\-\nu\bar{\fb}^{-1}&\fd\end{bmatrix}\right).
\end{aligned}
\end{equation}
Put 
\begin{align*}
   r&=\fd\fb^{-1}-\bbc^{-1}\Tr\bS AC^{-1}, 
   &w&=\varpi^{-2n}_v\fb^{-1}\fa, 
   &W&=C^{-1}D\in \Sym_2(\varpi^{-2n+e}_v\bZ_v).
\end{align*} 
With $\begin{bmatrix}\fA&\fB\\ \fC&\fD\end{bmatrix}=\cS^{-1}\imath\left(\eta_\bS g,\begin{bmatrix}0&\fb\\-\nu\bar{\fb}^{-1}&\fd\end{bmatrix}\right)$, we have
\begin{align*}
   \fC&=\begin{bmatrix}
   1&-\alphabS\\&1\\&&1
   \end{bmatrix}
   \begin{bmatrix}
   C&0\\-\begin{bmatrix}\alphaS&1\end{bmatrix}A&-\nu\bar{\fb}^{-1}
   \end{bmatrix},
   &\fD&=\begin{bmatrix}
   1&-\alphabS\\&1\\&&1
   \end{bmatrix}
   \begin{bmatrix}
   D&-\begin{bmatrix}\alphabS\\1\end{bmatrix}\fb\\-\begin{bmatrix}\alphaS&1\end{bmatrix}B&\fd.
   \end{bmatrix}
\end{align*}
\begin{align*}
   \det\fC&=-\bar{\fb}^{-1}\nu\det C,
   &\fC^{-1}\fD&=\begin{bmatrix}
   W&-C^{-1}\begin{bmatrix}\alphabS\\1\end{bmatrix}\fb\\[3ex]
   -\bar{\fb}\begin{bmatrix}\alphaS&1\end{bmatrix}\ltrans{C}^{-1}&\nu^{-1}\fb\bar{\fb}r
   \end{bmatrix}.
\end{align*}
Then
{\small
\begin{align*}
   I_{2,n}&= \frac{1-|\varpi_v|_v}{(1-|\varpi_v|_v^4)(1+|\varpi_v|_v)}\\
   &\,\times\int_{T'_\bS(\bQ_v)\backslash (\GL(2,\bQ_v)\times\cK^\times_v)}\int_{\bQ^\times_v}\int_{\varpi_v^{-e+1}\bZ_v\times \Sym_2(\varpi^{-2m+e}\bZ_v)}\int_{\bQ_v} |\nu\fb\bar{\fb}(\det C)^{-2}|^{s+\frac{3}{2}}_v\\
   &\times \Xi_v\left(-\nu\bar{\fb}(\det C)^{-1}\right)\chi_v\left(\nu\fb\bar{\fb}(\det C)^{-2}\right) 
   \cdot \Schw_v\begin{pmatrix}
   W&-C^{-1}\begin{bmatrix}\alphabS\\1\end{bmatrix}\fb\\[1em]
   -\bar{\fb}\begin{bmatrix}\alphaS&1\end{bmatrix}\ltrans{C}^{-1}&\nu^{-1}\fb\bar{\fb}r\end{pmatrix}\\
    &\times\psi_v(-\bbc r)\cdot  \Bes^{\Pi_v}_{\bS,\Lambda_v}(\varphi_{v,m_1,m_2})\left(\begin{bmatrix}\varpi^n_v \nu \ltrans{C}^{-1}\\&\varpi^{-n}_v C\end{bmatrix}\begin{bmatrix}&-\bid_2\\ \bid_2\end{bmatrix}\right)\\
   &\times \Whi^{\pi_v,\Upsilon_v}_{\bbc}(f_{v,m_3})\left(\begin{bmatrix}-\varpi^{-n}_v\nu\bar{\fb}^{-1}\\&-\varpi^n_v\fb\end{bmatrix}\begin{bmatrix}1&0\\w&1\end{bmatrix}\right)
   \cdot |\det C|^3_v |\fb\bar{\fb}|_v |\nu|^{-4}_v dr\,dW\,dw\,d^\times\nu\,d^\times C\,d^\times\fb.
\end{align*}
}
Writing
\begin{align*}
   \nu\ltrans{C}^{-1}&=\imath_\bS(\fz)\begin{bmatrix}\varpi^l_v\\&1\end{bmatrix}h_0,
   &   \nu\bar{\fb}^{-1}&=\fz\fy^{-1},
   &&l\geq -\mu_v,\,h_0\in\GL(2,\bZ_v),\,\fz,\fy\in\cK^\times_v,
\end{align*}
noticing that $\Schw_v$ is nonzero only if $W\in \Sym_2(\varpi^{-e}_v\bZ_p)$, and using \eqref{eq:Sugano-decomp}, we have
\begin{align*}
   |I_{2,n}|\leq &\,\sum_{l\geq -\mu_v}B_1|\varpi_v|_v^{-l}\int_{\cK^\times_v\times\bQ^\times_v}\int_{\varpi^{-e+1}_v\bZ_v\times\Sym_2(\varpi^{-e}_v\bZ_v)}\int_{\bQ_v} 
   |\varpi|^{2ls}_v |\nu|^{-s+\frac{5}{2}}_v|\fy\bar{\fy}|^{s+\frac{5}{2}}_v\\
   &\times  \Xi_v(-\varpi^l_v\fy)\,\chi_v(\varpi^{2l}_v\nu^{-1}\fy\bar{\fy})
   \cdot \Schw_{v,1,\sym}(W)\cdot\Schw_{v,1,\alt}(0) \\
   &\times \Schw_{v,0}\left(\begin{bmatrix}\alphaS&1\\\alphabS&1\end{bmatrix}^{-1}\begin{bmatrix} \varpi^l_v\alphaS\fy&\fy\\ \varpi^l_v\alphabS\bar{\fy}&\bar{\fy}\end{bmatrix}h_0\right)\cdot \psi_v(-\bbc r) \cdot \Schw_{v,2}(\nu\fy\bar{\fy}r)\\
   &\times \Bes^{\Pi_v}_{\bS,\Lambda_v}(\varphi_{v,m_1,m_2})
   \left(
   \begin{bmatrix}\varpi^{l+n}_v\\&\varpi^n_v\\&&\nu \varpi^{-l-n}_v\\&&&\nu\varpi^{-n}_v\end{bmatrix} \begin{bmatrix}&-h_0\\\ltrans{h}^{-1}_0\end{bmatrix}
   \right)\\
   &\times \Whi^{\pi_v,\Upsilon_v}_{\bbc}(f_{v,m_3})\left(\begin{bmatrix}-\varpi^{-n}_v\fy^{-1}\\&-\varpi^n_v\nu\bar{\fy}\end{bmatrix}\begin{bmatrix}0&1\\1&w\end{bmatrix}\right) \,dr\,dW\,dw\,d^\times\nu\,d^\times \fy.
\end{align*}
By our assumption that $\Phi_v$ is a characteristic function and the second condition in our choice of $e$,
\[
    \left\vert\int_{\bQ_v}\psi_v(-\bbc r) \cdot \Schw_{v,2}(\nu\fy\bar{\fy}r)\,dr\right\vert\leq \mathds{1}_{\varpi^{-e}\bZ_v}(\bbc \nu^{-1}\fy^{-1}\bar{\fy}^{-1})\,|\bbc\varpi^{2e}_v|^{-1}_v.
\]
The nonvanishing of $\Schw_{v,0}$ in the above integral implies that
\begin{align*}
   \fy\in \varpi^{-e-B_2}_v\cO_{\cK,v},
\end{align*}
and the nonvanishing of $\Whi^{\pi_v,\Upsilon_v}_{\bbc}$-term plus $w=\varpi^{-2n}_v\fb^{-1}\fa\in\varpi^{1-e}_v\bZ_v$  and the second condition in our choice of $e$ implies that
\begin{align*}
   \varpi^{-2n}_v\nu^{-1}\fy^{-1}\bar{\fy}^{-1}\in \varpi^{-3e}_v\bZ_v.
\end{align*}
Hence,
\begin{align*}
   |I_{2,n}|&\leq B_1|\bbc|^{-1}_v \sum_{l\geq 0-\mu_v}|\varpi_v|^{2ls-l-5e}_v \int_{\cK^\times_v\times\bQ^\times_v} \int_{\varpi^{-e+1}_v\bZ_v}
   \mathds{1}_{\varpi^{-e-B_2}_v\bZ_v}(\fy)\cdot\mathds{1}_{\varpi^{2n-3e}_v\bZ_v}(\nu^{-1}\fy^{-1}\bar{\fy}^{-1})\\
   &\times |\nu|^{-s+\frac{5}{2}}|\fy\bar{\fy}|^{s+\frac{5}{2}}_v
   \cdot \left\vert \Bes^{\Pi_v}_{\bS,\Lambda_v}(\varphi_{v,m_1,m_2})
   \left(
   \begin{bmatrix}\varpi^{l+n}_v\\&\varpi^n_v\\&&\nu \varpi^{-l-n}_v\\&&&\nu\varpi^{-n}_v\end{bmatrix} \begin{bmatrix}&-h_0\\\ltrans{h}^{-1}_0\end{bmatrix}
   \right)\right\vert\\
   &\times \left\vert \Whi^{\pi_v,\Upsilon_v}_{\bbc}(f_{v,m_3})\left(\begin{bmatrix}-\varpi^{-n}_v\fy^{-1}\\&-\varpi^n_v\nu\bar{\fy}\end{bmatrix}\begin{bmatrix}1&0\\w&1\end{bmatrix}\right)\right\vert\,dw\,d^\times\fy\,d^\times\nu
\end{align*}
There exists $B_3>0$ such that the bound for $\Bes^{\Pi_v}_{\bS,\Lambda_v}$ in \eqref{eq:Im3-bd6} holds and for all $g_0\in\GSp(4,\bZ_v)$ and $w\in \varpi^{-e+1}_v\bZ_v$,
\[
\begin{aligned}
  \left\vert \Bes^{\Pi_v}_{\bS,\Lambda_v}(\varphi_{v,m_1,m_2})\left(\begin{bsm}a_1\\&a_2\\&&\nu a^{-1}_1\\&&&\nu a^{-1}_2\end{bsm}g_0\right)\right\vert
   &\left\{\begin{array}{ll}  \leq B_3\,|\nu^{-1} a^2_1|^{-B_3}_v|\nu^{-1} a^2_2|^{-B_3}_v, &\nu^{-1} a^2_1, \nu^{-1} a^2_2\in \varpi^{-e-\val_v(\bbc)}_v\bZ_v, \\[1em]=0,&\text{otherwise,}\end{array}\right.\\
    \left\vert \Whi^{\pi_v,\Upsilon_v}_{\bbc}(f_{v,m_3})\left(\begin{bmatrix}\fa\\&\nu\bar{\fa}^{-1}\end{bmatrix}\begin{bmatrix}0&1\\1&w\end{bmatrix}\right)\right\vert
   &\left\{\begin{array}{ll}\leq B_3\,|\nu^{-1}\fa\bar{\fa}|^{-B_3}_v, &\nu^{-1}\fa\bar{\fa}\in \varpi^{-3e}\bZ_v\\[1em] =0, &\text{otherwise}.\end{array}\right.
\end{aligned}
\]
Then, with $n\gg 0$, $\mr{Re}(s)\gg 0$,
\begin{align*}
   |I_{2,n}|\leq &\,B_1|\bbc|^{-1}_v|\varpi_v|^{-e+1}_v 
   \sum_{l\geq -\mu_v}
   \int_{\bQ^\times_v}
   \int_{\cK^\times_v\cap \varpi^{-e-B_2}_v\bZ_v}
    |\varpi_v|^{2ls-l-5e}_v \cdot |\nu|^{-s+\frac{5}{2}}|\fy\bar{\fy}|^{s+\frac{5}{2}}_v 
   \cdot \mathds{1}_{\varpi^{2n-3e}_v\bZ_v}(\nu^{-1}\fy^{-1}\bar{\fy}^{-1})\\[1ex]
   &\hspace{10em}\times B_3\,|\varpi^{2l+2n}_v\nu^{-1}|^{-B_3}_v |\varpi^{2n}_v\nu^{-1}|^{-B_3}_v\cdot B_3\,|\varpi^{-2n}_v\nu^{-1}\fy\bar{\fy}^{-1}|^{-B_3}_v\,d^\times\fy\,d^\times\nu\\[1ex]
   =&\,B_1B^2_3|\bbc|^{-1}_v |\varpi_v|^{-2nB_3-6e+1}_v 
   \sum_{l\geq -\mu_v} |\varpi_v|^{2ls-l-2lB_3}_v \\
   &\times \int_{\bQ^\times_v}\int_{\cK^\times_v\cap \varpi^{-e-B_2}_v\bZ_v} |\fy\bar{\fy}|^{s+\frac{5}{2}+B}
   |\nu|^{-s+\frac{5}{2}+3B_3}
   \cdot \mathds{1}_{\varpi^{2n-3e}_v\bZ_v}(\nu^{-1}\fy^{-1}\bar{\fy}^{-1})\,d^\times\fy\,d^\times\nu.
\end{align*}
A direct computation (using $\sum\limits_{n\geq a} x^n\leq 2x^a$ for all $a\in\bZ$, $0\leq x\leq \frac{1}{2}$) shows that
\[
   |I_{2,n}|\leq 16|\bbc|^{-1}_vB_1B^2_3 |\varpi_v|^{(2n-7e-2\mu_v-4B_2)s-5n-8nB_3+B_4}_v,
\]
with $B_4$ a constant depending on $e,\bbc,\mu_v,B_2,B_3$, independent of $n,s$. It follows that, when $\mr{Re}(s)\gg 0$, $\lim\limits_{n\to\infty}\big(\lambda_v(\varphi_v)\lambda_v(f_v)\big)^{-n}I_{2,n}=0$.

\vspace{1em}
\noindent{\underline{Rewriting $I_{1,n}$.}}
\vspace{.5em}

With $g,h$ as in \eqref{eq:ghII} and $r,W,w$ as in \eqref{eq:r}\eqref{eq:Ww}, by our choice of $e$, when $(g,h)\in\cU_{1,n}$, $\varphi_{v,m_1+n,m_2+n}$ (resp. $f_{v,m_3+n}$) is invariant under the action of $\begin{bmatrix}\bid_2&W\\0&\bid_2\end{bmatrix}$ (resp. $\begin{bmatrix}1&w\\0&1\end{bmatrix}$). The computation leading to \eqref{eq:I3m-1} gives
\begin{align*}
   I_{1,n}=
   & \frac{1-|\varpi_v|_v}{(1-|\varpi_v|_v^4)(1+|\varpi_v|_v)}\int_{T'_\bS(\bQ_v)\backslash (\GL(2,\bQ_v)\times\cK^\times_v)}\int_{\bQ^\times_v}\int_{\bQ_v}
   |\det C|^{-2s}_v|\fa\bar{\fa}|^{-s-\frac{1}{2}}_v |\nu|^{3s+\frac{1}{2}}_v|r|^{-2s-3}_v\\
   &\times \chi_v\left(\nu^3 r^{-2}\fa^{-1}\bar{\fa}^{-1}(\det C)^{-2}\right)\, \Xi_v\left(\nu^2 r^{-1}\fa^{-1}(\det C)^{-1}\right)\\
   &\times \Schw_{v,1,\alt}\left(r^{-1}\nu(\det C)^{-1}\right)
   \Schw_{v,0}\left(r^{-1}\imath_\bS(\fa)^{-1}\nu\ltrans{C}^{-1}\right)\\
   &\times \psi_v(-\bbc r)\cdot \Bes^{\Pi_v}_{\bS,\Lambda_v}(\varphi_{v,m_1+n,m_2+n})\left(\begin{bmatrix}0&-\nu\ltrans{C}^{-1}\\C&0\end{bmatrix}\right)\cdot  \Whi^{\pi_v,\Upsilon_v}_{\bbc}(f_{v,m_3+n})\left(
   \begin{bmatrix}0&\nu \bar{\fa}^{-1}\\ -\fa&0\end{bmatrix}\right)\\  
   &\times\int_{ \Sym_2(\varpi^{-2n+e}_v\bZ_v)} \Schw_{v,1,\sym}\left(W+r^{-1}\nu C^{-1}\begin{bmatrix}\alphaS\alphabS&\frac{\alphaS+\alphabS}{2}\\ \frac{\alphaS+\alphabS}{2}&1\end{bmatrix}\ltrans{C}^{-1}\right) \,dW\\
   &\times \int_{\varpi^{-2n+e}_v\bZ_v} \Schw_{v,2}\left(w+r^{-1}\fa^{-1}\bar{\fa}^{-1}\nu\right) \,dw
   \,\, d r\,d^\times\nu\,d^\times C\,d^\times\fa.
\end{align*}
When $n\gg 0$,
\begin{align*}
   &\int_{ \Sym_2(\varpi^{-2n+e}_v\bZ_v)} \Schw_{v,1,\sym}\left(W+r^{-1}\nu C^{-1}\begin{bmatrix}\alphaS\alphabS&\frac{\alphaS+\alphabS}{2}\\ \frac{\alphaS+\alphabS}{2}&1\end{bmatrix}\ltrans{C}^{-1}\right) \,dW\\
   =&\, \mathds{1}_{\Sym_2(\varpi^{-2n+e}_v\bZ_v)}\left(r^{-1}\nu C^{-1}\begin{bmatrix}\alphaS\alphabS&\frac{\alphaS+\alphabS}{2}\\ \frac{\alphaS+\alphabS}{2}&1\end{bmatrix}\ltrans{C}^{-1}\right)
   \int_{\Sym_2(\bQ_v)} \Schw_{v,1,\sym}\left(W\right) \,dW,
\end{align*}
and
\begin{align*}
   \int_{\varpi^{-2n+e}_v\bZ_v} \Schw_{v,2}\left(w+r^{-1}\fa^{-1}\bar{\fa}^{-1}\nu\right) \,dw
   &=\mathds{1}_{\varpi^{-2n+e}_v\bZ_v} \left(r^{-1}\fa^{-1}\bar{\fa}^{-1}\nu\right)
   \int_{\bQ_v} \Schw_{v,2}\left(w\right) \,dw,
\end{align*}   
so we have
{\small
\begin{equation}\label{eq:I1m-1}
\begin{aligned}
   I_{1,n}=&\,C\cdot
   \int_{T'_\bS(\bQ_v)\backslash (\GL(2,\bQ_v)\times\cK^\times_v)}\int_{\bQ^\times_v}\int_{\bQ_v} |\det C|^{-2s}_v|\fa\bar{\fa}|^{-s-\frac{1}{2}}_v |\nu|^{3s+\frac{1}{2}}_v|r|^{-2s-3}_v\\
   &\times \chi_v\left(\nu^3 r^{-2}\fa^{-1}\bar{\fa}^{-1}(\det C)^{-2}\right)\cdot \Xi_v\left(\nu^2 r^{-1}\fa^{-1}(\det C)^{-1}\right)\\
   &\times \Schw_{v,1,\alt}\left(r^{-1}\nu(\det C)^{-1}\right)
   \Schw_{v,0}\left(r^{-1}\imath_\bS(\fa)^{-1}\nu\ltrans{C}^{-1}\right)\\
   &\times \psi_v(-\bbc r)\cdot \Bes^{\Pi_v}_{\bS,\Lambda_v}(\varphi_{v,m_1+n,m_2+n})\begin{pmatrix}&-\nu\ltrans{C}^{-1}\\C&\end{pmatrix}
   \cdot \Whi^{\pi_v,\Upsilon_v}_{\bbc}(f_{v,m_3+n})\begin{pmatrix}&\nu \bar{\fa}^{-1}\\ -\fa&\end{pmatrix}\\
   &\times\mathds{1}_{\Sym_2(\varpi^{-2n+e}_v\bZ_v)}\left(r^{-1}\nu C^{-1}\begin{bmatrix}\alphaS\alphabS&\frac{\alphaS+\alphabS}{2}\\ \frac{\alphaS+\alphabS}{2}&1\end{bmatrix}\ltrans{C}^{-1}\right)
   \mathds{1}_{\varpi^{-2n+e}_v\bZ_v} \left(r^{-1}\fa^{-1}\bar{\fa}^{-1}\nu\right) \, d r\,d^\times\nu\,d^\times C\,d^\times\fa,
\end{aligned}
\end{equation}
}
\hspace{-.5em}with 
\begin{align*}
   C=&\frac{1-|\varpi_v|_v}{(1-|\varpi_v|_v^4)(1+|\varpi_v|_v)}\int_{\Sym_2(\bQ_v)} \Schw_{v,1,\sym}\left(W\right) \,dW \int_{\bQ_v} \Schw_{v,2}\left(w\right) \,dw.
\end{align*}
Let
\begin{align*}
   T&=\imath_\bS\left(\fa\right)^{-1}\nu\ltrans{C}^{-1},
   &\nu'&=-\nu\fa^{-1}\bar{\fa}^{-1}
\end{align*}
Then
\begin{align*}
   \begin{bmatrix}0&-\nu\ltrans{C}^{-1}\\C&0\end{bmatrix}
   &=\begin{bmatrix}\imath_\bS(\fa)\\&\ltrans{\imath}_\bS(\bar{\fa})\end{bmatrix}
   \begin{bmatrix}T\\&\nu'\,\ltrans{T}^{-1}\end{bmatrix}
   \begin{bmatrix}&-\bid_2\\ \bid_2\end{bmatrix},\\
   \begin{bmatrix}0&\nu \bar{\fa}^{-1}\\ -\fa&0\end{bmatrix}
   &=-\fa
   \begin{bmatrix}&\nu'\\-1&\end{bmatrix}
   \begin{bmatrix}&1\\1\end{bmatrix}.
\end{align*}
We can replace $\int_{T'_\bS(\bQ_v)\backslash (\GL(2,\bQ_v)\times\cK^\times_v)}\int_{\bQ^\times_v}\cdots \,d^\times\nu \,d^\times C\,d^\times\fa$ by $\int_{\bQ^\times_v}\int_{\GL(2,\bQ_v)}\cdots\,d^\times  T\,d^\times\nu'$, and \eqref{eq:I1m-1} becomes
\begin{align*}
   I_{1,n}=&\,C\cdot\int_{\bQ^\times_v}\int_{\GL(2,\bQ_v)}\int_{\bQ_v}
   |\det T|^{2s}_v |\nu'|^{-s+\frac{1}{2}}_v |r|^{-2s-3}_v
   \chi_v\left(\nu^{\prime-1}(\det T)^2 r^{-2}\right)
   \Xi_v\left(r^{-1}\det T\right)\\
   &\times\Schw_{v,1,\alt}\left(\nu^{\prime-1}r^{-1}\det T\right)
   \cdot \Schw_{v,0}\left(r^{-1} T\right)\\
   &\times \psi_v(-\bbc r) \cdot \Bes^{\Pi_v}_{\bS,\Lambda_v}(\varphi_{m_1+n,m_2+n})\begin{pmatrix}&-T\\ \nu^{\prime}\,\ltrans{T}^{-1} \end{pmatrix} \cdot \Whi^{\pi_v}_\bbc(f_{v,m_3+n})\begin{pmatrix}&\nu'\\-1\end{pmatrix}\\
   &\times \mathds{1}_{\Sym_2(\varpi^{-2n+e}_v\bZ_v)}\left(r^{-1}\nu^{\prime-1} \ltrans{T}\begin{bmatrix}\alphaS\alphabS&\frac{\alphaS+\alphabS}{2}\\\frac{\alphaS+\alphabS}{2}&1\end{bmatrix} T\right)
   \cdot \mathds{1}_{\varpi^{-2n+e}_v\bZ_v} \left(r^{-1}\nu'\right)\,dr\,d^\times T\,d^\times\nu'.
\end{align*} 
Apply change of variables $T'=r^{-1}T$, $\nu''=\nu^{\prime-1}r^{-1}\det T$. We get
\begin{align*}
   I_{1,n}=&\,C\cdot \int_{\bQ^\times_v}\int_{\GL(2,\bQ_v)}\int_{\bQ_v}
   |\det T'|^{s+\frac{1}{2}}_v |\nu''|^{s-\frac{1}{2}}_v |r|^{s-\frac{5}{2}}_v
   \chi_v\left(\nu''r\det T'\right)
   \Xi_v\left(\det T'\right)\,\\
   &\times \Schw_{v,1,\alt}\left(\nu''\right)
   \, \Schw_{v,0}\left( T'\right)\cdot \psi_v(-\bbc r)\cdot \Bes^{\Pi_v}_{\bS,\Lambda_v}(\varphi_{v,m_1+n,m_2+n})\begin{pmatrix}&-T'\\ r^{-1}\nu^{\prime\prime-1}\,\ltrans{T}^{\prime,\mr{adj}} \end{pmatrix} \\
   &\times   
    \Whi^{\pi_v}_\bbc(f_{v,m_3+n})\begin{pmatrix}&\nu^{\prime\prime-1}\det T'\\-r^{-1}\end{pmatrix}\\
   &\times \mathds{1}_{\Sym_2(\varpi^{-2n+e}_v\bZ_v)}\left(\nu^{\prime\prime}(\det T')^{-1} \ltrans{T}^{\prime}\begin{bmatrix}\alphaS\alphabS&\frac{\alphaS+\alphabS}{2}\\\frac{\alphaS+\alphabS}{2}&1\end{bmatrix} T'\right)\cdot \mathds{1}_{\varpi^{-2n+e}_v\bZ_v} \left(\nu^{\prime\prime-1}\det T'\right)\,dr\,dT'\,d^\times\nu^{\prime\prime}.
\end{align*} 
Let 
\begin{align*}
   I'_{1,n}=&\,C\cdot \int_{\bQ^\times_v}\int_{\GL(2,\bQ_v)}\int_{\bQ_v}
   |\det T'|^{s+\frac{1}{2}}_v |\nu''|^{s-\frac{1}{2}}_v |r|^{s-\frac{5}{2}}_v
   \chi_v\left(\nu''r\det T'\right)
   \Xi_v\left(\det T'\right)\\
   &\times\Schw_{v,1,\alt}\left(\nu''\right)
   \, \Schw_{v,0}\left( T'\right)\cdot \psi_v(-\bbc r)\cdot\Bes^{\Pi_v}_{\bS,\Lambda_v}(\varphi_{v,m_1+n,m_2+n})\begin{pmatrix}&-T'\\ r^{-1}\nu^{\prime\prime-1}\,\ltrans{T}^{\prime,\mr{adj}} \end{pmatrix}  \\
   &\times  \Whi^{\pi_v}_\bbc(f_{v,m_3+n})\begin{pmatrix}&\nu^{\prime\prime-1}\det T'\\-r^{-1}\end{pmatrix} \,dr\,dT'\,d^\times\nu^{\prime\prime}
\end{align*} 
When $\Schw_{v,1,\alt}\left(\nu''\right)\, \Schw_{v,0}\left( T'\right)\neq 0$, the product of the two characteristic functions in the last row of the above integral for $I_{1,n}$ is nonzero only if
\begin{align*}
   &\nu^{\prime\prime}(\det T')^{-1}\notin \varpi^{-2n+3e+B_2}\bZ_v
   &&\text{or}
   &&  \nu^{\prime\prime-1}\det T'\notin \varpi_v^{-2n+e}\bZ_v,
\end{align*}
which implies 
\begin{align*}
     &\det T'\in\varpi_v^{2n-4e-B_2+1}\bZ_v
     &&\text{or}
     &&\nu^{\prime\prime}\in \varpi_v^{2n-3e+1}.
\end{align*}
Hence, when $\mr{Re}(s)$ is sufficiently large, $|I_{1,n}-I'_{1,n}|=O\left(|\varpi_v|^{ns}_v\right)$ as $n\ra\infty$. We obtain
\begin{align*}
    &\lim_{n\to\infty} \big(\lambda_v(\varphi_v)\lambda_v( f_v)\big)^{-n} I_{1,n}=\lim_{n\to\infty} \big(\lambda(\varphi_v)\lambda_v( f_v)\big)^{-n} I'_{1,n}
   =\text{RHS of }\eqref{eq:Zp-first reduction}.
\end{align*} 
\end{proof}

\subsection{Local zeta integrals for big-cell sections III}
Denote by $Q_{\GSp(4)}$ the standard Siegel parabolic subgroup of $\GSp(4)$ consisting of $\begin{bmatrix}\delta\, \ltrans{D}^{\mr{adj}}&B\\0&D \end{bmatrix}$ with $D\in\GL(2)$, $\delta\in\GL(1)$ and $BD^{-1}\in\Sym_2$, and denote by $B_{\GL(2)}$ the standard Borel subgroup of $\GL(2)$ consisting of upper triangular matrices. Next, we apply Proposition~\ref{Prop:Zbc-1} to the special case:
\begin{align*}
   \Pi_v &=\text{a subquotient of } \Ind_{Q_{\GSp(4)}}^{\GSp(4)}\sigma_{\sPi_v}\rtimes \eta_{\sPi_v,3},\\
   \pi_v &= \text{a subquotient of }\Ind^{\GL(2)}_{B_{\GL(2)}}\eta_{\pi_v,1}\boxtimes\eta_{\pi_v,2},
\end{align*} 
and obtain a formula which will be used to compute the local zeta integrals at $p$ for our $p$-adic interpolation. Here $\sigma_{\sPi_v}$ is an irreducible admissible representation of $\GL(2,\bQ_v)$, $\eta_{\sPi_v,3},\eta_{\pi_v,1},\eta_{\pi_v,2}$ are continuous characters of $\bQ^\times_v$, and $\Ind_{Q_{\GSp(4)}}^{\GSp(4)}\sigma_{\sPi_v}\rtimes \eta_{\sPi_v,3}$ (resp. $\Ind^{\GL(2)}_{B_{\GL(2)}}\eta_{\pi_v,1}\boxtimes\eta_{\pi_v,2}$)  denotes the (normalized) induced representation of $\GSp(4,\bQ_v)$ (resp. $\GL(2,\bQ_v)$) associated to
\begin{equation}\label{eq:induction}
   \begin{bmatrix}\delta\, \ltrans{D}^{\mr{adj}}&B\\0&D \end{bmatrix}\longmapsto\eta_{\sPi_v,3}(\delta)\,\sigma_{\sPi_v}(D),
   \qquad\qquad \begin{bmatrix}a&b\\0&d\end{bmatrix}\longmapsto \eta_{\pi_v,1}(a)\,\eta_{\pi_v,2}(d).
\end{equation}

Given $w_{v}\in\sigma_{\sPi_v}$, we define $\phi_{\GSp(4),v}(-;w_v)\in \Ind_{Q_{\GSp(4)}}^{\GSp(4)}\sigma_{\sPi_v}\rtimes\eta_{\sPi_v,3}$ as
\begin{equation}\label{eq:bc-phi}
   \phi_{\GSp(4),v}\left(g=\begin{bmatrix}A&B\\C&D\end{bmatrix};w_{\sigma_v}\right)
   =\mathds{1}_{\Sym_2(\bZ_v)}(C^{-1}D) 
\cdot \eta_{\sPi_v,3}\left(\nu_g(\det C)^{-1}\right)|\nu_g(\det C)^{-1}|^{3/2}_v\cdot \sigma_{\sPi_v}(C)w_{v},
\end{equation}
and let $V(w_v)\subset \Ind_{Q_{\GSp(4)}}^{\GSp(4)}\sigma_{\sPi_v}\rtimes \eta_{\sPi_v,3}$ be the subrepresentation generated by $\phi_{\GSp(4),v}(-;w_v)$. When $\sigma_{\sPi_v}$ is isomorphic to a quotient of $\Ind^{\GL(2)}_{B_{\GL(2)}}\eta_{\sPi_v,2}\boxtimes\eta_{\sPi_v,1}$, we define $\phi_{\GL(2),v}\in\Ind^{\GL(2)}_{B_{\GL(2)}}\eta_{\sPi_v,2}\boxtimes\eta_{\sPi_v,1}$ as
\begin{equation}\label{eq:phiGL2}
    \phi_{\GL(2),v}\left(h=\begin{bmatrix}a&b\\c&d\end{bmatrix}\right)= \mathds{1}_{\bZ_v}(c^{-1}d)\cdot \eta_{\sPi_v,1}(c)\,\eta_{\sPi_v,2}((\det h) c^{-1})\,|(\det h) c^{-2}|^{1/2}_v.
\end{equation}
It is easy to see that the above defined $\phi_{\GSp(4),v}(-;w_v)$  is an eigenvector for $U^{\GSp(4)}_{v,m,m}$  for all $m\geq 0$ (defined in \eqref{eq:loc-Uv}) with eigenvalue
\begin{equation}\label{eq:bc-lambda}
		\left(\omega_{\sPi_v}(\varpi_v)\eta^{-2}_{\sPi_v,3}(\varpi_v)|\varpi_v|^3_v\right)^m,
\end{equation}
and when $w_v$ is the projection of $\phi_{\GL(2),v}$, the corresponding $\phi_{\GSp(4),v}(-;w_v)$  is an eigenvector for $U^{\GSp(4)}_{v,m_1,m_2}$ for all $m_\geq m_2\geq 0$ with eigenvalue
\[
    \left(\eta_{\sPi_v,1}(\varpi_v)\eta^{-1}_{\sPi_v,3} |\varpi_v|^2_v\right)^{m_1} 
     \left(\eta_{\sPi_v,2}(\varpi_v)\eta^{-1}_{\sPi_v,3} |\varpi_v|_v\right)^{m_2}.
\]

\begin{prop}\label{prop:Zord}
Suppose that $\Pi_v$ is isomorphic to a quotient of $V(w_v)\subset \Ind_{Q_{\GSp(4)}}^{\GSp(4)}\sigma_{\sPi_v}\rtimes \eta_{\sPi_v,3}$ with $\sigma_v$ infinite dimensional, and is the only subquotient of $\Ind_{Q_{\GSp(4)}}^{\GSp(4)}\sigma_{\sPi_v}\rtimes \eta_{\sPi_v,3}$  having a nonzero $(\bS,\Lambda_v)$-Bessel functional.
\begin{enumerate}[leftmargin=2em,label=(\roman*)]
\item If $\varphi_v\in\Pi_v$ is the projection of $\phi_{\GSp(4),v}(-;w_v)\in V(w_v)$, then there exists a linear functional $\wal^{\sigma_{\sPi_v}}_{\bS,\Lambda_v}\in \Hom_{T_\bS(\bQ_v)}(\sigma_{\sPi_v},\Lambda_v)$ such that for all $A\in\GL(2,\bQ_v)$,
\[
\Bes^{\Pi_v}_{\bS,\Lambda_v}\left(\varphi_v\right)\begin{pmatrix}A\\&\delta\,\ltrans{A}^{\mr{adj}}\end{pmatrix}
=\eta_{\sPi_v,3}\left(\delta\right)|\delta|^{-3/2}_v\cdot \mathds{1}_{\Sym_2(\bZ_v)}\left((\delta\det A)^{-1}\cdot \ltrans{A}\bS A\right)\cdot\wal^{\sigma_{\sPi_v}}_{\bS,\Lambda_v}\big(\sigma_v(A)w_v).
\]

\item In addition to the assumptions in (i), if we further assume that $\Lambda_v$ is unitary, $\sigma_{\sPi_v}$ is unitarizable and isomorphic to a quotient of $\Ind^{\GL(2)}_{B_{\GL(2)}}\eta_{\sPi_v,2}\boxtimes\eta_{\sPi_v,1}$ and $w_v$ is the projection of $\phi_{\GL(2),v}$, then when $m_1\gg m_2\gg 0$,
\begin{equation}\label{eq:Bm1m2}
	\left(\eta_{\sPi_v,1}(\varpi_v)\eta^{-1}_{\sPi_v,3} |\varpi_v|^2_v\right)^{-m_1} 
	\left(\eta_{\sPi_v,2}(\varpi_v)\eta^{-1}_{\sPi_v,3} |\varpi_v|_v\right)^{-m_2}\Bes^{\Pi_v}_{\bS,\Lambda_v}\left(\varphi_v\right)\begin{psm}\varpi^{m_1}_v\\&\varpi^{m_2}_2\\&&\varpi^{-m_1}_v\\&&&\varpi^{-m_2}_v\end{psm}
\end{equation}
is independent of $m_1,m_2$, and this number is nonzero unless $v=\fv\bar{\fv}$ splits in $\cK=\bQ(\sqrt{-\det\bS})$, $\Lambda_\fv=\Lambda_{\bar{\fv}}$ and $\eta_{\sPi_v,1}=\Lambda_\fv|\cdot|^{1/2}_v, \eta_{\sPi_v,2}=\Lambda_{\bar{\fv}}|\cdot|^{-1/2}_v$.
\end{enumerate}
\end{prop}
\begin{proof}
(i) First, we show that our assumptions imply
\begin{equation}\label{eq:Wa-neq-0}
	 \Hom_{T_\bS}\left(\sigma_{\sPi_v},\Lambda_v\right)\neq 0.
\end{equation}
By \cite[Lemma 8]{Walds85} and \cite[Propositions~1.6,1.7]{Tunnell-ep}, $\Hom_{T_\bS}\left(\sigma_{\sPi_v},\Lambda_v\right)= 0$ implies that $v$ does not split in $\cK$, $\sigma_v\cong \mr{St}\otimes\eta_v\circ\det$ and $\Lambda_v=\eta_v\circ \Nm$ for some character $\eta_v$ of $\bQ^\times_v$ or $\sigma_v$ is supercuspidal. By checking the table in \cite[Theorem~6.2.2]{RSBessel}, when this happens, (under our assumptions, $\Pi_v$ can only be IIa,IVa,Va,VIa,X,XIa), $\Pi_v$ does not have nonzero $(\bS,\Lambda)$-Bessel functional. 

For every $\phi_v\in \Ind_{Q_{\GSp(4)}}^{\GSp(4)}\sigma_{\sPi_v}\rtimes \eta_{\sPi_v,3}$, there exists an open compact subgroup $\Omega\subset \Sym_2(\bQ_v)$ such that the function
\[
g\longmapsto \int_{\Omega} \phi_v\left(g\begin{bmatrix}\bid_2&X\\&\bid_2\end{bmatrix}\right)\cdot \psi_v(\Tr\bS X)\,dX
\]
is supported on $Q_{\GSp(4)}(\bQ_v)\begin{bmatrix}&-\bid_2\\\bid_2\end{bmatrix}Q_{\GSp(4)}(\bQ_v)$ \cite[Lemma 3.2]{LYF-RGGP}. It follows that there exists an open compact subgroup $\Omega'\subset \Sym_2(\bQ_v)$ such that for all open compact subgroups $\Omega^{\prime\prime}$ containing $\Omega'$, 
\[
\int_{\Omega'}\phi_v\begin{pmatrix}&-\bid_2\\ \bid_2&X\end{pmatrix}\cdot \psi_v(\Tr\bS X)\,dX
=\int_{\Omega^{\prime\prime}}\phi_v\begin{pmatrix}&-\bid_2\\ \bid_2&X\end{pmatrix}\cdot \psi_v(\Tr\bS X)\,dX.
\]
We denote the stable value by
\[
\int^{\mr{st}}_{\Sym_2(\bQ_v)} \phi_v\begin{pmatrix}&-\bid_2\\ \bid_2&X\end{pmatrix}\cdot \psi_v(\Tr\bS X)\,dX.
\]
Take nonzero $\wal^{\sigma_{\sPi_v}}_{\bS,\Lambda_v}\in \Hom_{T_\bS}\left(\sigma_{\sPi_v},\Lambda_v\right)$. Then
\begin{equation}\label{eq:int-st}
	\phi_v\longmapsto \wal^{\sigma_{\sPi_v}}_{\bS,\Lambda_v} \left(\int^{\mr{st}}_{\Sym_2(\bQ_v)} \phi_v\begin{pmatrix}&-\bid_2\\ \bid_2&X\end{pmatrix}\cdot \psi_v(\Tr\bS X)\,dX\right)
\end{equation}
gives a nonzero element in $\Hom_{R_\bS(\bQ_v)}\left(V(w_{\sigma_v}),\Lambda_v\right)$. Since we assume that any subquotient of $V(w_{\sigma_v})$ other than $\Pi_v$ has no nonzero $(\bS,\Lambda_v)$-Bessel functional, \eqref{eq:int-st} factors through $\Pi_v$ and gives a nonzero Bessel functional in $\Hom_{R_\bS(\bQ_v)}\left(\Pi_v,\Lambda_v\right)$. Thanks to the uniqueness of Bessel models, we can rescale $\wal^{\sigma_{\sPi_v}}_{\bS,\Lambda_v}$ such that \eqref{eq:int-st} gives rise to our fixed $\bes^{\Pi_v}_{\bS,\Lambda_v}$. Then
\begin{align*}
	&\Bes^{\Pi_v}_{\bS,\Lambda_v}(\varphi_v)\begin{pmatrix}A\\&\delta\,\ltrans{A}^{\mr{adj}}\end{pmatrix}\\
	=&\,\wal^{\sigma_{\sPi_v}}_{\bS,\Lambda_v} \left(\int^{\mr{st}}_{\Sym_2(\bQ_v)} \phi\left(\begin{bmatrix}&-\bid_2\\ \bid_2&X\end{bmatrix}\begin{bmatrix}A\\&\delta\,\ltrans{A}^{\mr{adj}}\end{bmatrix};w_v\right)\cdot \psi_v(\Tr\bS X)\,dX\right)\\
	=&\, \eta_{\sPi_v,3}\left(\delta\right)|\delta|^{-3/2}_v\cdot \mathds{1}_{\Sym_2(\bZ_v)}\left((\delta\det A)^{-1}\cdot \ltrans{A}\bS A\right)\cdot \wal^{\sigma_{\sPi_v}}_{\bS,\Lambda_v}\left(\sigma_v\left(A\right)w_v\right).
\end{align*}

(ii) We only need to show that 
\begin{equation}\label{eq:wal-ns} 
   \left(\eta_{\sPi_v,1}(\varpi_v) |\varpi_v|^{1/2}_v\right)^{-m_1} 
\left(\eta_{\sPi_v,2}(\varpi_v) |\varpi_v|^{-1/2}_v\right)^{-m_2}
  \wal^{\sigma_v}_{\bS,\Lambda_v}\left(\sigma_v\begin{pmatrix}\varpi^{m_1}_v\\&\varpi^{m_2}_v\end{pmatrix} w_v\right)
\end{equation}
is independent of $m_1,m_2$ as long as $m_1\gg m_2\gg 0$, and is nonzero except the situation described in the proposition. To do the computation, we separate into two cases.\\

\underline{$v$ does not split in $\cK=\bQ(\sqrt{-\det\bS})$.} In this case, we have $\cK^\times_v=\bQ^\times_v U_v$ with compact $U_v=\left\{\begin{array}{ll}\cO^\times_{\cK,v},&\text{$v$ inert,}\\ \cO^\times_{\cK,v}\cup\varpi_v\cO^\times_{\cK,v},&\text{$v$ ramified.}\end{array}\right.$  We can define an element $\wal_0\in \Hom_{T_{\bS}}\left(\Ind^{\GL(2)}_{B_{\GL(2)}}\eta_{\sPi_v,2}\boxtimes\eta_{\sPi_v,1},\Lambda_v\right)$ by 
\[
\wal_0(\phi) = \int_{U_v}\phi(\imath_{\bS}(\fz))\Lambda^{-1}_v(\fz)\,d\fz.
\]
If the induction $\Ind^{\GL(2)}_{B_{\GL(2)}}\eta_{\sPi_v,2}\boxtimes\eta_{\sPi_v,1}$ is reducible, then $\sigma_v$ is isomorphic to the Steinberg representation twisted by a character $\eta_v$. \eqref{eq:Wa-neq-0} implies $\Lambda_v\neq \eta_v\circ\Nm$ (\cite[Proposition~1.7]{Tunnell-ep}), so $I_0$ factors through $\sigma_v$.

It suffices to show that  $\wal_0\left(\begin{bmatrix}\varpi^{m_1}_v\\&\varpi^{m_2}_v\end{bmatrix}\phi_{\GL(2),v}\right)$ equals $\eta_{\sPi_v,2}(\varpi_v)^{m_1}\eta_{\sPi_v,1}(\varpi_v)^{m_2}|\varpi_v|^{\frac{m_1-m_2}{2}}_v $ multiplied by a nonzero number that does not depend on $m_1,m_2$ when $m_1\gg m_2\gg 0$. With $\fz\in\cO_{\cK,v}$ written as $\fz=x+y\alphaS$, $x,y\in\bQ_v$, we have
\[
\imath_{\bS}(\fz)=\begin{bmatrix}x+y(\alphaS+\alphabS)&y\\-y\alphaS\alphabS&x\end{bmatrix},
\] 
and when $m_1\gg m_2$, by the definition of $\phi_{\GL(2),v}$ in \eqref{eq:phiGL2}, we have
{\small
\begin{equation}\label{eq:phiz}
\begin{aligned}
	\left(\begin{bmatrix}\varpi^{m_1}_v\\&\varpi^{m_2}_v\end{bmatrix}\right)\phi_{\GL(2),v})(\imath_{\bS}(\fz))
	=&\,\eta_{\sPi_v,1}(\varpi_v)^{m_1}\eta_{\sPi_v,2}(\varpi_v)^{m_2}|\varpi_v|^{-\frac{m_1-m_2}{2}}_v
	\cdot\mathds{1}_{\varpi^{m_1-m_2}_v\alphaS\alphabS\bZ_p}(y^{-1}x)\\
	&\times \eta_{\sPi_v,2}\left(\frac{(x+y\alphaS)(x+y\alphabS)}{y\alphaS\alphabS}\right)\,\eta_{\sPi_v,1}(y\alphaS\alphabS)
	\cdot \left|\frac{(x+y\alphaS)(x+y\alphabS)}{(y\alphaS\alphabS)^2}\right|^{1/2}_v.
\end{aligned}
\end{equation}
}
With $m_1\gg m_2$, when the right hand side does not vanish, $y^{-1}x$ is very close to $0$, so
\begin{align*}
	\left(\begin{bmatrix}\varpi^{m_1}_v\\&\varpi^{m_2}_v\end{bmatrix}\right)\phi_{\GL(2),v})(\imath_{\bS}(\fz))
	&=\eta_{\sPi_v,2}(\varpi_v)^{m_1}\eta_{\sPi_v,1}(\varpi_v)^{m_2}|\varpi|^{-\frac{m_1-m_2}{2}}_v
	\eta_{\sPi_v,1}(\alphaS\alphabS)\cdot \eta_{\sPi_v,1}\eta_{\sPi_v,2}(y)\\
	\Lambda_v(\fz)&=\Lambda_v(y\alphaS)=\Lambda_v(\alphaS)\cdot\eta_{\sPi_v,1}\eta_{\sPi_v,2}(y).
\end{align*}
Plugging this into the integral, we get
\begin{align*}
    \wal_0\left(\begin{bmatrix}\varpi^{m_1}_v\\&\varpi^{m_2}_v\end{bmatrix}\phi_{\GL(2),v}\right)=&\,\eta_{\sPi_v,2}(\varpi_v)^{m_1}\eta_{\sPi_v,1}(\varpi_v)^{m_2}|\varpi_v|^{\frac{m_1-m_2}{2}}_v \\
    &\times \eta_{\sPi_v,1}(\alphaS\alphabS)\Lambda^{-1}_v(\alphaS) \frac{\vol((1+\varpi^{m_1-m_2}_v\alphabS\bZ_v)\cap U_v)}{|\varpi_v|^{\frac{m_1-m_2}{2}}}.
\end{align*}
The factor in the second row is nonzero and does not depend on $m_1,m_2$ as long as $m_1\gg m_2\gg 0$.

\vspace{.5em}
\underline{$v=\fv\bar{\fv}$ splits in $\cK=\bQ(\sqrt{-\det\bS})$.} Denote by $\varrho_\fv$ the projection $\cK_v\ra \cK_\fv\cong\bQ_v$.  Let $\Whi^{\sigma_v}$ be a  Whittaker functional of $\sigma_v$ with respect to $\psi_v$. Then up to rescale,
\begin{align*}
	&\cI^{\sigma_v}_{\bS,\Lambda_v}\left(\begin{bmatrix}\varpi^{m_1}_v\\&\varpi^{m_2}_v\end{bmatrix}w_v\right)\\
	=&\,\int_{\bQ^\times_v} \Whi^{\sigma_v} \left(\begin{bmatrix}t\\&1\end{bmatrix}\begin{bmatrix}1&1\\ -\varrho_\fv(\alphabS)&-\varrho_\fv(\alphaS)\end{bmatrix}^{-1}\begin{bmatrix}\varpi^{m_1}_v\\&\varpi^{m_2}_v\end{bmatrix}w_v\right)
	\cdot \Lambda^{-1}_{\fv}(t)\,d^\times t.
\end{align*}
(Under our assumption that $\Lambda_v$ is unitary and $\sigma_v$ is unitarizable, the integral converges absolutely.) With $m_1\gg m_2$, 
\begin{align*}
	&\Whi^{\sigma_v} \left(\begin{bmatrix}t\\&1\end{bmatrix}\begin{bmatrix}1&1\\ -\varrho_\fv(\alphabS)&-\varrho_\fv(\alphaS)\end{bmatrix}^{-1}\begin{bmatrix}\varpi^{m_1}_v\\&\varpi^{m_2}_v\end{bmatrix}w_v\right)\\
	=&\, \Whi^{\sigma_v} \left(\begin{bmatrix}t\\&1\end{bmatrix}\varpi^{m_2}_v\varrho_\fv(\alphabS-\alphaS)^{-1}\begin{bmatrix}\varrho_\fv(\alphabS-\alphaS)&1\\ 0&1\end{bmatrix}\begin{bmatrix}\varpi^{m_1-m_2}_v\\&1\end{bmatrix}w_v\right)\\
	=&\,\eta_{\sPi_v,1}\eta_{\sPi_v,2}\left(\varpi^{m_2}_v\varrho_\fv(\alphabS-\alphaS)^{-1}\right)\cdot \psi_v(t)\cdot \Whi^{\sigma_v} \left(\begin{bmatrix}\varrho_\fv(\alphabS-\alphaS)t\varpi^{m_1-m_2}_v&\\ &1\end{bmatrix}w_v\right).
\end{align*}
Since when $\Ind^{\GL(2,\bQ_v)}_{B_{\GL(2)}(\bQ_v)}(\eta_{\sPi_v,2},\eta_{\sPi_v,1})$ is reducible, its one dimensional subrepresentation has no Whittaker model, we can compute $\Whi^{\sigma_v}\left(\begin{bmatrix}\varpi^a_v\\&1\end{bmatrix}w_{\sigma_p}\right)$ by 
\begin{align*}
	\int_{\bQ_v}\phi_0\left(\begin{bmatrix}&-1\\1&x\end{bmatrix}\begin{bmatrix}a\\&1\end{bmatrix}\right) \psi_v(-x)\,dx
	&=\int_{\bQ_v} \mathds{1}_{\bZ_v}(a^{-1}x)\cdot \eta_{\sPi_v,1}(a)|a|^{-1/2}_v\,\psi_v(-x)\,dx\\
	&=\mathds{1}_{\bZ_v}(a)\cdot \eta_{\sPi_v,1}(a)|a|^{1/2}_v.
\end{align*}
Plug it into the above equation. We get 
\begin{align*}
	&\cI^{\sigma_v}_{\bS,\Lambda_v}\left(\begin{bmatrix}\varpi^{m_1}_v\\&\varpi^{m_2}_v\end{bmatrix}w_{\sigma_v}\right)\\
	= &\,\eta_{\sPi_v,1}\eta_{\sPi_v,2}\left(\varpi^{m_2}_v\varrho_\fv(\alphabS-\alphaS)^{-1}\right)
	\cdot \eta_{\sPi_v,1}(\varrho_\fv(\alphabS-\alphaS)) \left(\eta_{\sPi_v,1}(\varpi_v)|\varpi_v|^{1/2}_v\right)^{m_1-m_2}\\
	&\times\int_{\bQ^\times_v}\mathds{1}_{\bZ_v}\left(\varrho_\fv(\alphaS-\alphabS)\varpi^{m_1-m_2}_vt\right)  \cdot \psi_v(t) \eta_{\sPi_v,1}(t)|t|^{1/2}_v \Lambda^{-1}_{\fv}(t)\,d^\times t
\end{align*}
Since $m_1\gg m_2$, the integral over $\bQ^\times_v$ equals
\begin{align*}
	\int_{\bQ^\times_v}  \psi_v(t) \eta_{\sPi_v,1}(t)|t|^{1/2}_v \Lambda^{-1}_{\fv}(t)\,d^\times t
	=\gamma_v\left(\frac{1}{2},\eta^{-1}_{\sPi_v,1}\Lambda_\fv\right),
\end{align*}
and
\begin{equation}\label{eq:dagger1}
	\begin{aligned}
		&\left(\eta_{\sPi_v,2}(\varpi_v)^{m_1}\eta_{\sPi_v,1}(\varpi_v)^{m_2}|\varpi_v|^{\frac{m_1-m_2}{2}}_v\right)^{-1} \cI^{\sigma_v}_{\bS,\Lambda_v}\left(\begin{bmatrix}\varpi^{m_1}_v\\&\varpi^{m_2}_v\end{bmatrix}w_v\right)\\
		=&\, \eta^{-1}_{\sPi_v,2}(\varrho_\fv(\alphabS-\alphaS))\cdot \gamma_v\left(\frac{1}{2},\eta^{-1}_{\sPi_v,1}\Lambda_\fv\right),
	\end{aligned}
\end{equation}
independent of $m_1,m_2$. If \eqref{eq:Bm1m2} is zero, the computation shows that $\gamma_v\left(\frac{1}{2},\eta^{-1}_{\sPi_v,1}\Lambda_\fv\right)=0$, so  $\eta_{\sPi_v,1}=\Lambda_\fv|\cdot|^{1/2}_v$. By interchanging $\fv$ and $\bar{\fv}$, we also know that $\eta_{\sPi_v,1}=\Lambda_{\bar{\fv}}|\cdot|^{1/2}_v$. Hence, we have $\Lambda_\fv=\Lambda_{\bar{\fv}}$, and also $\eta_{\sPi_v,2}=\Lambda_\fv\Lambda_{\bar{\fv}}\eta^{-1}_{\sPi_v,2}=\Lambda_{\fv}|\cdot|^{-1/2}_v$.
\end{proof}

\begin{prop}\label{prop:ZpWals}
Assume the assumptions in (i) of Proposition~\ref{prop:Zord} and also suppose that $\pi_v$ is infinite dimensional and is isomorphic to a quotient of $\Ind^{\GL(2)}_{B_{\GL(2)}}\eta_{\pi_v,1}\boxtimes\eta_{\pi_v,2}$ with $f_v\in\pi_v$ the projection of 
\begin{equation}
	h=\begin{bmatrix}a&b\\c&d\end{bmatrix}\longmapsto \mathds{1}_{\bZ_v}(c^{-1}d)\cdot \eta_{\pi_v,1}(\nu_h c^{-1})\,\eta_{\pi_v,2}(c)\,|\nu_h c^{-2}|^{1/2}_v.
\end{equation}
Then
\begin{align*}
   &Z_v\left(\dsec^\bc_{v,\Schw_v}(s,\chi,\Xi),\Bes^{\Pi_v}_{\bS,\Lambda_{v}}(\varphi_{v,,m_1,m_2}),\Whi^{\pi_v,\Upsilon_v}_\bbc(f_{v,m_3})\right)\\
   =&\, C\cdot\chi_v(-1)\,\chi^{-1}_v\eta_{\pi_v,1}\eta_{\sPi_v,3}(\bbc)\,|\bbc|^{-s}_v \cdot \left(\eta^{-1}_{\sPi_v,3}(\varpi_v)|\varpi_v|^{3/2}_v\right)^{m_1+m_2} \lambda(f_v)^{m_3} \cdot\whi^{\pi_v}_\bbc(f_v)\\
   &\times \gamma_v\left(s+\frac{1}{2},\sigma^\vee_{\sPi_v}\times \chi_v\eta^{-1}_{\pi_v,1}\right)^{-1}
   \gamma_v\left(s+\frac{1}{2},\pi^\vee_v\times \chi_v\eta^{-1}_{\sPi_v,3}\right)^{-1}\\
   &\times\int_{\bQ^\times_v} |\nu|^{-s+\frac{1}{2}}_v\cdot \chi^{-1}_v\eta_{\pi_v,2}\eta_{\sPi_v,3}(\nu)\cdot \cF \Schw_{v,1,\alt}\left(\nu \right) \,d^\times\nu\\
   &\times \int_{\GL(2,\bQ_v)}\hspace{-1em} (\cF\Schw_{v,0})\left(T\right)\cdot |\det T|^{-s+1}_v\, \chi^{-1}_v\eta_{\pi_v,1}\Lambda_{\bQ,v}(\det T)
   \cdot  \wal^{\sigma_{\sPi_v}}_{\bS,\Lambda_v}\left(\sigma_{\sPi_v}\left(T^{-1}\begin{bmatrix}\varpi^{m_1}_v\\&\varpi^{m_2}_v\end{bmatrix} \right)w_v\right) \,dT
\end{align*}
with the constant $C$ as in Proposition~\ref{Prop:Zbc-1}.

\end{prop}

\begin{proof}
Under our assumptions on $f_v$, a straightforward computation shows that
\begin{align*}
   \Whi^{\pi_v}_{\bbc}(f_v)\begin{pmatrix}a\\&d\end{pmatrix}=\mathds{1}_{\bZ_v}(\bbc ad^{-1})\cdot \eta_{\pi_v,1}(d)\,\eta_{\pi_v,2}(a)\, |ad^{-1}|^{1/2}_v \cdot \whi^{\pi_v}_{\bbc}(f_v).
\end{align*}
Plug this and the formula in (i) of Proposition~\ref{prop:Zord} for $\Bes^{\Pi_v}_{\bS,\Lambda_v}(\varphi_v)$  into \eqref{eq:Zp-first reduction}. We obtain
\begin{align*}
   &Z_v\left(\dsec^\bc_{v,\Schw_v}(s,\chi,\Xi),\Bes^{\Pi_v}_{\bS,\Lambda_{v}}(\varphi_{v,,m_1+n,m_2+n}),\Whi^{\pi_v,\Upsilon_v}_\bbc(f_{v,m_3+n})\right)\\   
   =&\,C\cdot \lim_{n\to\infty}\big(\lambda(\varphi_v)\lambda_v(f_{v})\big)^{-n}
   \int_{\bQ^\times_v}\int_{\GL(2,\bQ_v)}\int_{\bQ_v}
   |\det T|^{s+\frac{1}{2}}_v |\nu|^{s-\frac{1}{2}}_v |r|^{s-\frac{5}{2}}_v
   \chi_v\left(\nu r\det T\right)
   \Xi_v\left(\det T\right)\\
   &\times\Schw_{v,1,\alt}\left(\nu\right)
   \, \Schw_{v,0}\left( T\right)\cdot \psi_v(-\bbc r)
   \cdot \eta^{-1}_{\sPi_v,3}\left(-\varpi^{2n+m_1+m_2}_v r\nu\right)|\varpi^{2n+m_1+m_2}_v r\nu|^{3/2}_v\\
   &\times  \mathds{1}_{\Sym_2(\bZ_v)}\left(\varpi^{2n}_v r\nu(\det T)^{-1} \begin{bmatrix}\varpi^{m_1}_v\\&\varpi^{m_2}_v\end{bmatrix} \ltrans{T}\bS T\begin{bmatrix}\varpi^{m_1}_v\\&\varpi^{m_2}_v\end{bmatrix}\right)\\
    &\times  \wal^{\sigma_{\sPi_v}}_{\bS,\Lambda_v}\left(\sigma_{\sPi_v}\left(-T\begin{bmatrix}\varpi^{m_1+n}_v\\&\varpi_v^{m_2+n}\end{bmatrix}\right)w_v\right) \\
   &\times \whi^{\pi_v}_\bbc(f_v)\cdot \mathds{1}_{\bZ_v}\left(\varpi^{2n+2m_3}_v\bbc \nu^{-1}r\det T\right)\cdot |\varpi^{2n+2m_3}_v\bbc \nu^{-1}r\det T|^{1/2}_v\\
   &\times \eta^{-1}_{\pi_v,1}\left(-r\varpi^{n+m_3}_v\right)\eta_{\pi_v,2}\left(\varpi^{n+m_3}_v \nu^{-1}\det T\right)
   \,dr\,d^\times T\,d^\times\nu,
\end{align*}   
and by plugging in \eqref{eq:bc-lambda}, it becomes
\begin{equation}\label{eq:ZvII-In}
\begin{aligned}
    &Z_v\left(\dsec^\bc_{v,\Schw_v}(s,\chi,\Xi),\Bes^{\Pi_v}_{\bS,\Lambda_{v}}(\varphi_{v,m_1+n,m_2+n}),\Whi^{\pi_v,\Upsilon_v}_\bbc(f_{v,m_3+n})\right)\\   
   =&\,C\cdot \left(\eta^{-1}_{\sPi_v,3}(\varpi_v)|\varpi_v|^{3/2}_v\right)^{m_1+m_2}\lambda_v(f_v)^{m_3}
   \cdot \eta_{\sPi_v,3}\Lambda_{\bQ,v}\eta_{\pi_v,1}(-1)
   \cdot |\bbc|^{1/2}\,  \whi^{\pi_v}_\bbc(f_v)\cdot \lim_{n\to\infty} I_n
\end{aligned}
\end{equation}
with
\begin{align*}
   I_n=&\,\int_{\bQ^\times_v}\int_{\GL(2,\bQ_v)}\int_{\bQ_v}
   |\det T|^{s+1}_v |\nu|^{s+\frac{1}{2}}_v |r|^{s-\frac{1}{2}}_v\chi_v\left(\nu r\det T\right)
   \Xi_v\left(\det T\right)\\
   &\times \eta^{-1}_{\sPi_v,3}\left( r\nu\right)\,\eta^{-1}_{\pi_v,1}\left(r\right)\eta_{\pi_v,2}\left(\nu^{-1}\det T\right)
   \cdot \Schw_{v,1,\alt}\left(\nu\right)
   \, \Schw_{v,0}\left(T\right)\\
   &\times \psi_v(-\bbc r)
   \cdot \wal^{\sigma_{\sPi_v}}_{\bS,\Lambda_v}\left(\sigma_{\sPi_v}\left(T\begin{bmatrix}\varpi_v^{m_1}\\&\varpi_v^{m_2}\end{bmatrix}\right)w_v\right)\\
   &\times  
   \mathds{1}_{\Sym_2(\bZ_v)}\left(\varpi^{2n}_v r\nu(\det T)^{-1} \begin{bmatrix}\varpi^{m_1}_v\\&\varpi^{m_2}_v\end{bmatrix} \ltrans{T}\bS T\begin{bmatrix}\varpi^{m_1}_v\\&\varpi^{m_2}_v\end{bmatrix}\right)\\
   &\times  \mathds{1}_{\bZ_v}\left(\varpi^{2n+2m_3}_v\bbc \nu^{-1}r\det T\right)\,dr\,d^\times T\,d^\times\nu.
\end{align*}
Let
\begin{align*}
   I'_n=&\,\int_{\bQ^\times_v}\int_{\GL(2,\bQ_v)}\int_{\bQ_v}
   |\det T|^{s+1}_v |\nu|^{s+\frac{1}{2}}_v |r|^{s-\frac{1}{2}}_v \chi_v\left(\nu r\det T\right)
   \Xi_v\left(\det T\right)\\
   &\times \eta^{-1}_{\sPi_v,3}\left( r\nu\right)\,\eta^{-1}_{\pi_v,1}\left(r\right)\eta_{\pi_v,2}\left(\nu^{-1}\det T\right)
   \cdot \Schw_{v,1,\alt}\left(\nu\right)
   \, \Schw_{v,0}\left(T\right)\\
   &\times \psi_v(-\bbc r)
   \cdot \wal^{\sigma_{\sPi_v}}_{\bS,\Lambda_v}\left(\sigma_{\sPi_v}\left(T\begin{bmatrix}\varpi_v^{m_1}\\&\varpi_v^{m_2}\end{bmatrix}\right)w_v\right)
   \,dr\,d^\times T\,d^\times\nu.
\end{align*}
When $\mr{Re}(s)$ is sufficiently large, $|I_n-I'_n|=O\left(|\varpi_v|^{ns}_v\right)$ as $n\ra\infty$. Hence,
\begin{align*}
   &\lim_{n\to\infty}I_n=\lim_{n\to\infty}I'_n
   \\
   =&\,\int_{\GL(2,\bQ_v)}|\det T|^{s+1}_v\cdot \chi_v\eta^{-1}_{\pi_v,1}\Lambda^{-1}_{\bQ,v}(\det T)\cdot \Schw_{v,0}\left( T\right) 
   \cdot \wal^{\sigma_{\sPi_v}}_{\bS,\Lambda_v} \left(\sigma_{\sPi_v}\left(T\begin{bmatrix}\varpi^{m_1}_v\\&\varpi^{m_2}_v\end{bmatrix}\right)w_v\right)\,d^\times T\\
   &\times \int_{\bQ_v}\psi_v(-\bbc r)\cdot |r|^{s-\frac{1}{2}}_v\cdot \chi_v\eta^{-1}_{\pi_v,1}\eta^{-1}_{\sPi_v,3}(r)\,d r
   \int_{\bQ^\times_v}|\nu|^{s+\frac{1}{2}}_v \cdot \chi_v\eta^{-1}_{\sPi_v,3}\eta^{-1}_{\pi_v,2}(\nu)\cdot  \Schw_{v,1,\alt}\left(\nu\right)\,d^\times\nu
\end{align*}
For the integral involving $T$, by applying the functional equation by Godement--Jacquet \cite[Theorem 13.8]{GJDR} for $\GL(2)$, we get
\begin{align*}
   \text{the integral involving $T$}
   =&\,   \gamma_v\left(s+\frac{1}{2},\sigma_{\sPi_v}\times \chi_v\eta^{-1}_{\pi_v,1}\Lambda^{-1}_{\bQ,v}\right)^{-1}\\
   &\times\int_{\GL(2,\bQ_v)} (\cF\Schw_0)\left(T\right)\cdot |\det T|^{-s+1}_v\cdot \chi^{-1}_v\eta_{\pi_v,1}\Lambda_{\bQ,v}(\det T)\\
   &\times \wal^{\sigma_{\sPi_v}}_{\bS,\Lambda_v}\left(\sigma_{\sPi_v}\left(T^{-1}\begin{bmatrix}\varpi^{m_1}_v\\&\varpi^{m_2}_v\end{bmatrix} \right)w_v\right) \,d^\times T.
\end{align*} 

For the integral involving $r$, we can rewrite it as
\begin{align*}
   &(1-|\varpi_v|_v)\sum_{m\in\bZ} |\varpi^m_v|_v^{s+\frac{1}{2}} \chi_v\eta^{-1}_{\pi_v,1}\eta^{-1}_{\sPi_v,3}(\varpi^m_v)\int_{\bZ^\times_v} \psi_v(-\bbc \varpi^m_v r_0)\cdot \chi_v\eta^{-1}_{\pi_v,1}\eta^{-1}_{\sPi_v,3}(r_0)\,d^\times r_0\\
   =&\sum_{m\in\bZ} |\varpi^m_v|_v^{s+\frac{1}{2}} \chi_v\eta^{-1}_{\pi_v,1}\eta^{-1}_{\sPi_v,3}(\varpi^m_v) \cdot 
   \cF(\chi_v\eta^{-1}_{\pi_v,1}\eta^{-1}_{\sPi_v,3})^\circ (\bbc \varpi^m_v)\\
   =&\,|\bbc|^{-s-\frac{1}{2}}_v \chi^{-1}_v\eta_{\pi_v,1}\eta_{\sPi_v,3}(\bbc)\int_{\bQ^\times_v}  |r|^{s+\frac{1}{2}} \chi_v\eta^{-1}_{\pi_v,1}\eta^{-1}_{\sPi_v,3}(r)\cdot\cF(\chi_v\eta^{-1}_{\pi_v,1}\eta^{-1}_{\sPi_v,3})^\circ (r) \,d^\times r.
\end{align*}
(See \eqref{eq:char-circ} for notation $(\chi_v\eta^{-1}_{\pi_v,1}\eta^{-1}_{\sPi_v,3})^\circ$.) Then, applying the functional equation for $\GL(1)$, we have
\begin{align*}
   \text{the integral involving $r$}
   =&\,|\bbc|^{-s-\frac{1}{2}}_v \chi^{-1}_v\eta_{\pi_v,1}\eta_{\sPi_v,3}(\bbc)\cdot \gamma_v\left(s+\frac{1}{2},\chi_v\eta^{-1}_{\pi_v,1}\eta^{-1}_{\sPi_v,3}\right)^{-1}\\
   &\times\int_{\bQ^\times_v} |r|^{-s+\frac{1}{2}}_v\cdot \chi^{-1}_v\eta_{\pi_v,1}\eta_{\sPi_v,3}(r)\cdot  (\chi_v\eta^{-1}_{\pi_v,1}\eta^{-1}_{\sPi_v,3})^\circ (- r)\,d^\times r\\
   =&\,|\bbc|^{-s-\frac{1}{2}}_v \chi^{-1}_v\eta_{\pi_v,1}\eta_{\sPi_v,3}(-\bbc)\cdot \gamma_v\left(s+\frac{1}{2},\chi_v\eta^{-1}_{\pi_v,1}\eta^{-1}_{\sPi_v,3}\right)^{-1}.
\end{align*}
Also, we have
\begin{align*}
   \text{the integral involving $\nu$}
   =&\, \gamma_v\left(s+\frac{1}{2},\chi_v\eta^{-1}_{\pi_v,2}\eta^{-1}_{\sPi_v,3}\right)^{-1}
   \int_{\bQ^\times_v} |\nu|^{-s+\frac{1}{2}}_v\cdot \chi^{-1}_v\eta_{\pi_v,2}\eta_{\sPi_v,3}(\nu)\cdot \cF \Schw_{v,1,\alt}\left(\nu \right) \,d^\times\nu.
\end{align*}
Combining these formulas, we get a formula for $\lim\limits_{n\to\infty}I_n$. Then the proposition follows by plugging the formula for $\lim\limits_{n\to\infty}I_n$ into \eqref{eq:ZvII-In} and using that
\begin{align*}
    \gamma_v\left(s+\frac{1}{2},\sigma_v\times \chi_v\eta^{-1}_{1,v}\Lambda^{-1}_{\bQ,v}\right)^{-1}&=\gamma_v\left(s+\frac{1}{2},\sigma^\vee_v\times \chi_v\eta^{-1}_{1,v}\right)^{-1}\\
    \gamma_v\left(s+\frac{1}{2},\chi_v\eta^{-1}_{1,v}\eta^{-1}_{\sPi_v,3}\right)^{-1}\gamma_v\left(s+\frac{1}{2},\chi_v\eta^{-1}_{2,v}\eta^{-1}_{\sPi_v,3}\right)^{-1}&=\gamma_v\left(s+\frac{1}{2},\pi^\vee_v\times\chi_v\eta^{-1}_{\sPi_v,3}\right)^{-1}.
\end{align*}
\end{proof}

The following proposition shows that we can apply Proposition~\ref{prop:ZpWals} for the local zeta integrals at $p$ for $p$-ordinary $\Pi$ with $\Pi_p$ isomorphic to a holomorphic discrete series.

\begin{prop}
Let $\Pi$ be a $p$-ordinary irreducible cuspidal automorphic representation of $\GSp(4,\bA_\bQ)$ (see \ref{sec:ord-at-p} for the definition of being $p$-ordinary) with $\Pi_\infty\cong \cD_{l_1,l_2}$, $l_1\geq l_2\geq 3$. Let $\Lambda$ be a Hecke character of $\cK=\bQ(\sqrt{-\det\bS})$. If $\Pi_\infty$ has a nonzero $(\bS,\Lambda_\infty)$-Bessel functional, then $\Pi_p$ and $(\bS,\Lambda_p)$ satisfies the conditions in Proposition~\ref{prop:ZpWals}. 
\end{prop}
\begin{proof}
The $p$-ordinarity of $\Pi$ implies that $\Pi_p$ is isomorphic to a subquotient of $V(w_p)\subset\Ind_{Q_{\GSp(4)}}^{\GSp(4)}(\sigma_p, \eta_{\sPi_p,3})$ for some  $w_p\in\sigma_p$ and $\sigma_p$ isomorphic to a quotient of $\Ind^{\GL(2)}_{B_{\GL(2)}}(\eta_{\sPi_p,2},\eta_{\sPi_p,1})$ with characters $\eta_{\sPi_p,1},\eta_{\sPi_p,2}$ satisfying
\begin{align}
   \label{eq:chi-vp} \val_p\big(\eta_{\sPi_p,1}(p)\eta^{-1}_{\sPi_p,3}(p)\big)&=-l_1+1,
   &\val_p\big(\eta_{\sPi_p,2}(p)\eta^{-1}_{\sPi_p,3}(p)\big)&=-l_2+2.
\end{align}
In the notation of \cite{RSLocal}, $\Pi_p$ is isomorphic to a subquotient of $\eta_{\sPi_p,1}\eta^{-1}_{\sPi_p,3}\times\eta_{\sPi_p,2}\eta^{-1}_{\sPi_p,3}\rtimes\eta_{\sPi_p,3}$. 

Now suppose that $\Ind_{Q_{\GSp(4)}}^{\GSp(4)}(\sigma_p, \eta_{\sPi_p,3})$ is reducible. Then by the classification in \cite{RSLocal}, the possibilities are III,IV,V,VI. \eqref{eq:chi-vp} plus $l_2\geq 3$ excludes V,VI, so  we are left with the following two possibilities:
\begin{enumerate}[label=(\roman*)]
\item $l_2=3$, $\eta_{\sPi_p,2}\eta^{-1}_{\sPi_p,3}=|\cdot|_p$, $\eta_{\sPi_p,1}\eta^{-1}_{\sPi_p,3}\neq |\cdot|^2_p$, with IIIa and IIIb  appearing as subquotients of the induction, and $\Pi_p$ belongs to IIIa
\item $l_1=l_2=3$, $\eta_{\sPi_p,2}\eta^{-1}_{\sPi_p,3}=|\cdot|_p$ $\eta_{\sPi_p,1}\eta^{-1}_{\sPi_p,3}=|\cdot|^2_p$,  with IVa and IVc appearing as subquotients of the induction, and $\Pi_p$ belongs to IVa.
\end{enumerate}
It suffices to show that if $\cD_{l_1,l_2}$ has a nonzero $(\bS,\Lambda_\infty)$-Bessel functional, then at most one irreducible subquotient of $\Ind_{Q_{\GSp(4)}}^{\GSp(4)}(\sigma_p, \eta_{\sPi_p,3})$ has a nonzero $(\bS,\Lambda_p)$-Bessel functional. By \cite[Theorem~6.2.2]{RSBessel}, this is automatically true if $\cK$ does not split at $p$,  and if $p$ splits in $\cK$ as $\fp\bar{\fp}$, we only need to show that $\Lambda_{\fp}\Lambda^{-1}_{\bar{\fp}}\neq (\eta_{\sPi_p,1}\eta^{-1}_{\sPi_p,3})^{\pm 1}$ in case (i), and $\Lambda_{\fp}\Lambda^{-1}_{\bar{\fp}}\neq |\cdot|_p^{\pm 2}$ in case (ii). Suppose that $\cD_{l_1,l_2}$ has a nonzero $(\bS,\Lambda_\infty)$-Bessel functional with $\Lambda_\infty$ having type $(k_1,k_2)$. Then  we know that  $-l_1+l_2\leq k_1-k_2\leq l_1-l_2$ \cite[Theorem~13.4]{Prasad-Takloo}. Hence, in case (i), $l_2=3$, we have $ \val_p\left(\Lambda_\fp\Lambda^{-1}_{\bar{\fp}}(p)^{\pm 1}\right)  \geq -l_1+3$.  Combining this with \eqref{eq:chi-vp}, we see that $\Lambda_{\fp}\Lambda^{-1}_{\bar{\fp}}\neq (\eta_{\sPi_p,1}\eta^{-1}_{\sPi_p,3})^{\pm 1}$.  In case (ii), $l_1=l_2=3$, we have  $\val_p\left(\Lambda_\fp\Lambda^{-1}_{\bar{\fp}}(p)\right)=0$, which implies $\Lambda_{\fp}\Lambda^{-1}_{\bar{\fp}}\neq |\bdot|_p^{\pm 2}$.

\end{proof}

\section{The Eisenstein measure}

\subsection{The setup}

\subsubsection{The assumptions on $\Pi$ and $\pi$.} \label{sec:ord-at-p}


From now on, we assume that we are given a tempered cuspidal automorphic representation $\Pi$ of $\GSp(4,\bA_\bQ)$ with a unitary central character $\omega_{\sPi}$ and a cuspidal automorphic representation $\pi$ of $\GL(2,\bA_\bQ)$ with a unitary central character $\omega_\pi$ satisfying that
\begin{enumerate}
\item[--] $\Pi_\infty\cong \cD_{l_1,l_2}$, the holomorphic discrete series representation of $\GSp(4,\bR)$ with trivial central character and Harish-Chandra parameter $(l_1-1,l_2-2)$, ($l_1\geq l_2\geq 3$)
\item[--] $\pi_\infty\cong \cD_l$, the holomorphic discrete series representation of $\GL(2,\bR)$ with trivial central character and Harish-Chandra parameter $l-\frac{1}{2}$, ($l\geq 2$),
\item[--] $\min\{-l_1+l_2+l,\,l_1+l_2-l\}\geq 3$, 
\item[--] $\Pi$ and $\pi$ are both \emph{$p$-ordinary}, {\it i.e.} $\Pi$ (resp. $\pi$) contains an $\bU_p$-eigenvector $\varphi$ (resp. $f$) such that (for all $m_1\geq m_2\geq 0$, $m\geq 0$)
\begin{align*}
   U^{\GSp(4)}_{p,m_1,m_2}\varphi&=\lambda^{m_1}_1\lambda^{m_2}_2\varphi, 
   &\val_p(\lambda_1)&=-l_1-1,
   \quad\val_p(\lambda_2)=-l_2+1,  \\
   U^{\GL(2)}_{p,m}f&=\lambda^m_p f,
   &\val_p(\lambda)&=-l,
\end{align*}
where the $\bU_p$-operators are defined as in \S\ref{sec:Uv}. (We call such a $\varphi$ (resp. $f$) a $p$-ordinary Siegel modular form (resp. modular form).)
\end{enumerate}
Let $\epsilon=\left\{\begin{array}{ll}0,&l_1+l_2+l\text{ even},\\1,&l_1+l_2+l\text{ odd}.\end{array}\right.$

\subsubsection{The auxiliary data}\label{sec:aux}
In order to construct the $p$-adic $L$-function for $\Pi\times\pi$, we pick the following auxiliary data:
\begin{enumerate}
\item[--] an imaginary quadratic field $\cK$ and a positive definite symmetric form 
\[\bS=\begin{bmatrix}\bba&\frac{\bbb}{2}\\ \frac{\bbb}{2}&\bbc\end{bmatrix}\in\Sym_2(\bQ)_{>0}
\] 
such that $\cK=\bQ(\sqrt{-\det\bS})$, 
\item[--] a (unitary) Hecke character $\Lambda:\cK^\times\backslash\bA^\times_{\cK}\ra\bC^\times$ with $\Lambda|_{\bA^\times_{\bQ}}=\omega_{\sPi}$ and, denoting its $\infty$-type of $\Lambda$ by $(\frac{r_{\sLambda}}{2},-\frac{r_{\sLambda}}{2})$, $-l_1+l_2\leq r_{\sLambda}\leq l_1-l_2$,
\item[--] a (unitary) Hecke character $\Upsilon:\cK^\times\backslash\bA^\times_{\cK}\ra\bC^\times$ with $\Upsilon|_{\bA^\times_{\bQ}}=\omega_\pi$.
\end{enumerate}
(Note that by \cite[Theorem~13.4]{Prasad-Takloo}, the conditions on $\Lambda$ are necessary for $\Pi$ to have a nontrivial $(\bS,\Lambda)$-Bessel model.) We put $\Xi=(\Lambda\Upsilon)^{-c}$.

\vspace{.5em}

We also fix a positive integer $N$ coprime to $p$ and a positive integer $m_0$  such that 
\begin{align*}
   \Pi^{K_{\GSp(4),f}}&\neq\{0\}, 
   &\pi^{K_{\GL(2),f}}&\neq \{0\},
\end{align*}
with
\begin{equation}\label{eq:KK}
\begin{aligned}  
   K_{\GSp(4),f}&=\prod_{v\nmid Np}\GSp(4,\bZ_v)\times\prod_{v\,|\,N} K'_{\GSp(4),v}\left(\varpi^{\val_v(N)}_v\right)\times K^1_{\GSp(4),p}(p^{m_0}),\\
   K_{\GL(2),f}&=\prod_{v\nmid Np}\GL(2,\bZ_v)\times\prod_{v\,|\,N} K^1_{\GL(2),v}\left(\varpi^{\val_v(N)}_v\right)\times  K^1_{\GU(1,1),p}(p^{m_0}),
\end{aligned}
\end{equation}
(where the level groups $K'_{\GSp(4)}$, $K^1_{\GSp(4)}$, $K^1_{\GL(2)}$ are defined in \eqref{eq:Ks}), and moreover
\begin{align*}
   &\bS\in \Sym_2(\bZ_v), \,\bbc\in \bZ^\times_v, \, 4\det\bS=\mr{disc}(\cK_v/\bQ_v),
   &&\text{for all $v\nmid Np\infty$},\\
   &\Lambda_v|_{1+N\cO_{K,v}}=\Upsilon_v|_{1+N\cO_{K,v}}=\triv,
   &&\text{for all $v\,|\,N$},
\end{align*}
Put
\[
    U^p_N=\prod_{v\nmid Np}\bZ^\times_v\times\prod_{v\,|\,N}1+N\bZ_v.
\]

\subsection{Nearly holomorphic automorphic forms and the ordinary projection}

Let $G=\GSp(4)$ or $\GU(1,1)$ and 
\[
   \bH_G=\left\{\begin{array}{ll}
   \{z\in \Sym_2(\bC):\im(z)>0\}, &G=\GSp(4),\\[1ex]
   \{z\in\bC:\im{z}>0\},&G=\GU(1,1)
   \end{array}\right.
\]
be the symmetric domain of $G$. A smooth function on $\vec{F}(z):\bH_G\ra V$, with $V$ a finite dimensional $\bC$-vector space, is called {\it nearly holomorphic of degree $\leq e$} if there exist holomrphic functions $\vec{F}_1(z),\dots,\vec{F}_r(z)$ on $\bH_G$ and polynomials $P_1,\dots,P_r$ in $\bC[\Sym_2]$ (resp. $\bC[\Sym_1]$) of degree $\leq e$ such that
\[
   \vec{F}(z)=\sum_{j=1}^r \vec{F}_j(z)\cdot P_j\left(\im(z)^{-1}\right).
\] 

Let $K_\infty\subset G(\bR)$ be a maximal connected compact subgroup identified with $\U(2,\bR)$ (resp. $\U(1,\bR)$) via $ai+b\mapsto \begin{bmatrix}a&b\\-b&a\end{bmatrix}$. Given a finite dimensional algebraic representation $\rho:\GL(2)\ra V_\rho$ (resp. $\GL(1)\ra V_\rho$), we fix a linear functional $\ell_{\mr{hw}}$  projecting $V_\rho$ to its highest weight space. $\rho$ induces a representation $\rho:K_\infty\ra V_\rho(\bC)$. Let $K_{G,f}\subset G(\bA_f)$ be an open compact subgroup. For $z\in \bH_G$, we write 
\[
   g(z)=\begin{bmatrix}\im(z)^{\frac{1}{2}}&\mr{Re}(z)\im(z)^{-\frac{1}{2}}\\0&\im(z)^{-\frac{1}{2}}\end{bmatrix}.
\] 
The space of \emph{nearly holomorphic automorphic forms on $G$ of weight $\rho$, level $K_{G,f}$, and degree $\leq e$} is defined to be
\begin{align*}
   \cN^e_G(K_{G,f},\rho)=\left\{\varphi\in \cA(G)^{K_{G,f}}\middle|\,\begin{aligned}&\text{for all $g_f\in G(\bA_{\bQ,f})$ there exists $\vec{F}:\bH_G\ra V_\rho(\bC)$}\\
   &\text{nearly holomorphic of degree $\leq e$ such that}\\
   &\text{$\varphi\big(g_fg(z)k_\infty\big)=\ell_{\mr{hw}}\left(\rho(k_\infty)\vec{F}(z)\right)$ for all $k_\infty\in K_\infty,z\in\bH_G$} \end{aligned}\right\}.
\end{align*}
We let $\cN^e_G(K_{G,f},\rho,\bar{\bQ})\subset \cN^e_G(K_{G,f},\rho)$ denote the subspace consisting of the forms $\varphi$ such that for all $g_f\in G(\bA_{\bQ,f})$, the Fourier coefficients of $\varphi\big(g_fg(z)\big)$ belong to $\bar{\bQ}[\im(z)^{-1}]$.

\begin{thm}\label{thm:NHF}
\begin{enumerate}
\item $\dim_{\bC}\cN^e_G(K_{G,f},\rho)<\infty$.
\item $\cN^e_G(K_{G,f},\rho)=\cN^e_G(K_{G,f},\rho,\bar{\bQ})\otimes_{\bar{\bQ}}\bC$.
\end{enumerate}
\end{thm}
\begin{proof}
These can be deduced from \cite[Lemma 14.3, Proposition 14.10]{Sh00} or by interpreting nearly holomorphic forms as the images of global sections of automorphic vector bundles under a $C^\infty$-splitting of the vector bundle $\cH^1_\dR$ as in \cite[Section 6]{EisDiff} and \cite[Section 2]{NHF}. 
\end{proof}

In the following, we assume that $K_{G,f}$ contains $K^1_{G,p}(p^\infty)$. Then it is easy to see that the $\bU_p$-operators preserve $\cN^e_G(K_{G,f},\rho,\bar{\bQ})$.  Also, if we normalize them as 
\begin{align*}
   \tilde{U}^{\GSp(4)}_{p,m_1,m_2}&=p^{m_1(l_1+1)+m_2(l_2-1)}\,U^{\GSp(4)}_{p,m_1,m_2},
   &\tilde{U}^{\GU(1,1)}_{p,m_3}&=p^{m_3 l}\,U^{\GU(1,1)}_{p,m_3},
\end{align*}
then the $p$-adic norm of $\tilde{U}^{\GSp(4)}_{p,m_1,m_2}$ (resp. $\tilde{U}^{\GU(1,1)}_{p,m_3})$ on 
\begin{align*}
   \cN^e_{\GSp(4)}(K_{G,f},\rho_{l_1,l_2},\bar{\bQ}_p)&= \cN^e_{\GSp(4)}(K_{G,f},\rho_{l_1,l_2},\bar{\bQ})\otimes_{\bar{\bQ}}\bar{\bQ}_p\\
   (\text{resp. } \cN^e_{\GU(1,1)}(K_{G,f},\rho_l,\bar{\bQ}_p)&=\cN^e_{\GU(1,1)}(K_{G,f},\rho_l,\bar{\bQ})\otimes_{\bar{\bQ}}\bar{\bQ}_p)
\end{align*} 
is bounded uniformly for all $m_1\geq m_2\geq 0,m_3\geq 0$, where $\rho_{l_1,l_2}$ (resp. $\rho_l$) is an irreducible representation of $\GL(2)$ (resp. $\GL(1)$) of highest weight $(l_1,l_2)$ (resp. $l$). This boundedness plus the finiteness of the dimension in Theorem~\ref{thm:NHF} implies the convergence of $\lim\limits_{n\to\infty} \left(\tilde{U}^{\GSp(4)}_{p,2,1}\right)^{n!}$ on  $\cN^e_{\GSp(4)}(K_{\GSp(4),f},\rho_{l_1,l_2},\bar{\bQ}_p)$, and $\lim\limits_{n\to\infty}\left(\tilde{U}^{\GU(1,1)}_{p,1}\right)^{n!}$ on $\cN^e_{\GU(1,1)}(K_{\GU(1,1),f},\rho_l,\bar{\bQ}_p)$. We define the \emph{ordinary projectors} $e^{\GSp}_{\ord}$, $e^{\GU(1,1)}_\ord$ as these limits.

Put
\begin{align*}
   \cM_{l_1,l_2}(K_{\GSp(4),f},\bar{\bQ}_p)&=\cN^0_{\GSp(4)}(K_{\GSp(4),f},\rho_{l_1,l_2},\bar{\bQ}_p),\\
   \cM_l(K_{\GU(1,1),f},\bar{\bQ}_p)&=\cN^0_{\GL(2)}(K_{\GU(1,1),f},\rho_{l_1,l_2},\bar{\bQ}_p).
\end{align*}

\begin{prop}\label{prop:e-conv}
The ordinary projector $e^{\GSp(4)}_{\ord}$ (resp. $e^{\GU(1,1)}_\ord$) maps the space of nearly holomorphic forms $\bigcup\limits_{e\geq 0} \cN^e_{\GSp(4)}\left(K_{\GSp(4),f},\rho_{l_1,l_2},\bar{\bQ}_p\right)$ (resp. $\bigcup\limits_{e\geq 0} \cN^e_{\GU(1,1)}\left(K_{\GU(1,1),f},\rho_l,\bar{\bQ}_p\right)$) into $\cM_{l_1,l_2}\left(K_{\GSp(4),f},\bar{\bQ}_p\right)$ (resp. $\cM_l\left(K_{\GU(1,1),f},\bar{\bQ}_p\right)$).
\end{prop}
\begin{proof}
This follows from the observation that the normalized $\bU_p$-actions on the non-holomorphic part are divisible by $p$. 
\end{proof}

\subsection{The choice of sections for Siegel Eisenstein series}\label{sec:dsec-choice}

For each finite order Dirichlet character $\chi:\bQ^\times\backslash\bA^\times_{\bQ}/U^p_N\bR^\times_{+}\ra\bC^\times$, and each integer $k$ satisfying the condition \eqref{eq:k-crt}, we choose a suitable section in $\dsec_k(s,\chi,\Xi)\in  I(s,\chi,\Xi)$, place by place, such that when $k,\chi$ varies,
\[
    e^{\GSp(4)}_\ord e^{\GU(1,1)}_\ord \left(E^\Sieg(\bdot\,;\dsec_k(s,\chi,\Xi)|_{s=\frac{2k+\epsilon-1}{2},\GSp(4)\times\GU(1,1)}\right)
\]
is interpolated by a $p$-adic measure valued in $p$-ordinary holomorphic automorphic forms on $\GSp(4)\times\GU(1,1)$ of weight $(l_1,l_2;l)$. Here for an automorphic form $F$ on $\GU(3,3)$, we write $F|_{\GSp(4)\times\GU(1,1)}$ to mean 
 the extension by zero of its restriction (via \eqref{eq:gp-ebd}) to $\GSp(4)\times_{\GL(1)}\GU(1,1)$.

\subsubsection{The archimedean place.}\label{sec:sec-inf}

Let $\fp^+_{\GU(3,3)}\subset (\Lie\GU(3,3)(\bR))\otimes_{\bR}\bC$ be the subalgebra of consisting 
\begin{align*}
  &\begin{aligned}
    \mu^+_{\GU(3,3),X}&=\frac{1}{2}\begin{bmatrix}\bid_3&\bid_3\\ i\cdot \bid_3&-i\cdot \bid_3\end{bmatrix}
    \begin{bmatrix}\bid_3&X\\&\bid_3\end{bmatrix}
    \begin{bmatrix}\bid_3&\bid_3\\ i\cdot \bid_3&-i\cdot \bid_3\end{bmatrix}^{-1}\\
    &=\begin{bmatrix}X\\&-X\end{bmatrix}\otimes 1+\begin{bmatrix}&X\\X\end{bmatrix}\otimes i,
\end{aligned}
   &&X\in\Her_3(\cK\otimes_{\bQ}\bR).
\end{align*}
For $1\leq i,j\leq 3$, put
\begin{equation*}
   \mu^+_{\GU(3,3),ij}=\mu^+_{\GU(3,3),\frac{E_{ij}+E_{ji}}{2}\otimes 1+\frac{E_{ij}-E_{ji}}{2\sqrt{\bbd}}\otimes \sqrt{\bbd}},
\end{equation*}
where $\sqrt{\bbd}$ is a purely imaginary number in $\cK$.  The sections we will use  are obtained by applying such differential operators to the  classical scalar-weight sections in \S\ref{sec:sw-dsec}. 

For an integer $k$ such that
\begin{equation}\label{eq:krange}
    -\frac{\min\{-l_1+l_2+l,\,l_1+l_2-l\}}{2}+2 \leq k+\frac{\epsilon}{2}\leq \frac{\min\{-l_1+l_2+l,\,l_1+l_2-l\}}{2}-1,
\end{equation}
$s=k+\frac{\epsilon}{2}$ is a critical point for $L(s,\tilde{\Pi}\times\tilde{\pi}\times\chi)$. In order to get this $L$-value,  we will use 
\begin{equation*}
    D_{t_k,r_{\sLambda},l_1,l_2,l}\cdot \dsec^{t_k}_\infty(s,\chi,\Xi),
\end{equation*}
where $\dsec^{t_k}_\infty(s,\chi,\Xi)$ is defined as in \eqref{eq:sw-dsec} with
\begin{equation}\label{eq:tk}
    t_k=\left\{\begin{array}{ll}2k+\epsilon+2,&k+\frac{\epsilon}{2}\geq \frac{1}{2},\\ -2k-\epsilon+4,&k+\frac{\epsilon}{2}\leq \frac{1}{2},\end{array}\right.
\end{equation} 
and the differential operator  $D_{t_k,r_{\sLambda},l_1,l_2,l}\in \bC[\fp^+_{\GU(3,3)}]$ is 
\begin{align*}
   D_{t_k,r_{\sLambda},l_1,l_2,l}
   =&\,(\alphaS-\alphabS)^{l_1-l} (2\pi i)^{\frac{3t_k-l_1-l_2-l}{2}}
   \cdot \left(\mu^+_{\GU(3,3),31}\right)^{\frac{l_1-l_2+r_{\scriptscriptstyle \Lambda}}{2}}\left(\mu^+_{\GU(3,3),13}\right)^{\frac{l_1-l_2-r_{\scriptscriptstyle \Lambda}}{2}}\\
   &\times\det\begin{bmatrix}\mu^+_{\GU(3,3),13}&\mu^+_{\GU(3,3),23}\\ \mu^+_{\GU(3,3),31}&\mu^+_{\GU(3,3),32}\end{bmatrix}^{\frac{-l_1+l_2+l-t_k}{2}}
   \left(\frac{\mu^+_{\GU(3,3),21}-\mu^+_{\GU(3,3),12}}{2}\right)^{\frac{l_1+l_2-l-t_k}{2}}.
\end{align*}
(Note that the condition \eqref{eq:krange} and the condition on the $\infty$-type of $\Lambda$ guarantees that the exponents above are all non-negative, and also that the definition of $t_k$ and $\Lambda(-1)=\omega_\pi(-1)$ guarantee that all the exponents are integers.)

\subsubsection{The place $p$.}\label{sec:secp}
It is well-known (by the theory of Jacquet modules) that the ordinarity of $\Pi$ (resp. $\pi$) at $p$ implies that there exists continuous characters $\eta_{\sPi_p,1},\eta_{\sPi_p,2},\eta_{\sPi_p,3}:\bQ^\times_p\ra \bC^\times$  (resp. $\eta_{\pi_p,1},\eta_{\pi_p,2}:\bQ^\times_p\ra\bC^\times$) such that 
\begin{equation}\label{eq:Ind-eta}
\begin{aligned}
   \Pi_p&\cong \text{a subqutient of  $\Ind_{B}^{\GSp(4)}(\eta_{\sPi_p,1},\eta_{\sPi_p,2},\eta_{\sPi_p,3})$},\\
   \pi_p&\cong \text{a subqutient of $\Ind^{\GL(2)}_{B}(\eta_{\pi_p,1},\eta_{\pi_p,2})$},
\end{aligned}
\end{equation}
where the inductions induce the characters
\begin{align*} 
   \begin{bmatrix}\delta a_1&\ast&\ast&\ast\\&\delta a_2&\ast&\ast\\ &&a_2\\&&\ast&a_1\end{bmatrix}&\longmapsto \eta_{\sPi_p,1}(a_1)\eta_{\sPi_p,2}(a_2)\eta_{\sPi_p,3}(\delta)
   &\begin{bmatrix}a&\ast\\&d\end{bmatrix}&\longmapsto \eta_{\pi_p,1}(a)\,\eta_{\pi_p,2}(d),
\end{align*}
and
\begin{equation}\label{eq:xi-val}
\begin{aligned}
    \val_p\left(\eta_{\sPi_p,1}\,\eta^{-1}_{\sPi_p,3}(p)\right)&=-l_1+1,
    \quad&\val_p\left(\eta_{\sPi_p,2}\,\eta^{-1}_{\sPi_p,3}(p)\right)&=-l_2+2,\\
    \val_p\left(\eta^{-1}_{\pi_p,1}\,\eta_{\pi_p,2}(p)\right)&=-l+1.
\end{aligned}
\end{equation}
If $\varphi\in\Pi$ (resp. $f\in\pi$) is a $p$-ordinary form with $U^{\GSp(4)}_{p,m_1,m_2}$-eigenvalue $\lambda_{p,1}(\varphi)^{m_1}\lambda_{p,2}(\varphi)^{m_2}$ (resp. $U^{\GL(2)}_{p,m_3}$-eigenvalue $\lambda_p(f)^{m_3}$), then we have
\begin{align*}
  \lambda_{p,1}(\varphi)&=p^{-2}\cdot \eta_{\sPi_p,1}\,\eta^{-1}_{\sPi_p,3}(p),
   &\lambda_{p,2}(\varphi)&=p^{-1}\cdot \eta_{\sPi_p,2}\,\eta^{-1}_{\sPi_p,3}(p),\\
   \lambda_p(f)&= p^{-1}\cdot  \eta^{-1}_{\pi_p,1}\,\eta_{\pi_p,2}(p).
\end{align*}

\vspace{.5em}
Our choice of $\dsec_p(s,\chi,\Xi)$ will depend on the nebentypus of the ordinary vectors in $\Pi_p$, $\pi_p$. Inside the induction, the nonzero eigenvalues for $U^{\GSp(4)}_{p,m_1,m_2}$, $m_1\geq m_2>0$  (resp. $U^{\GL(2)}_{p,m}$, $m>0$) are $\left(\eta_{\sPi_p,1}\,\eta^{-1}_{\sPi_p,3}(p)\,p^2\right)^{\pm m_1} \left(\eta_{\sPi_p,2}\,\eta^{-1}_{\sPi_p,3}(p)\,p\right)^{\pm m_2}$, $\left(\eta_{\sPi_p,1}\,\eta^{-1}_{\sPi_p,3}(p)\,p^2\right)^{\pm m_2} \left(\eta_{\sPi_p,2}\,\eta^{-1}_{\sPi_p,3}(p)\,p\right)^{\pm m_1}$ (resp. $\left(\eta^{-1}_{\pi_p,1}\,\eta_{\pi_p,2}(p)\,p\right)^{\pm m}$). The condition \eqref{eq:xi-val} implies that these eigenvalues are pairwise distinct. The (generalized) eigenspace for each nonzero eigenvalue is one dimensional. In particular, the ordinary vector inside $\Ind_{B}^{\GSp(4)}(\eta_{\sPi_p,1},\eta_{\sPi_p,2},\eta_{\sPi_p,3})$ (resp. $\Ind^{\GL(2)}_{B}(\eta_{\pi_p,1},\eta_{\pi_p,2})$) is unique up to scalar. One can easily write down the ordinary vectors as 
\begin{align*}
   g=\begin{bmatrix}a_1&a_2&b_1&b_2\\a_3&a_4&b_3&b_4\\c_1&c_2&d_1&d_2\\c_3&c_4&d_3&d_4\end{bmatrix}
   &\longmapsto\mathds{1}_{\bZ_p}(c^{-1}_3c_4)\mathds{1}_{\Sym_2(\bZ_p)}\left(\begin{bmatrix}c_1&c_2\\c_3&c_4\end{bmatrix}^{-1}\begin{bmatrix}d_1&d_2\\d_3&d_4\end{bmatrix}\right)
   \cdot\left\vert\nu^3_g(c_1c_2-c_3c_4)^{-4}c^2_3\right\vert^{1/2}_p\\
    &\quad\quad\times \eta_{\sPi_p,1}(c_3)\,\eta_{\sPi_p,2}(c_2-c^{-1}_3c_1c_4)\,\eta_{\sPi_p,3}(\nu_g(c_1c_4-c_2c_3)^{-1}),\\
  h=\begin{bmatrix}a&b\\c&d\end{bmatrix} 
  &\longmapsto \left\vert\nu_h c^{-2}\right\vert^{1/2}_p\, \eta^{-1}_{\pi_p,1}(\nu_h c^{-1})\,\eta_{\pi_p,2}(c).
\end{align*}
We also easily see that the ordinary vectors in $\Pi_p$, $\pi_p$ have
\begin{align*}
   \text{$U^{\GSp(4)}_{p,m_1,m_2}$-eigenvalue} &= \left(\eta_{\sPi_p,1}\,\eta^{-1}_{\sPi_p,3}(p)\,p^{2}\right)^{ m_1} \left(\eta_{\sPi_p,2}\,\eta^{-1}_{\sPi_p,3}(p)\,p\right)^{ m_2},\\
   \text{$U^{\GL(2)}_{p,m}$-eigenvalue} & = \left(\eta^{-1}_{\pi_p,1}\,\eta_{\pi_p,2}(p)\,p\right)^{m},
\end{align*} 
and they have the nebentypus such that the right translation of $B(\bZ_p)$ is via the character
\begin{align*}
   \begin{bmatrix} a_1&\ast&\ast&\ast\\& a_2&\ast&\ast\\ &&\delta a_2\\&&\ast&\delta a_1\end{bmatrix}&\longmapsto \eta_{\sPi_p,1}(a_1)\,\eta_{\sPi_p,2}(a_2)\,\eta_{\sPi_p,3}(\delta) \\
   \begin{bmatrix}a&\ast\\&d\end{bmatrix}&\longmapsto \eta_{\pi_p,1}(d)\,\eta_{\pi_p,2}(a).
\end{align*}
\\


We define $\Schw_{\chi^\circ_p,\Lambda^\circ_p,\,\eta^\circ_{\Pi_p},\eta^\circ_{\pi_p}}$ as the Schwartz function on $\Her_3(\cK_p)$ with Fourier transform
\begin{align*}
   &\cF\Schw_{\chi^\circ_p,\Lambda^\circ_p,\,\eta^\circ_{\Pi_p},\eta^\circ_{\pi_p}}\begin{pmatrix}w_{11}&\fw_{12}&\fw_{13}\\ \bar{\fw}_{12}&w_{22}&\fw_{23}\\ \bar{\fw}_{13}&\bar{\fw}_{23}&w_{33}\end{pmatrix}\\
   =&\, \mathds{1}_{\Sym_3(\bZ_p)}\begin{pmatrix}w_{11}&\frac{\fw_{12}+\bar{\fw}_{12}}{2}\\ \frac{\fw_{12}+\bar{\fw}_{12}}{2}&w_{22}\end{pmatrix}
   \cdot  \mathds{1}_{\bZ_p}(w_{33})
   \cdot \chi^\circ_p\eta^{\circ-1}_{\pi_p,2}\eta^{\circ-1}_{\Pi_p,3}\left(\frac{(\alphaS-\alphabS)(\fw_{12}-\bar{\fw}_{12})}{2}\right)\\
   &\times \mathds{1}_{(\alphaS-\alphabS)\cO_{\cK,p}}(\fw_{23})\cdot \Lambda^{\circ}_p(\fw_{13})  \,\,\eta^{\circ-1}_{\Pi_p,2}(\fw_{13}\bar{\fw}_{13})   \cdot  \chi^\circ_p\eta^{\circ-1}_{\pi_p,1}\eta^{\circ-1}_{\Pi_p,1}\left(\frac{\bar{\fw}_{13}\fw_{23}-\fw_{13}\bar{\fw}_{23}}{\alphaS-\alphabS}\right).
\end{align*}
(See \eqref{eq:char-circ} for the notation of superscript $^\circ$ attached to a character.) For our construction of $p$-adic $L$-functions, we will use
\[
    \dsec^\bc_{p,\Schw_{\chi^\circ_p,\Lambda^\circ_p,\,\eta^\circ_{\sPi_p},\eta^\circ_{\pi_p}}}(s,\chi,\Xi).
\]
(This choice of $\Schw_{\chi^\circ_p,\Lambda^\circ_p,\,\eta^\circ_{\Pi_p},\eta^\circ_{\pi_p}}$ makes the nebentypus at p of the restriction of the corresponding Siegel Eisenstein series match those of the $p$-ordinary forms in $\Pi$ and $\pi$.)

\subsubsection{The places dividing $N$}\label{sec:fN}
For a place $v\mid N$, we will use 
\[
    \dsec^\bc_{v,\Schw_{v,N}}(s,\chi,\Xi)
\]
with
\begin{align*}
   \Schw_{v,N}\begin{pmatrix}w_{11}&\bar{\fw}_{21}&\bar{\fw}_{31}\\ \fw_{21}&w_{22}& \bar{\fw}_{32}\\ \fw_{31}&\fw_{32}&w_{33}\end{pmatrix}
   =&\, \mathds{1}_{\sym_3(N\bZ_v)}\begin{pmatrix}w_{11}&\frac{\fw_{21}+\bar{\fw}_{21}}{2}\\ \frac{\fw_{21}+\bar{\fw}_{21}}{2}&w_{22}\end{pmatrix}
   \left(\cF^{-1}\mathds{1}_{-\bbc(1+N\bZ_v)}\right)(w_{33})\\
   &\hspace{-5em}\times\mathds{1}_{\bbc^2 N^3\bZ_v}\left(\frac{\fw_{21}-\bar{\fw}_{21}}{\alphaS-\alphabS}\right)
   \mathds{1}_{K^1_{\GL(2),v}(\varpi^{\val_v(N)}_v)}\left(\begin{bmatrix}\alphaS&1\\ \alphabS&1\end{bmatrix}^{-1}\begin{bmatrix}\fw_{31}&\fw_{32}\\ \bar{\fw}_{31}&\bar{\fw}_{32}\end{bmatrix}\right)
\end{align*}

\subsection{The Fourier coefficients of Siegel Eisenstein series}\label{sec:FC-SE}
Given $\bbeta\in\Her_3(\cK)$, the $\bbeta$-Fourier coefficient of the Siegel Eisenstein series is computed as 
\begin{align*}
	E^\Sieg_{\bbeta}(g;\dsec(s,\chi,\Xi))
	=\int_{\Her_3(\cK)\backslash\Her_3(\bA_\cK)} E^\Sieg\left(\begin{bmatrix}\bid_3&\fx\\0&\bid_3\end{bmatrix} g;\dsec(s,\chi,\Xi)\right) \psi_{\bA_\bQ}\left(-\Tr\bbeta\fx\right)\,d\fx.
\end{align*}
When $\bbeta$ is non-degenerate,
\begin{equation}\label{eq:EtoW}
E^\Sieg_{\bbeta}(g;\dsec(s,\chi,\Xi))=\prod_v W_{\bbeta,v}\left(g_v;\dsec_v(s,\chi,\Xi)\right)
\end{equation}
with
\begin{equation}\label{eq:FCv}
	W_{\bbeta,v}\left(g;\dsec_v(s,\chi,\Xi)\right)
	=\int_{\Her_3(\cK_v)} \dsec_v(s,\chi,\Xi)\left(\begin{bmatrix}&-\bid_3\\ \bid_3\end{bmatrix}\begin{bmatrix}\bid_3&\fx\\0&\bid_3\end{bmatrix}g\right)\psi_v \left(-\Tr\bbeta\fx\right)\,d\fx.
\end{equation}
When $\bbeta$ is degenerate and $g=\begin{bmatrix}\fA\\&\nu\ltrans{\bar{\fA}}^{-1}\end{bmatrix}$, \eqref{eq:EtoW} also holds if there exists $v$ such that $\dsec(s,\chi,\Xi)$ is supported on the big Bruhat cell.

For a finite place $v$ where $\chi_v$, $\Xi_v$ are unramified, we have
\begin{equation}\label{eq:FC-unr}
	\begin{aligned}
		&W_{\bbeta,v}\left(\begin{bmatrix}\fA\\&\nu\ltrans{\bar{\fA}}^{-1}\end{bmatrix};\dsec^\sph_v(s,\chi,\Xi)\right)\\
		=&\,\Xi_v\left(\det(\nu\bar{\fA}^{-1})\right) \chi_v\left(\det(\nu\fA^{-1}\bar{\fA}^{-1})\right) |\det(\nu\fA^{-1}\bar{\fA}^{-1})|^{s-\frac{3}{2}}_v
		\cdot d_{3,v}\big(s,\Xi(\chi\circ\Nm)\big) \\
		&\times \mathds{1}_{\Her_c(\cO_{\cK,v})^*}\left(\nu^{-1}\ltrans{\bar{\fA}}\bbeta\fA\right)\cdot h_{v,\nu^{-1}\ltrans{\bar{\fA}}\bbeta\fA}\left(\Xi_{\bQ,v}\chi^2_v(\varpi_v)|\varpi_v|_v^{2s+3} \right),
	\end{aligned}
\end{equation}
where $h_{v,\nu^{-1}\ltrans{\bar{\fA}}\bbeta\fA}$ is a monic polynomial in $\bZ[X]$, and is the constant $1$ if $\nu^{-1}\ltrans{\bar{\fA}}\bbeta\fA$ belongs to $\GL(3,\cO_{\cK,v})$ \cite[Theorem~13.6, Proposition~14.9]{ShEuler}, and the factor $d_{3,v}\big(s,\Xi(\chi\circ\Nm)\big)$ is defined in \eqref{eq:dnv}.

For a big-cell section associated to a Schwartz function $\Phi_v$ on $\Her_3(\cK_v)$ defined in \S\ref{sec:bigcell}, an easy computation gives
\begin{align*}
	W_{\bbeta,v}\left(\begin{bmatrix}\fA\\&\nu\ltrans{\bar{\fA}}^{-1}\end{bmatrix}; \dsec^{\bc}_{v,\Schw_v}(s,\chi,\Xi)\right)
	=&\,\Xi_v\left(\det(\nu\bar{\fA}^{-1})\right) \chi_v\left(\det(\nu\fA^{-1}\bar{\fA}^{-1})\right) |\det(\nu\fA^{-1}\bar{\fA}^{-1})|^{s-\frac{3}{2}}_v\\
	&\times \cF\Schw_v\left(\nu^{-1}\ltrans{\bar{\fA}}\bbeta\fA\right).
\end{align*}

For the classical scalar-weight archimedean sections, by \cite[Theorem!4.2]{ShHyper}, we have
\begin{align*}
   &\left.W_{\bbeta,\infty}\left(\begin{bsm}\left(\frac{z-\ltrans{\bar{z}}}{2i}\right)^{1/2}&\frac{z+\ltrans{\bar{z}}}{2}\left(\frac{z-\ltrans{\bar{z}}}{2i}\right)^{-1/2}\\0&\left(\frac{z-\ltrans{\bar{z}}}{2i}\right)^{-1/2}\end{bsm}_\infty; \dsec^{t_k}_\infty(s,\chi,\Xi)\right)\right|_{s=\frac{2k+\epsilon-1}{2}}\\
   =&\,\left.\frac{2^{3(s+\frac{t_k-1}{2})}\,i^{t_k}\,\pi^{3(s+\frac{t_k+1}{2})}}{\prod_{j=0}^2\Gamma\left(s+\frac{t_k+3}{2}-j\right)}\right|_{s=\frac{2k+\epsilon-1}{2}}
   \cdot  e^{2\pi i\Tr\bbeta z}
   \left\{\begin{array}{ll}
   	(\det\bbeta)^{2k+\epsilon-1}, &k+\frac{\epsilon}{2}\geq \frac{1}{2}\\[2ex]
   	1, &k+\frac{\epsilon}{2}\leq \frac{1}{2}.
   \end{array}\right.
\end{align*}

\subsection{The Siegel Eisenstein measure}
Let $U^p_N=\prod\limits_{v\nmid N}\bZ^\times_v\times\prod\limits_{v\mid N} (1+N\bZ_v)\subset \hat{\bZ}^\times$. For each finite order Dirichlet character $\chi:\bQ^\times\backslash\bA^\times_{\bQ}/U^p_N\bR^\times_{+}\ra\bC^\times$, and each integer $k$ lying in the range \eqref{eq:krange}, let 
\begin{equation}\label{eq:Esieg}
   E^\Sieg_{k,\chi}=
   \frac{\prod_{j=0}^2\Gamma\left(s+\frac{t_k+3}{2}-j\right)}{2^{3(s+\frac{t_k-1}{2})}\,i^{t_k}\,\pi^{3(s+\frac{t_k+1}{2})}}
    \cdot d^{Np\infty}_3\big(s,\Xi\,(\chi\circ\Nm)\big)\cdot E^\Sieg\left(\,\bdot\,;\dsec_k(s,\chi,\Xi)\right)_{\bbeta}\bigg\vert_{s=\frac{2k+\epsilon-1}{2}}
\end{equation}
with $\dsec_k(s,\chi,\Xi)$ given as
\begin{align*}
   \dsec_k(s,\chi,\Xi)=&\motimes\limits_{v\nmid Np\infty} \dsec^\sph_v(s,\chi,\Xi) \motimes\limits_{v\mid N} \dsec^\bc_{v,\Schw_{v,N}}(s,\chi,\Xi)\\
    &\motimes \dsec^\bc_{p,\Schw_{\chi^\circ_p,\Lambda^\circ_p,\,\eta^\circ_{\sPi_p},
    		\eta^\circ_{\pi_p}}}(s,\chi,\Xi)\motimes D_{t_k,r_{\sLambda},l_1,l_2,l}\cdot \dsec^{t_k}_\infty(s,\chi,\Xi),
\end{align*}
where the local sections are chosen in \S\ref{sec:dsec-choice}.

Write $Y=\left(Y_{ij}\right)_{1\leq,i,j\leq 3}$, and denote by $\bC[Y]$ the ring of $\bC$-valued polynomials in entries of $Y$.

\begin{prop}\label{prop:EFC}
For $\fA\in\GL(2,\bA^p_{\cK,f})$, $\nu\in\bA^{\times,p}_f$ and $\bbeta\in\Her_3(\cK)$, there exists 
\begin{equation}\label{eq:EbetaY}
  \left(E^\Sieg_{k,\chi}\right)_{\bbeta}\left(\begin{bmatrix}\fA\\&\nu\ltrans{\bar{\fA}}^{-1}\end{bmatrix}^p_f\right)\in \bC[Y]
\end{equation}
such that 
\begin{equation}\label{eq:Ebeta}
\begin{aligned}
   &\left(E^\Sieg_{k,\chi}\right)_{\bbeta}\left(\begin{bmatrix}\fA\\&\nu\ltrans{\bar{\fA}}^{-1}\end{bmatrix}^p_f\right)\left(Y=\left(\frac{\ltrans{z}-\bar{z}}{2i}\right)^{-1}\right)\cdot e^{2\pi i\Tr\bbeta z}\\
   =&\,\int_{\Her_3(\cK)\backslash\Her_3(\bA_{\cK})}E^\Sieg_{k,\chi}
    \left(\begin{bmatrix}\bid_3&\varsigma\\0&\bid_3\end{bmatrix}\begin{bmatrix}\fA\\&\nu\ltrans{\bar{\fA}}^{-1}\end{bmatrix}^p_f\begin{bsm}\left(\frac{z-\ltrans{\bar{z}}}{2i}\right)^{1/2}&\frac{z+\ltrans{\bar{z}}}{2}\left(\frac{z-\ltrans{\bar{z}}}{2i}\right)^{-1/2}\\0&\left(\frac{z-\ltrans{\bar{z}}}{2i}\right)^{-1/2}\end{bsm}_\infty\right)\cdot \psi_{\bA_\bQ}\left(-\Tr\bbeta\varsigma\right)\,d\varsigma.
\end{aligned}
\end{equation}
We know that \eqref{eq:EbetaY} is $0$ unless $\bbeta>0$, and for $\bbeta>0$, we have the formula
\begin{align*}   
    &\left(E^\Sieg_{k,\chi}\right)_{\bbeta}\left(\begin{bmatrix}\fA\\&\nu\ltrans{\bar{\fA}}^{-1}\end{bmatrix}^p_f\right)(Y=0) \\
   =&\,  \Xi(\det(\nu\ltrans{\bar{\fA}}^{-1}))  \chi(\det(\nu\fA^{-1}\ltrans{\bar{\fA}}^{-1}))|\det(\nu\fA^{-1}\ltrans{\bar{\fA}}^{-1})|^{k+\frac{\epsilon}{2}+1}  \\
   &\times\prod_{v\nmid Np\infty} h_{v,\nu^{-1}\ltrans{\bar{\fA}}\bbeta\fA}\left(\Xi_{\bQ,v}\chi^2_v(\varpi_v)|\varpi_v|^{2k+\epsilon+2}\right)
    \prod_{v\mid N}\cF\Schw^\vol_{v,\val_v(N)} \left(\nu^{-1}_v\ltrans{\bar{\fA}}_v\bbeta\fA_v\right)\\
    &\times \Lambda^\circ_p\Lambda_\infty(\beta_{13})\cdot \eta^{\circ-1}_{\sPi_p,2}(\beta_{13}\bar{\beta}_{13}) (\beta_{13}\bar{\beta}_{13})^{\frac{l_1-l_2}{2}}
    \cdot\eta^{\circ-1}_{\sPi_p,1}\eta^{\circ-1}_{\pi_p,1}\left(\frac{\bar{\beta}_{13}\beta_{23}-\beta_{13}\bar{\beta}_{23}}{\alphaS-\alphabS}\right)\left(\frac{\bar{\beta}_{13}\beta_{23}-\beta_{13}\bar{\beta}_{23}}{\alphaS-\alphabS}\right)^{\frac{-l_1+l_2+l+\epsilon}{2}}\\
    &\times \eta^{\circ-1}_{\sPi_p,3}\eta^{\circ-1}_{\pi_p,2}\left( \frac{(\alphaS-\alphabS)(\beta_{12}-\bar{\beta}_{12})}{2}\right)\left( \frac{(\alphaS-\alphabS)(\beta_{12}-\bar{\beta}_{12})}{2}\right)^{\frac{l_1+l_2-l+\epsilon}{2}}\\
   &\times \chi^\circ_p \left(\frac{\beta_{12}-\bar{\beta}_{12}}{2}\det\begin{bmatrix}\bar{\beta}_{13}&\bar{\beta}_{23}\\ \beta_{13}&\beta_{23}\end{bmatrix}\right) 
   \left\{\begin{array}{ll}
    \left(\frac{\beta_{12}-\bar{\beta}_{12}}{2}\det\begin{bmatrix}\bar{\beta}_{13}&\bar{\beta}_{23}\\ \beta_{13}&\beta_{23}\end{bmatrix}\right)^{-k-\epsilon-1} (\det\bbeta)^{2k+\epsilon-1}, &k+\frac{\epsilon}{2}\geq \frac{1}{2}\\[2ex]
  \left(\frac{\beta_{12}-\bar{\beta}_{12}}{2}\det\begin{bmatrix}\bar{\beta}_{13}&\bar{\beta}_{23}\\ \beta_{13}&\beta_{23}\end{bmatrix}\right)^{k-2} , &k+\frac{\epsilon}{2}\leq \frac{1}{2}.
   \end{array} \right.
\end{align*}
\end{prop}
\begin{proof}
This is a straightforward computation. Because the section at $p$ is supported on the big cell, we have
\begin{align*}
   \text{RHS of \eqref{eq:Ebeta}}
   =&\, d^{Np\infty}_3\big(s,\Xi\,(\chi\circ\Nm)\big)\prod_{v\nmid Np\infty} 
    W_{\bbeta,v}\left(\begin{bmatrix}\fA\\&\nu\ltrans{\bar{\fA}}^{-1}\end{bmatrix}_v,\dsec^\sph_v(s,\chi,\Xi)\right)\\
    &\hspace{-6em}\times  \prod_{v\mid N} W_{\bbeta,v}\left(\begin{bmatrix}\fA\\&\nu\ltrans{\bar{\fA}}^{-1}\end{bmatrix}_v,\dsec^\bc_{v,\Schw_{v,N}}(s,\chi,\Xi)\right) 
   \cdot W_{\bbeta,p}\left(\bid_3,\dsec^\bc_{p,\Schw_{\chi^\circ_p,\Lambda^\circ_p,\,\eta^\circ_{\sPi_p},\eta^\circ_{\pi_p}}}(s,\chi,\Xi)\right) \\
   &\hspace{-6em}\times \left.\frac{\prod_{j=0}^2\Gamma\left(s+\frac{t_k+3}{2}-j\right)}{2^{3(s+\frac{t_k-1}{2})}\,i^{t_k}\,\pi^{3(s+\frac{t_k+1}{2})}}\, W_{\beta,\infty}\left(\begin{bsm}\left(\frac{z-\ltrans{\bar{z}}}{2i}\right)^{1/2}&\frac{z+\ltrans{\bar{z}}}{2}\left(\frac{z-\ltrans{\bar{z}}}{2i}\right)^{-1/2}\\0&\left(\frac{z-\ltrans{\bar{z}}}{2i}\right)^{-1/2}\end{bsm},D_{t_k,r_{\sLambda},l_1,l_2,l}\cdot \dsec^{t_k}_\infty(s,\chi,\Xi)\right)\right|_{s=\frac{2k+\epsilon-1}{2}}.
\end{align*}
From the formulas in \S\ref{sec:FC-SE} and the choice of $\Schw_{\chi^\circ_p,\Lambda^\circ_p,\,\eta^\circ_{\Pi_p},\eta^\circ_{\pi_p}}$ in \S\ref{sec:secp}, we obtain
\begin{equation}\label{eq:FC-hol}
\begin{aligned}
   &\text{RHS of \eqref{eq:Ebeta} for $l_1=l_2=l=t_k$}\\
   =&\,\Xi(\det(\nu\ltrans{\bar{\fA}}^{-1}))  \chi(\det(\nu\fA^{-1}\ltrans{\bar{\fA}}^{-1}))|\det(\nu\fA^{-1}\ltrans{\bar{\fA}}^{-1})|^{k+\frac{\epsilon}{2}-2}\\
    &\times \prod_{v\nmid Np\infty} g_{\nu^{-1}_v\ltrans{\bar{\fA}_v}\bbeta\fA_v,v}\left(\Xi_{\bQ,v}\chi^2_v(\varpi_v)|\varpi_v|^{2k+\epsilon-1}\right)
    \prod_{v\mid N}\cF\Schw^\vol_{v,\val_v(N)} \left(\nu^{-1}_v\ltrans{\bar{\fA}}_v\bbeta\fA_v\right)\\
    &\times \Lambda^\circ_p(\beta_{13})\,\eta^{\circ-1}_{\sPi_p,2}(\beta_{13}\bar{\beta}_{13}) 
    \,\eta^{\circ-1}_{\sPi_p,1}\eta^{\circ-1}_{\pi_p,1}\left(\frac{\bar{\beta}_{13}\beta_{23}-\beta_{13}\bar{\beta}_{23}}{\alphaS-\alphabS}\right)
     \eta^{\circ-1}_{\sPi_p,3}\eta^{\circ-1}_{\pi_p,2}\left( \frac{(\alphaS-\alphabS)(\beta_{12}-\bar{\beta}_{12})}{2}\right)\\
   &\times \chi^\circ_p \left(\frac{\beta_{12}-\bar{\beta}_{12}}{2}\det\begin{bmatrix}\bar{\beta}_{13}&\bar{\beta}_{23}\\ \beta_{13}&\beta_{23}\end{bmatrix}\right) 
    \cdot e^{2\pi i\Tr\bbeta z}
   \left\{\begin{array}{ll}
    (\det\bbeta)^{2k+\epsilon-1}, &k+\frac{\epsilon}{2}\geq \frac{1}{2}\\[2ex]
  1, &k+\frac{\epsilon}{2}\leq \frac{1}{2}.
   \end{array} \right.
\end{aligned}
\end{equation}
For general $l_1,l_2,l$, it is easy to see that $D_{t_k,r_{\sLambda},l_1,l_2,l}\cdot  e^{2\pi i\Tr\bbeta z}$ is the product of $e^{2\pi i\Tr\bbeta z}$ with a polynomial in the entries of $\left(\frac{\ltrans{z}-\bar{z}}{2i}\right)^{-1}$, therefore there is a polynomial $\left(E^\Sieg_{k,\chi}\right)_{\bbeta}\left(\begin{bmatrix}\fA\\&\nu\ltrans{\bar{\fA}}^{-1}\end{bmatrix}^p_f\right)\in \bC[Y]$ such that \eqref{eq:Ebeta} holds. Although the precise formula $D_{t_k,r_{\sLambda},l_1,l_2,l}\cdot  e^{2\pi i\Tr\bbeta z}$ is hard to compute in general, it is easy to see that
\begin{equation}\label{eq:D-lead}
\begin{aligned}
   &D_{t_k,r_{\sLambda},l_1,l_2,l}\cdot  e^{2\pi i\Tr\bbeta z}\\
   = &\, (\alphaS-\alphabS)^{l_1-l}  \left(\begin{aligned}&
      \beta_{13}^{\frac{l_1-l_2+r_{\scriptscriptstyle \Lambda}}{2}}\bar{\beta}_{13}^{\frac{l_1-l_2-r_{\scriptscriptstyle \Lambda}}{2}}
   (\bar{\beta}_{13}\beta_{23}-\beta_{13}\bar{\beta}_{23})^{\frac{-l_1+l_2+l-t_k}{2}}
   (\beta_{12}-\bar{\beta}_{12})^{\frac{l_1+l_2-l-t_k}{2}}\\[1ex]
   &+\text{terms with degree $\geq 1$ in entries of $(\ltrans{z}-\bar{z})^{-1}$}\end{aligned}\right) e^{2\pi i\Tr\bbeta z}.
\end{aligned}
\end{equation}
Combing the equations \eqref{eq:FC-hol} and \eqref{eq:D-lead} and regrouping some terms prove the formula for $\left(E^\Sieg_{k,\chi}\right)_{\bbeta}\left(\begin{bmatrix}\fA\\&\nu\ltrans{\bar{\fA}}^{-1}\end{bmatrix}^p_f\right)(Y=0)$.
\end{proof}

\begin{prop}\label{prop:measE}
There exists a $p$-adic measure 
\[
   \cE\in\pzM eas\left(\bQ^\times\backslash \bA^\times_{\bQ,f} /U^p_N,\cM_{l_1,l_2}\left(K_{\GSp(4),f},\bar{\bQ}_p\right)\otimes \cM_l\left(K_{\GU(1,1),f},\bar{\bQ}_p\right)\right)
\] 
such that for all finite order Dirichlet characters $\chi:\bQ^\times\backslash\bA^\times_{\bQ}/U^p_N\bR^\times_{+}\ra\bC^\times$, and integers $k$ lying in the range \eqref{eq:krange}
\begin{align*}
   \cE\left((\chi|\bdot|^k)_{p_\adic}\right)=e^{\GSp(4)}_\ord e^{\GL(2)}_\ord\left(E^\Sieg_{k,\chi}\big|_{\GSp(4)\times\GU(1,1)}\right).
\end{align*}
Here the level group $K_{\GSp(4),f}, K_{\GL(2),f}$ are defined in \eqref{eq:KK} and the Eisenstein series $E^\Sieg_{k,\chi}$ is given in \eqref{eq:Esieg}.
\end{prop}
\begin{proof}

Thanks to the strong approximation for $\Sp(4)$ and $\SU(1,1)$,  we can pick $\nu_i\in\bA^{\times,p}_{\bQ,f}$, $i=1,\dots,c_1$, and $\nu'_j\in\bA^{\times,p}_{\bQ,f}$, $\fa_j\in \bA^{\times,p}_{\cK,f}$, $j=1,\dots,c_2$, such that each connected component of the Shimura variety for $\GSp(4)\times\GL(2)$ of principal level $N$ contains a point corresponding to $\left(\begin{bsm}\vphantom{\fa_j}\bid_2\\&\vphantom{\fa^{-1}_j}\nu_i\cdot\bid_2\end{bsm},\begin{bsm}\fa_j\\&\nu'_j\bar{\fa}^{-1}_j\end{bsm}\right)$ for some $i,j$. Taking the Fourier coefficients at these points gives an injection
\begin{equation}\label{eq:q-exp}
\begin{aligned}
   \varepsilon_{\qexp}: &\,\cM_{l_1,l_2}\left(K_{\GSp(4),f},\bar{\bQ}_p\right)
   \otimes\cM_l\left(K_{\GU(1,1),f},\bar{\bQ}_p\right)\\
   &\hspace{8em}\lhra \bigoplus_{\substack{1\leq i\leq c_1\\1\leq j\leq c_2}}\bar{\bQ}_p\llb\Sym_2(\bQ)_{\geq 0}\times \bQ_{\geq 0}\rrb.
\end{aligned}
\end{equation}

Denote by $\fX_{\mr{cl}}$ the set of pairs $(\chi,k)$ as in the statement in the proposition. Let 
\[
   A_{(i,j)}(\bbeta;\chi,k)=\left\{\begin{array}{ll}
   \left(E^\Sieg_{k,\chi}\right)_{\bbeta}\left(\begin{bsm}\bid_2\\&\fa_j\\&&\nu_i\cdot\bid_2\\&&&\nu'_j\bar{\fa}^{-1}\end{bsm}\right)(Y=0), &\nu_i=\nu'_j,\\
   0,&\nu_i\neq\nu'_j,
   \end{array}\right.
\]
and define the modified $A'_{(i,j)}(\bbeta;\chi,k)$ as
\[
   \left\{\begin{array}{ll} 
   A_{(i,j)}(\bbeta;\chi,k)\left(\frac{\beta_{12}-\bar{\beta}_{12}}{2}\det\begin{bmatrix}\bar{\beta}_{13}&\bar{\beta}_{23}\\ \beta_{13}&\beta_{23}\end{bmatrix}\det(\bbeta)^{-1}\right)^{-2k-\epsilon+1}, &k\geq 1,\\[2ex]
   A_{(i,j)}(\bbeta;\chi,k),&k\leq 0.
   \end{array}\right.
\]
(Note that with this modification, the formula for  $A'_{(i,j)}(\bbeta;\chi,k)$ is the one in Proposition~\ref{prop:EFC} with the last term replaced by simply $\left(\frac{\beta_{12}-\bar{\beta}_{12}}{2}\det\begin{bmatrix}\bar{\beta}_{13}&\bar{\beta}_{23}\\ \beta_{13}&\beta_{23}\end{bmatrix}\right)^{k+\frac{\epsilon}{2}-2}$, and has a uniform shape for both $k+\frac{\epsilon}{2}\geq \frac{1}{2}$ and $k+\frac{\epsilon}{2}\leq \frac{1}{2}$.) For $\beta_1\in\Sym_2(\bQ)$, $\beta_2\in\Sym_1(\bQ)$, define
\begin{align*}
   a_{(i,j)}(\beta_1,\beta_2;\chi,k)&=\sum_{\substack{\bbeta\in\Her_3(\cK)_{> 0}\\ \begin{bsm}\beta_{11}&\frac{\beta_{12}+\bar{\beta}_{12}}{2}\\ \frac{\beta_{12}+\bar{\beta}_{12}}{2}&\beta_{22}\end{bsm}=\beta_1,\,\beta_{33}=\beta_2}}  A_{(i,j)}(\bbeta;\chi,k),\\
   a'_{(i,j)}(\beta_1,\beta_2;\chi,k)&=\sum_{\substack{\bbeta\in\Her_3(\cK)_{> 0}\\ \begin{bsm}\beta_{11}&\frac{\beta_{12}+\bar{\beta}_{12}}{2}\\ \frac{\beta_{12}+\bar{\beta}_{12}}{2}&\beta_{22}\end{bsm}=\beta_1,\,\beta_{33}=\beta_2}} A'_{(i,j)}(\bbeta;\chi,k).
\end{align*}
It is not difficult to see that$a'_{(i,j)}(\beta_1,\beta_2;\chi,k)$ is interpolated by a $p$-adic measure
\[
   a'_{(i,j)}(\beta_1,\beta_2)\in \pzM eas\left((\bQ^\times\backslash\bA^\times_{\bQ,f}/U^p_N,\bar{\bQ}_p\right),
\]
({\it i.e.} for all $(\chi,k)\in \fX_{\mr{cl}}$, the value of $a'_{(i,j)}(\beta_1,\beta_2)$ at $(\chi|\bdot|^k)_{p\adic}$ equals $a'_{(i,j)}(\beta_1,\beta_2;\chi,k)$.)

The form $E^\Sieg_{k,\chi}$ is is obtained from applying differential operators to a holomorphic form on $\GU(3,3)$, and by our choice of the differential operator $D_{t_k,r_{\sLambda},l_1,l_2,l}$, it has weight $(l_1,l_2)$ for $\GSp(4)$ and weight $l$ for $\GU(1,1)$. Hence, we know that 
\begin{equation}\label{eq:EinN} 
 E^\Sieg_{k,\chi}\big|_{\GSp(4)\times\GU(1,1)}\in \bigcup_{e\geq 0} \begin{array}{l}\cN^e_{\GSp(4)}\left(K^p_{\GSp(4),f}K^1_{\GSp(4),p}(p^\infty),\rho_{l_1,l_2},\bar{\bQ}_p\right)\\
 \,\otimes \,\cN^e_{\GU(1,1)}\left(K^p_{\GU(1,1),f}K^1_{\GU(1,1),p}(p^\infty),\rho_{l},\bar{\bQ}_p\right)\end{array}.
\end{equation} 
At $\left(\begin{bsm}\vphantom{\fa_j}\bid_2\\&\vphantom{\fa^{-1}_j}\nu_i\cdot\bid_2\end{bsm}g(z_1),\begin{bsm}\fa_j\\&\nu_j\bar{\fa}^{-1}_j\end{bsm}g(z_2)\right)$, $z_1\in\bH_{\GSp(4)},z_2\in\bH_{\GU(1,1)}$, the $(\beta_1,\beta_2)$-th Fourier coefficient of $U^{\GSp(4)}_{p,m_1,m_2}U^{\GU(1,1)}_{p,m_3} \left(E^\Sieg_{k,\chi}\big|_{\GSp(4)\times\GU(1,1)}\right)$ equals
\begin{align*}
   &\sum_{x\in\bZ/p^{m_1-m_2}\bZ}
   a_{(i,j)}\left(\begin{bsm}p^{m_1-m_2}&\\ x&1\end{bsm}p^{m_2}\beta_1\begin{bsm}p^{m_1-m_2}&x\\&1\end{bsm},p^{m_3}\beta_2;\chi,k\right)\\
   &+\,\text{terms with degree $\geq 1$ in entries of $p^{m_2}\im(z_1)^{-1},p^{m_3}\im(z_2)^{-1}$ with $p$-integral coefficients}. 
\end{align*}
It follows from \eqref{eq:EinN} and Proposition~\ref{prop:e-conv} that the $p$-adic limit
\[
    \lim_{n\to\infty} \sum_{x\in\bZ/p^{n}\bZ}
   a_{(i,j)}\left(\begin{bsm}p^{n}&\\ x&1\end{bsm}p^{n}\beta_1\begin{bsm}p^{n}&x\\&1\end{bsm},p^{n}\beta_2;\chi,k\right)
\]
exists. Denoting this limit by $a_{\ord,(i,j)}(\beta_1,\beta_2;\chi,k)$, we have
\[
   \varepsilon_{\qexp}\left(e^{\GSp(4)}_\ord e^{\GL(2)}_\ord\left(E^\Sieg_{k,\chi}\big|_{\GSp(4)\times\GU(1,1)}\right)\right)_{\beta_1,\beta_2}=a_{\ord,(i,j)}(\beta_1,\beta_2;\chi,k).
\]
By noting that 
\[
   \frac{\beta_{12}-\bar{\beta}_{12}}{2}\det\begin{bmatrix}\bar{\beta}_{13}&\bar{\beta}_{23}\\ \beta_{13}&\beta_{23}\end{bmatrix}\det(\bbeta)^{-1}\equiv 1\mod p^n, 
   \quad \text{if } \bbeta\equiv\begin{bsm}0&\frac{\beta_{12}-\bar{\beta}_{12}}{2}&\beta_{13}\\-\frac{\beta_{12}-\bar{\beta}_{12}}{2}&0&\beta_{23}\\ \bar{\beta}_{13}&\bar{\beta}_{23}&0\end{bsm} \mod p^n,
\]
we have 
\begin{equation}\label{eq:a'-cong}
\begin{aligned}
   &\lim_{n\to\infty} \sum_{x\in\bZ/p^{n}\bZ}
   a'_{(i,j)}\left(\begin{bsm}p^{n}&\\ x&1\end{bsm}p^{n}\beta_1\begin{bsm}p^{n}&x\\&1\end{bsm},p^{n}\beta_2;\chi,k\right)\\
   =&\,\lim_{n\to\infty} \sum_{x\in\bZ/p^{n}\bZ}
   a_{(i,j)}\left(\begin{bsm}p^{n}&\\ x&1\end{bsm}p^{n}\beta_1\begin{bsm}p^{n}&x\\&1\end{bsm},p^{n}\beta_2;\chi,k\right)=a_{\ord,(i,j)}(\beta_1,\beta_2;\chi,k).
\end{aligned}
\end{equation}
Since $\fX_{\mr{cl}}$ is dense in $\Hom_\cont\left(\bQ^\times\backslash \bA^\times_{\bQ,f} /U^p_N,\bar{\bQ}^\times_p\right)$, from \eqref{eq:a'-cong}, we deduce that
\[
    \lim_{n\to\infty} \sum_{x\in\bZ/p^{n}\bZ}
   a'_{(i,j)}\left(\begin{bsm}p^{n}&\\ x&1\end{bsm}p^{n}\beta_1\begin{bsm}p^{n}&x\\&1\end{bsm},p^{n}\beta_2\right)
\]
exists in $\pzM eas\left(\bQ^\times\backslash \bA^\times_{\bQ,f}/U^p_N,\bar{\bQ}_p\right)$, and denoting the limit by $a_{\ord,(i,j)}(\beta_1,\beta_2)$, we have
\[
   \varepsilon_{\qexp}\left(e^{\GSp(4)}_\ord e^{\GL(2)}_\ord\left(E^\Sieg_{k,\chi}\big|_{\GSp(4)\times\GU(1,1)}\right)\right)_{\beta_1,\beta_2}=a_{\ord,(i,j)}(\beta_1,\beta_2)\left((\chi|\bdot|^k)_{p\adic}\right).
\]
(Note that the level of $E^\Sieg_{k,\chi}\big|_{\GSp(4)\times\GU(1,1)}$ depends on $\chi$, but its ordinary projection is invariant under the level group in \eqref{eq:KK} and belongs to the domain of $\varepsilon_{\qexp}$.)
\\

Putting all the $a_{\ord,(i,j)}(\beta_1,\beta_2)$ together gives rise to an element 
\begin{equation*}
   \cE_{\qexp}\in \pzM eas\left(\bQ^\times\backslash \bA^\times_{\bQ,f}/U^p_N,\bar{\bQ}_p\llb\Sym_2(\bQ)_{>0}\times\bQ_{>0}\rrb\right)
\end{equation*}
satisfying that for all $(\chi,k)\in \fX_{\mr{cl}}$,
\[
    \cE_{\qexp}\left((\chi|\bdot|^k)_{p\adic}\right)=\varepsilon_{\qexp}\left(e^{\GSp(4)}_\ord e^{\GL(2)}_\ord\left(E^\Sieg_{k,\chi}\big|_{\GSp(4)\times\GU(1,1)}\right)\right).
\]
Because the image of $\varepsilon_{\qexp}$ in \eqref{eq:q-exp} is finite dimensional, it is closed in $\bigoplus_{\substack{1\leq i\leq c_1\\1\leq j\leq c_2}}\bar{\bQ}_p\llb\Sym_2(\bQ)_{\geq 0}\times \bQ_{\geq 0}\rrb$. The density of $\fX_{\mr{cl}}$ in $\Hom_\cont\left(\bQ^\times\backslash \bA^\times_{\bQ,f} /U^p_N,\bar{\bQ}^\times_p\right)$ imply that $\cE_{\qexp}$ takes values inside the image of $\varepsilon_{\qexp}$. Thus, there exists 
\[
   \cE\in\pzM eas\left(\bQ^\times\backslash \bA^\times_{\bQ,f} /U^p_N,\cM_{l_1,l_2}\left(K_{\GSp(4),f},\bar{\bQ}_p\right)\otimes \cM_l\left(K_{\GU(1,1),f},\bar{\bQ}_p\right)\right)
\] 
such that $\varepsilon_{\qexp}(\cE)=\cE_{\qexp}$. This $\cE$ is the desired $p$-adic measure.

\end{proof}

\section{The $p$-adic $L$-function}

\subsection{The modified Euler factors for $p$-adic interpolations}
With $\Pi,\pi$ as in \S\ref{sec:ord-at-p} and characters $\eta_{\sPi_p,1},\eta_{\sPi_p,2},\eta_{\sPi_p,3},\eta_{\pi_p,1},\eta_{\pi_p,2}:\bQ^\times_p\ra\bC^\times$ as in \eqref{eq:Ind-eta}\eqref{eq:xi-val}, unfolding the definitions in \cite{CoPerrin,Coates}, we have the the modified Euler factor at $p$ and $\infty$:
\begin{align}
   \label{eq:Ep} E_p(s,\tilde{\Pi}\times\tilde{\pi}\times\chi)=&\,\gamma_p\left(s,\chi_p\eta^{-1}_{\pi_p,1}\eta^{-1}_{\sPi_p,1}\right)^{-1} \gamma_p\left(s,\chi_p\eta^{-1}_{\pi_p,1}\eta^{-1}_{\sPi_p,2}\right)^{-1}  
    \gamma_p\left(s,\pi^\vee_p\times\chi_p\eta^{-1}_{\sPi_p,3}\right)^{-1},\\
   \label{eq:Einf} E_\infty(s,\tilde{\Pi}\times\tilde{\pi}\times\chi)=&\,e^{-(4s+l_1+l_2+l)\cdot \frac{\pi i}{2}}\Gamma_\bC\left(s+\frac{l_1+l_2+l}{2}-2\right)\Gamma_\bC\left(s+\frac{l_1+1_2-l}{2}-1\right)\\
   &\times\Gamma_\bC\left(s+\frac{-l_1+l_2+l}{2}-1\right)\Gamma_\bC\left(s+\frac{l_1-1_2+l}{2}\right),
\end{align}
where $\Gamma_\bC(s)=2(2\pi)^{-s}\Gamma(s)$.

\subsection{The main theorem}
Let $\Pi,\pi$ be as in \S\ref{sec:ord-at-p}. For a $p$-ordinary Siegel modular form $\varphi$ with $U^{\GSp(4)}_{p,m_1,m_2}$-eigenvalue $\lambda_{p,1}(\varphi)^{m_1}\lambda_{p,2}(\varphi)^{m_2}$, and a $p$-ordinary modular form $f\in\pi$ with $U^{\GL(2)}_{p,1}$-eigenvalue $\lambda_p(f)$, we define the modified Bessel period $\bes^\dagger_{\bS,\Lambda}(\varphi)$ and modified Petersson norms $\bfP(\varphi,\varphi)$, $\bfP(f,f)$ as follows.
\begin{align}
     \bes^\dagger_{\bS,\Lambda}(\varphi)=&\,
     \lambda_{p,1}(\varphi)^{-m_1}\lambda_{p,2}(\varphi)^{-m_2}
     \cdot \bes_{\bS,\Lambda}\left(\begin{bsm}p^{m_1}\\&p^{m_2}\\&&p^{-m_1}\\&&&p^{-m_2}\end{bsm}_p\varphi\right),\label{eq:bes-dagger}\\    
   \bfP(\varphi,\varphi)=&\,\lambda_{p,1}(\varphi)^{-m_1}\lambda_{p,2}(\varphi)^{-m_2} \label{eq:Pet} \\
    &\times\int_{[\GSp(4)]} \varphi(g)\,\varphi\left(g\begin{bmatrix}&\bid_2\\\bid_2\end{bmatrix}_{p\infty}\begin{bsm}p^{m_1}\\&p^{m_2}\\&&p^{-m_1}\\&&&p^{-m_2}\end{bsm}_p\right)\cdot \omega_\Pi(\nu_g)^{-1} \,dg,\nonumber
\end{align}
with $m_1\gg m_2\gg 0$,  and 
\begin{equation}\label{eq:Pet2}
   \bfP(f,f)=\lambda_p(f)^{-m_3}\int_{[\GL(2)]} f(g)\,f\left(g\begin{bmatrix}&1\\1\end{bmatrix}_{p\infty} \begin{bmatrix}p^{m_3}\\&p^{-m_3}\end{bmatrix}_p\right)\cdot\omega_\pi(\det g)^{-1}\,dg,
\end{equation}
with $m_3\gg 0$. (By (ii) of Proposition~\ref{prop:Zord}, we know that the right hand side of \eqref{eq:bes-dagger} does not depend on the choice of $m_1,m_2$ as long as $m_1-m_2,m_2$ are sufficiently large. The independence of the right hand side of \eqref{eq:Pet} and \eqref{eq:Pet2} on $m_1,m_2,m_3$ for $m_1\gg m_2\gg 0$, $m_3\gg0$ is easy to check. Also, by \cite[Th{\`e}or{\'e}m on page 91]{MVW} and \cite{Arthur-GSp4,Gee-Taibi}, we know that $\bar{\Pi}=\Pi\otimes\omega^{-1}_\Pi\circ\nu$.)

\begin{thm}\label{thm:main}
Let $\Pi,\pi$ and and $\epsilon\in\{0,1\}$ be as in \S\ref{sec:ord-at-p}. Fix the auxiliary data $\bS,\Lambda,N$ as in \S\ref{sec:aux}, and  take holomorphic $p$-ordinary forms 
$\varphi_\ord\in \Pi^{K_{\GSp(4),f}}$, $f_\ord\in \pi^{K_{\GL(2),f}}$ with algebraic Fourier coefficients (where the level group is defined in \eqref{eq:KK}).
There exists a $p$-adic measure 
\[
   \cL^{N}_{\Pi,\pi}\in\pzM eas\left(\bQ^\times\backslash \bA^\times_{\bQ,f} /U^p_N,\bar{\bQ}_p\right)
\]
satisfying the interpolation property: For all finite order characters $\chi:\bQ^\times\backslash \bA^\times_{\bQ}/U^p_N\ra\bC^\times$ and integers $k$ such that $s=k+\frac{\epsilon}{2}$ is critical for $L(s,\tilde{\Pi}\times\tilde{\pi}\times\chi)$ (or equivalently  satisfying \eqref{eq:krange}),
\begin{equation}\label{eq:interp}
\begin{aligned}
   \cL^N_{\Pi,\pi}\left((\chi|\bdot|^k)_{p_\adic}\right)
   = &\,\frac{\bes^\dagger_{\bS,\Lambda}\left(\varphi_{\ord}\right)
   \,\whi_\bbc(f_\ord)}{\bfP(\varphi_\ord,\varphi_\ord)\,\bfP(f_\ord,f_\ord)} \\
     &\times  I_\infty(k,\Pi_\infty,\pi_\infty,\Lambda_\infty)\cdot
   E_p\left(k+\frac{\epsilon}{2},\tilde{\Pi}\times\tilde{\pi}\times\chi\right)
   \cdot L^{Np\infty}\left(k+\frac{\epsilon}{2},\tilde{\Pi}\times\tilde{\pi}\times\chi\right),
\end{aligned}
\end{equation}
with the $E_p$-factor defined in \eqref{eq:Ep} and 
\begin{equation}\label{eq:azI}
\begin{aligned}
   &I_\infty(k,\Pi_\infty,\pi_\infty,\Lambda_\infty)\\
   =&\,(-1)^k\, \bbc^{\frac{2k-\epsilon+l+l_1+l_2}{2}} \left(-(\alphaS-\alphabS)^2\right)^{\frac{2k-\epsilon+l-r_{\sLambda}}{2}}
   \cdot \frac{\prod_{j=0}^2\Gamma\left(s+\frac{t_k+3}{2}-j\right)}{2^{3(s+\frac{t_k-1}{2})}\pi^{3(s+\frac{t_k+1}{2})}i^{t_k}}\\
    &\times \left.  Z_\infty\left(D_{t_k,r_{\sLambda},l_1,l_2,l}\cdot \dsec^{t_k}_\infty(s,\chi,\Xi),\Bes^{\cD_{l_1,l_2}}_{\bS,\Lambda_{\infty}}\left(\begin{bsm}&\bid_2\\ \bid_2\end{bsm}\cdot \varphi_{\ord,\infty}\right),\Whi^{\cD_l,\Xi^{-c}_\infty}_{\bbc}\left(\begin{bsm}&1\\1\end{bsm}\cdot f_{\ord,\infty}\right)\right)\right|_{s=\frac{2k+\epsilon-1}{2}}.
\end{aligned}
\end{equation}
(See \S\ref{sec:sec-inf} for some notations. Since $\Whi^{\cD_l}_{\bbc}\left(f_\infty\right)$ is supported on $\GL(2,\bR)_+$ \eqref{eq:Wfinfty}, we know that the archimedean integral above is independent of $\chi$.)
\end{thm}

\begin{proof}
Let $\varphi'_\ord = \Pi\left(\begin{bmatrix}&-\bid_2\\ \bid_2\end{bmatrix}_N\right)\varphi_\ord$. By considering the transfer of $\Pi$ to $\GL(5)$ \cite[Theorem~8.1.2]{Gee-Taibi} and using the nonvanishing of $L$-functions at $s=1$ for cuspidal automorphic representations of $\GL(n)$ \cite{JS-nonvan} and $\Pi$ being tempered, we can deduce that the degree $5$ (partial) $L$-function of $\Pi$ twisted by a finite-order character is nonvanishing at $s=1$. Then the argument in \cite[page 478-480]{GaSH} shows the algebraicity of $\frac{\left<\,\cdot\,,\varphi'_\ord\right>}{\bfP(\varphi_\ord,\varphi_\ord)}$. Also, we know the algebraicity of $\frac{\left<\,\cdot\,,f_\ord\right>}{\bfP(f_\ord,f_\ord)}$. Hence, we can define the linear functional 
\[
   \sL_{\varphi_\ord,f_\ord}:\cM_{l_1,l_2}\left(K_{\GSp(4),f}(N),\bar{\bQ}\right)
   \otimes_{\bar{\bQ}_p}\cM_l\left(K_{\GL(2),f}(N),\bar{\bQ}\right)\lra\bar{\bQ}
\]
as
{\small
\begin{align*}
   \sL_{\varphi_\ord,f_\ord}(\varphi\otimes f)=&\,\bfP(\varphi_\ord,\varphi_\ord)^{-1}\bfP(f_\ord,f_\ord)^{-1}\cdot \lambda_{p,1}(\varphi_{\ord})^{-m_1}\lambda_{p,2}(\varphi_{\ord})^{-m_2} \lambda_p(f_\ord)^{-m_3}\\
   &\hspace{-1em}\times\left<\varphi\otimes f, \Pi\left(\begin{bsm}&\bid_2\\ \bid_2\end{bsm}_{p\infty}\begin{bsm}p^{m_1}\\&p^{m_2}\\&&p^{-m_1}\\&&&p^{-m_2}\end{bsm}_p\right)\varphi'_\ord \otimes \pi\left(\begin{bsm}&1\\1\end{bsm}_{p\infty}\begin{bsm}p^{m_3}\\&p^{-m_3}\end{bsm}\right)f_\ord\right>.
\end{align*}
}
\hspace{-.5em}with $m_1\gg m_2\gg 0$, $m_3\gg 0$.  (One can check that the right hand side is independent of $m_1,m_2,m_3$ once $m_1-m_2,m_2,m_3$ are sufficiently large.) We denote its $\bar{\bQ}_p$-linear extension to $\cM_{l_1,l_2}\left(K_{\GSp(4),f},\bar{\bQ}_p\right) \otimes_{\bar{\bQ}_p}\cM_l\left(K_{\GL(2),f},\bar{\bQ}_p\right)$ also by $\sL_{\varphi_\ord,f_\ord}$. Since $\sL_{\varphi_\ord,f_\ord}$ is a $\bar{\bQ}_p$-linear functional on a finite dimensional $\bar{\bQ}_p$-vector space, it is $p$-adically continuous, so we can apply it to the $p$-adic measure $\cE$ in Proposition~\ref{prop:measE} to get
\[
   \mu^N_{\Pi,\pi}=\sL_{\varphi_\ord,f_\ord}(\cE)\in \pzM eas \left(\bQ^\times\backslash \bA^\times_{\bQ,f} /U^p_N,\bar{\bQ}_p\right).
\]

Next, we compute the interpolations of $\mu^N_{\Pi,\pi}$. By the integral representation stated in Theorem~\ref{thm:Furusawa}, and the computation in Proposition~\ref{prop:Z-vol} of local zeta integrals for the types of sections chosen in \S\ref{sec:fN} for $v\,|\,N$, we have
{\small
\begin{equation}\label{eq:mu-eval}
\begin{aligned}
   \mu^N_{\Pi,\pi}\left((\chi|\bdot|^k)_{p_\adic}\right)
   =&\,\frac{\bes_{\bS,\Lambda}\left(\varphi_\ord\right) 
   \whi_\bbc(f_\ord)}{\bfP(\varphi_\ord,\varphi_\ord)\bfP(f_\ord,f_\ord)}
   \cdot \prod_{v\,|\,N} \frac{|\bbc^2 N^7|_v}{ (1-|\varpi_v|_v)^2(1-|\varpi_v|^2_v)} \\
   &\times I_p\left(\frac{2k+\epsilon-1}{2}\right)
   I_\infty\left(\frac{2k+\epsilon-1}{2}\right)
   \cdot L^{Np\infty}\left(k+\frac{\epsilon}{2},\tilde{\Pi}\times\tilde{\pi}\times\chi\right)
\end{aligned}
\end{equation}
}
\hspace{-.5em}with
{\small
\begin{align*}
   I_p(s)&=\frac{Z_p\left(\dsec^\bc_{p,\Schw_{\chi^\circ_p,\Lambda^\circ_p,\eta^\circ_{\sPi_p},\eta^\circ_{\pi_p}}}(s,\chi,\Xi),\Bes^{\Pi_p}_{\bS,\Lambda_p}(\varphi_{\ord,p,m_1,m_2}),\Whi^{\pi_p,\Upsilon_p}_\bbc(f_{\ord,p,m_3})\right)}{\lambda_1(\varphi_{\ord,p})^{m_1}\lambda_2(\varphi_{\ord,p})^{m_2}\lambda(f_{\ord,p})^{m_3}\cdot\bes^{\Pi_p}_{\bS,\Lambda_p}\left(\varphi_{\ord,p}\right)}\\
   I_\infty(s)&=2^{-3(s+\frac{t_k-1}{2})}i^{-t_k}\pi^{-3(s+\frac{t_k+1}{2})}\prod_{j=0}^2\Gamma\left(s+\frac{t_k+3}{2}-j\right)\\
  &\quad\times Z_\infty\left(D_{t_k,r_{\sLambda},l_1,l_2,l}\cdot \dsec^{t_k}_\infty(s,\chi,\Xi),\Bes^{\cD_{l_1,l_2}}_{\bS,\triv}\left(\begin{bsm}&\bid_2\\ \bid_2\end{bsm}\cdot \varphi_{\ord,\infty}\right),\Whi^{\cD_l,\Xi^{-c}_\infty}_{\bbc}\left(\begin{bsm}&1\\1\end{bsm} \cdot f_{\ord,\infty}\right)\right).
\end{align*} 
}

We need to compute $I_p(s)$. It follows from Proposition~\ref{prop:ZpWals} and our choice of $\Schw_{\chi^\circ_p,\Lambda^\circ_p,\eta^\circ_{\sPi_p},\eta^\circ_{\pi_p}}$ in \S\ref{sec:secp} that
{\small
\begin{equation}\label{eq:Ips}
\begin{aligned}
   I_p(s)=&\,\chi_p(-1)\,\chi^{-1}_p\eta_{\pi_p,1}\eta_{\sPi_p,3}(\bbc)|\bbc|^{-s}_p\cdot \whi^{\pi_p}_\bbc(f_{\ord,p}) 
   \cdot E_p\left(s+\frac{1}{2},\tilde{\Pi}\times\tilde{\pi}\times\chi\right)\\
   &\times \lambda_1(\varphi_{\ord,p})^{-m_1}\lambda_2(\varphi_{\ord,p})^{-m_2}\left(\eta_{\sPi_p,3}(p)p^{3/2}\right)^{-m_1-m_2}\bes^{\Pi_p}_{\bS,\Lambda_p}\left(\varphi_{\ord,p}\right)^{-1}
   \cdot I_{p,1}(s)\,I_{p,2}(s),
\end{aligned}
\end{equation}
}
\hspace{-.5em}with
{\small
\begin{align*}
   I_{p,1}(s)&=\int_{\bQ^\times_p} |\nu|^{-s+\frac{1}{2}}_p\cdot \chi^{-1}_p\eta_{\pi_p,2}\eta_{\sPi_p,3}(\nu)\cdot \cF \Schw_{p,1,\alt}\left(\nu \right) \,d^\times\nu,\\
   I_{p,2}(s)&=\int_{\GL(2,\bQ_p)} (\cF\Schw_{p,0})\left(T\right)\cdot |\det T|^{-s+1}_p \chi^{-1}_p\eta_{\pi_p,1}\Lambda_{\bQ,p}(\det T)\cdot \wal^{\sigma_p}_{\bS,\Lambda_p}\left(\sigma_p\left(T^{-1}\begin{bmatrix}p^{m_1}\\&p^{m_2}\end{bmatrix} \right)w_{\sigma_p}\right) \,dT
\end{align*}
}
\hspace{-.5em}where
{\small
\begin{align}
   \label{eq:altSchwp} \cF \Schw_{p,1,\alt}(\nu)=&\,\chi^\circ_p\eta^{\circ-1}_{\pi_p,2}\eta^{\circ-1}_{\sPi_p,3}(\nu),\\
   \label{eq:0Schwp}  \cF\Schw_{p,0}\left(\begin{bmatrix}\fw_{13}&\bar{\fw}_{13}\\ \fw_{23}&\bar{\fw}_{23}\end{bmatrix}\begin{bmatrix}\alphaS&1\\ \alphabS&1\end{bmatrix}\right)
   = &\,\mathds{1}_{(\alphaS-\alphabS)\cO_{\cK,p}}(\fw_{23})
    \cdot \Lambda^{\circ}_p(\fw_{13})  \,\,\eta^{\circ-1}_{\sPi_p,2}(\fw_{13}\bar{\fw}_{13})     \\
    \nonumber&\times  \chi^\circ_p\eta^{\circ-1}_{\pi_p,1}\eta^{\circ-1}_{\sPi_p,1}\left(\frac{\bar{\fw}_{13}\fw_{23}-\fw_{13}\bar{\fw}_{23}}{\alphaS-\alphabS}\right).
\end{align}
}
\hspace{-.5em}It is easy to see that
{\small
\begin{equation}\label{eq:Ip1s}
   I_{p,1}(s)=1.
\end{equation}
}
\hspace{-.5em}Writing $T=\begin{bmatrix}\fy_1&\fy_2\\ \bar{\fy}_1&\bar{\fy}_2\end{bmatrix}^{-1}\begin{bmatrix}\alphaS&1\\ \alphabS&1\end{bmatrix}$ with $\fy_1,\fy_2\in\cK_p$, 
{\small
\begin{equation}\label{eq:F-Schw-p}
\begin{aligned}
    &(\cF\Schw_{p,0})\left(T\right)\cdot |\det T|^{-s+1}_p \chi^{-1}_p\eta_{\pi_p,1}\Lambda_{\bQ,p}(\det T)\\
     =&\,\chi^{-1}_p\eta_{\pi_p,1}\Lambda_{\bQ,p}\left(-(\alphaS-\alphabS)^2\right)\left|-(\alphaS-\alphabS)^2\right|^{-s+1}_p \\
    &\times \mathds{1}_{\cO_{\cK,p}}(\fy_1)
    \cdot \Lambda^{\circ-1}_p\left((\alphaS-\alphabS)\fy_2\right)
    \,\eta^{\circ-1}_{\sPi_p,1}\left((\alphaS-\alphabS)^{-1}(\frac{\fy_1}{\fy_2}-\frac{\bar{\fy}_1}{\bar{\fy}_2})\right).
\end{aligned}
\end{equation}
}
\hspace{-.5em}We can also write 
{\small
\begin{equation}\label{eq:ISy}
\begin{aligned}
   \wal^{\sigma_p}_{\bS,\Lambda_p}\left(\sigma_p\left(T^{-1}\begin{bmatrix}p^{m_1}\\&p^{m_2}\end{bmatrix} \right)w_{\sigma_p}\right)
   &= \wal^{\sigma_p}_{\bS,\Lambda_p}\left(\sigma_p\left(\begin{bmatrix}\alphaS&1\\ \alphabS&1\end{bmatrix}^{-1}\begin{bmatrix}\fy_1&\fy_2\\ \bar{\fy}_1&\bar{\fy}_2\end{bmatrix}\begin{bmatrix}p^{m_1}\\&p^{m_2}\end{bmatrix} \right)w_{\sigma_p}\right)\\
   &=\Lambda_p(\fy_2)\cdot \wal^{\sigma_p}_{\bS,\Lambda_p}\left(\sigma_p\left(\begin{bmatrix}\alphaS&1\\ \alphabS&1\end{bmatrix}^{-1}\begin{bmatrix}\fy_1/\fy_2&1\\ \bar{\fy}_1/\bar{\fy}_2&1\end{bmatrix}\begin{bmatrix}p^{m_1}\\&p^{m_2}\end{bmatrix} \right)w_{\sigma_p}\right)\\
   &=\Lambda_p(\fy_2)\cdot \wal^{\sigma_p}_{\bS,\Lambda_p}\left(\sigma_p\left(\begin{bmatrix}\frac{\fy_1/\fy_2-\bar{\fy}_1/\bar{\fy}_2}{\alphaS-\alphabS}&0\\ \frac{\alphaS(\bar{\fy}_1/\bar{\fy}_2)-\alphabS(\fy_1/\fy_2)}{\alphaS-\alphabS}&1\end{bmatrix}\begin{bmatrix}p^{m_1}\\&p^{m_2}\end{bmatrix} \right)w_{\sigma_p}\right)\\
\end{aligned}
\end{equation}
}
\hspace{-.5em}In our setting here, we can use
\begin{align}
   \label{eq:sigma-Ind}\sigma_p&=\text{a quotient of }\Ind^{\GL(2)}_{B} (\eta_{\sPi_p,2},\eta_{\sPi_p,1}),
\end{align}
and $w_{\sigma_p}\in \sigma_p$ is a multiple of the projection of the section
\begin{equation}\label{eq:w-Ind}
   h=\begin{bmatrix}a&b\\c&d\end{bmatrix}\longmapsto \mathds{1}_{\bZ_p}(c^{-1}d)\cdot |(\det h)c^{-2}|^{1/2}_p \, \eta_{\sPi_p,1}(c)\,\eta_{\sPi,2}((\det h)c^{-1})
\end{equation}
Then with $m_1\gg m_2$, from \eqref{eq:F-Schw-p}\eqref{eq:ISy}, we see that when $\cF\Schw_{p,0}(T)\neq 0$,
{\small
\begin{equation}\label{eq:plgu2}
\begin{aligned}
  &\wal^{\sigma_p}_{\bS,\Lambda_p}\left(\sigma_p\left(T^{-1}\begin{bmatrix}p^{m_1}\\&p^{m_2}\end{bmatrix} \right)w_{\sigma_p}\right)\\
   =&\, \Lambda_p(\fy_2) \,\eta_{\sPi_p,1}\left((\alphaS-\alphabS)^{-1}(\frac{\fy_1}{\fy_2}-\frac{\bar{\fy}_1}{\bar{\fy}_2})\right)\cdot \wal^{\sigma_p}_{\bS,\Lambda_p}\left(\sigma_p\left(\begin{bmatrix}p^{m_1}\\&p^{m_2}\end{bmatrix} \right)w_{\sigma_p}\right)\\
   =&\,\Lambda_p(\fy_2) \,\eta_{\sPi_p,1}\left((\alphaS-\alphabS)^{-1}(\frac{\fy_1}{\fy_2}-\frac{\bar{\fy}_1}{\bar{\fy}_2})\right)\cdot \left(\eta_{\sPi_p,3}(p)p^{3/2}\right)^{m_1+m_2}\bes^{\Pi_p}_{\bS,\Lambda_p}\left(\Pi_p\begin{bsm}p^{m_1}\\&p^{m_2}\\&&p^{-m_1}\\&&&p^{-m_2}\end{bsm}\varphi_{\ord,p}\right).
\end{aligned}
\end{equation}
}
\hspace{-.5em}Then with \eqref{eq:F-Schw-p}\eqref{eq:plgu2}, we have
{\small
\begin{equation}\label{eq:Ip2s-1}
\begin{aligned}
   I_{p,2}(s)=&\,\frac{1-\left(\frac{\cK}{p}\right)p^{-1}}{1+p^{-1}}\,\left|\frac{\alphaS-\alphabS}{\sqrt{\disc(\cK/\bQ)}}\right|_p
   \cdot \chi^{-1}_p\eta_{\pi_p,1}\left(-(\alphaS-\alphabS)^2\right)\Lambda_p(-\alphaS+\alphabS) \left|-(\alphaS-\alphabS)^2\right|^{-s+1}_p\\
   &\times
      \left(\eta_{\sPi_p,3}(p)p^{3/2}\right)^{m_1+m_2}\bes^{\Pi_p}_{\bS,\Lambda_p}\left(\Pi_p\begin{bsm}p^{m_1}\\&p^{m_2}\\&&p^{-m_1}\\&&&p^{-m_2}\end{bsm}\varphi_{\ord,p}\right).
\end{aligned}
\end{equation}
}
\hspace{-.5em}Combining \eqref{eq:Ips}\eqref{eq:Ip1s}\eqref{eq:Ip2s-1} gives
{\small
\begin{equation}\label{eq:Ips2}
\begin{aligned}
   I_p(s)=&\,\frac{1-\left(\frac{\cK}{p}\right)p^{-1}}{1+p^{-1}}\,\left|\frac{\alphaS-\alphabS}{\sqrt{\disc(\cK/\bQ)}}\right|_p
   \cdot  \chi_p(-1)\,\chi^{-1}_p\eta_{\pi_p,1}\eta_{\sPi_p,3}(\bbc) |\bbc|^{-s}_p \\
   &\times \chi^{-1}_p\eta_{\pi_p,1}\left(-(\alphaS-\alphabS)^2\right)\Lambda_p(-\alphaS+\alphabS)\left|-(\alphaS-\alphabS)^2\right|^{-s+1}_p \cdot  E_p\left(s+\frac{1}{2},\tilde{\Pi}\times\tilde{\pi}\times\chi\right)\\
   &\times  \frac{\bes^{\Pi_p}_{\bS,\Lambda_p}\left(\Pi_p\begin{bsm}p^{m_1}\\&p^{m_2}\\&&p^{-m_1}\\&&&p^{-m_2}\end{bsm}\varphi_{\ord,p}\right)
   \,\whi^{\pi_p}_\bbc\left(f_{\ord,p}\right)}{\lambda_1(\varphi_{\ord,p})^{m_1}\lambda_2(\varphi_{\ord,p})^{m_2} 
   \,\bes^{\Pi_p}_{\bS,\Lambda_p}\left(\varphi_{\ord,p}\right)^{\phantom{M^M}}}.
\end{aligned}
\end{equation}
}
\hspace{-.5em}We can write \eqref{eq:Ips2}  as
{\small
\begin{align*}
   I_p\left(s\right)=&\,C_{k,\chi,\Pi,\pi}\cdot (-1)^k\, \bbc^{\frac{2k+\epsilon+l+l_1+l_2}{2}} \left(-(\alphaS-\alphabS)^2\right)^{\frac{2k+\epsilon+l-r_{\sLambda}}{2}}\\
   &\times E_p\left(s+\frac{1}{2},\tilde{\Pi}\times\tilde{\pi}\times\chi\right)
  \cdot \frac{\bes^{\Pi_p}_{\bS,\Lambda_p}\left(\Pi_p\begin{bsm}p^{m_1}\\&p^{m_2}\\&&p^{-m_1}\\&&&p^{-m_2}\end{bsm}\varphi_{\ord,p}\right)
   \,\whi^{\pi_p}_\bbc\left(f_{\ord,p}\right)}{\lambda_1(\varphi_{\ord,p})^{m_1}\lambda_2(\varphi_{\ord,p})^{m_2} 
   \,\bes^{\Pi_p}_{\bS,\Lambda_p}\left(\varphi_{\ord,p}\right)^{\phantom{M^M}}}.
\end{align*}
}
\hspace{-.5em}with 
{\small
\begin{align*}
   C_{k,\chi,\Pi,\pi}(s)=&\,|\bbc|^{k-s}_p|(\alphaS-\alphabS)^2|^{k-s+1}_p\cdot \prod_{v\,|\,N} \frac{|\bbc^2 N^7|_v}{ (1-|\varpi_v|_v)^2(1-|\varpi_v|^2_v)}\cdot \frac{1-\left(\frac{\cK}{p}\right)p^{-1}}{1+p^{-1}}\,\left|\frac{\alphaS-\alphabS}{\sqrt{\disc(\cK/\bQ)}}\right|_p \\
   &\times\chi_p(-1)(-1)^k 
   \cdot \left(\chi^{-1}_p|\bdot|^{-k}_p\eta_{\pi_p,1}\eta_{\sPi_p,3}\right)(\bbc)\,\bbc^{-\frac{2k+\epsilon+l+l_1+l_2}{2}} \\
   &\times \left(\chi^{-1}_p|\bdot|^{-k}_p\eta_{\pi_p,1}\right)\left(-(\alphaS-\alphabS)^2\right)\Lambda_p(-\alphaS+\alphabS) 
   \left(-(\alphaS-\alphabS)^2\right)^{-\frac{2k+\epsilon+l-r_{\sLambda}}{2}}.
\end{align*}
}
\hspace{-.5em}Plugging this formula into \eqref{eq:mu-eval}, we get
{\small
\begin{align*}
   &\mu^N_{\Pi,\pi}\left((\chi|\bdot|^k)_{p_\adic}\right)\\
   =&\, C_{k,\chi,\Pi,\pi}\left(\frac{2k+\epsilon-1}{2}\right)
   \cdot \frac{\bes^\dagger_{\bS,\Lambda}\left(\varphi_{\ord}\right)
   \,\whi^{\pi}_\bbc\left(f_{\ord}\right)}{\bfP(\varphi_\ord,\varphi_\ord)\bfP(f_\ord,f_\ord)} \\
   &\times I_\infty(k,\Pi_\infty,\pi_\infty,\Lambda_\infty)
   \cdot E_p\left(\frac{2k+\epsilon-1}{2},\tilde{\Pi}\times\tilde{\pi}\times\chi\right)
   \cdot L^{Np\infty}\left(\frac{2k+\epsilon-1}{2},\tilde{\Pi}\times\tilde{\pi}\times\chi\right)
\end{align*}
}
\hspace{-.5em}with 
\begin{align*}
   I_\infty(k,\Pi_\infty,\pi_\infty,\Lambda_\infty)= (-1)^k\, \bbc^{\frac{2k+\epsilon+l+l_1+l_2}{2}} \left(-(\alphaS-\alphabS)^2\right)^{\frac{2k+\epsilon+l-r_{\sLambda}}{2}}\cdot I_\infty\left(\frac{2k+\epsilon-1}{2}\right).
\end{align*}
One can see that $C_{k,\chi,\Pi,\pi}\left(\frac{2k+\epsilon-1}{2}\right)$ is interpolated by a $p$-adic measure $\mu\in \pzM eas\left(\bQ^\times\backslash\bA^\times_{\bQ,f}/U^p_N,\bar{\bQ}_p\right)$ and the convolution with $\mu$ gives an invertible map from $\pzM eas\left(\bQ^\times\backslash\bA^\times_{\bQ,f}/U^p_N,\bar{\bQ}_p\right)$ to itself ({\it i.e.} the $p$-adic Mellin transform of $\mu$ belongs to $\left(\cO_{\bar{\bQ}_p}\llb \bQ^\times\backslash\bA^\times_{\bQ,f}/U^p_N\rrb\otimes \bQ\right)^\times$). Therefore, we can take $\cL^N_{\Pi,\pi}\in \pzM eas\left(\bQ^\times\backslash\bA^\times_{\bQ,f}/U^p_N,\bar{\bQ}_p\right)$ such that $\mu^N_{\Pi,\pi}=\mu\ast \cL^N_{\Pi,\pi}$.
\end{proof}

This proves part (i) of Theorem~\ref{thm:intro}. Next, we prove parts (ii) and (iii).

\begin{prop}
In the setting of Theorem~\ref{thm:main} and  further assuming that $l_1=l_2=l$, we have the more precise interpolation formula:
\begin{equation}\label{eq:int-sc}
\begin{aligned}
   &\cL^N_{\Pi,\pi}\left((\chi|\bdot|^k)_{p_\adic}\right)\\
   = &\,|\bbc^2(\alphaS-\alphabS)|\, 2^{-6}i^{\epsilon}\cdot \frac{\bes^\dagger_{\bS,\Lambda}\left(\varphi_{\ord}\right)
   \,\whi_\bbc(f_\ord)}{\bfP(\varphi_\ord,\varphi_\ord)\,\bfP(f_\ord,f_\ord)} \\
   &\times 
   E_p\left(k+\frac{\epsilon}{2},\tilde{\Pi}\times\tilde{\pi}\times\chi\right)
   E_\infty\left(k+\frac{\epsilon}{2},\tilde{\Pi}\times\tilde{\pi}\times\chi\right)
   \cdot  L^{Np\infty}\left(k+\frac{\epsilon}{2},\tilde{\Pi}\times\tilde{\pi}\times\chi\right).
\end{aligned}
\end{equation}
with the $E_p$, $E_\infty$ factors defined in \eqref{eq:Ep}\eqref{eq:Einf}.
\end{prop}

\begin{proof}
When $l_1=l_2=l$, we can compute $I_\infty(s)$ as follows. In this case, the $(\bS,\Lambda)$-Bessel period vanishes on $\Pi$  unless $\Lambda_\infty=\triv$, so we can assume that $\Lambda_\infty=\triv$. Let
\[
   D'_{t_k,l}=(2\pi i)^{\frac{3t_k-3l}{2}}\cdot 
   \det\begin{bmatrix}
   \mu^+_{\GU(3,3),11}& \mu^+_{\GU(3,3),12}&\mu^+_{\GU(3,3),13}\\
   \mu^+_{\GU(3,3),21}& \mu^+_{\GU(3,3),22}&\mu^+_{\GU(3,3),23}\\
   \mu^+_{\GU(3,3),31}& \mu^+_{\GU(3,3),32}&\mu^+_{\GU(3,3),33}\\
   \end{bmatrix}^{\frac{l-t_k}{2}}.
\]
Then 
\[
    D'_{t_k,l}-D_{t_k,0,l,l,l}\in \left(\fp^+_{\Sp(4)},\fp^+_{\U(1,1)}\right)\bC[\fp^+_{\GU(3,3)}],
\]
where we view $\fp^+_{\Sp(4)}\times\fp^+_{\U(1,1)}$ as a subalgebra of $\fp^+_{\GU(3,3)}$ via the embedding \eqref{eq:gp-ebd}. Because $\Bes^{\cD_{l,l}}_{\bS,\triv}\left(\begin{bsm}&\bid_2\\ \bid_2\end{bsm}\cdot\varphi_{\ord,\infty}\right)$ (resp. $\Whi^{\cD_l,\Xi^{-c}_\infty}_{\bbc}(\begin{bsm}&1\\1\end{bsm}\cdot f_{\ord,\infty})$) is annihilated by $\fp^-_{\GSp(4)}$ (resp. $\fp^-_{\GU(1,1)}$), we have
\begin{align*}
   &Z_\infty\left(D_{t_k,0,l,l,l}\cdot \dsec^{t_k}_\infty(s,\chi,\Xi),\,\Bes^{\cD_{l_1,l_2}}_{\bS,\triv}\left(\begin{bsm}&\bid_2\\ \bid_2\end{bsm}\cdot \varphi_{\ord,\infty}\right),\Whi^{\cD_l,\Xi^{-c}_\infty}_{\bbc}\left(\begin{bsm}&1\\1\end{bsm} \cdot f_{\ord,\infty}\right)\right)\\
   =&\,Z_\infty\left(D'_{t_k,l}\cdot \dsec^{t_k}_\infty(s,\chi,\Xi),\,\Bes^{\cD_{l_1,l_2}}_{\bS,\triv}\left(\begin{bsm}&\bid_2\\ \bid_2\end{bsm}\cdot \varphi_{\ord,\infty}\right),\Whi^{\cD_l,\Xi^{-c}_\infty}_{\bbc}\left(\begin{bsm}&1\\1\end{bsm} \cdot f_{\ord,\infty}\right)\right).
\end{align*}
On the other hand, by \cite[Theorem~12.13,Lemma~13.9]{Sh00}, we have
{\small
\begin{align*}
   D'_{t_k,l}\cdot \dsec^{t_k}_\infty(s,\chi,\Xi)
   =(2\pi i)^{\frac{3t_k-3l}{2}}\cdot
   i^{\frac{t_k+3l}{2}} \prod_{j=0}^2 \frac{\Gamma\left(s+\frac{l+3}{2}-j\right)}{\Gamma\left(s+\frac{t_k+3}{2}-j\right)} 
   \cdot \dsec^{l}_\infty(s,\chi,\Xi).
\end{align*}
}
Therefore,
{\small
\begin{align*}
   I_\infty(s)
   =&\,2^{-3(s+\frac{l-1}{2})} \pi^{-3(s+\frac{l+1}{2})}  i^{t_k}
   \prod_{j=0}^2 \Gamma\left(s+\frac{l+3}{2}-j\right) \\
   &\times Z_\infty \left(\dsec^{l}_\infty(s,\chi,\Xi),\,\Bes^{\cD_{l_1,l_2}}_{\bS,\triv}\left(\begin{bsm}&\bid_2\\ \bid_2\end{bsm}\cdot \varphi_{\ord,\infty}\right),\Whi^{\cD_l,\Xi^{-c}_\infty}_{\bbc}\left(\begin{bsm}&1\\1\end{bsm} \cdot f_{\ord,\infty}\right)\right).
\end{align*}
}
Plugging in the formula in Proposition~\ref{eq:Zlinf} for $t=l$, we get
{\small
\begin{align*}
   I_\infty(s)
   &= |\bbc|^{-s-\frac{3l}{2}+\frac{3}{2}}|\alphaS-\alphabS|^{-2s-l}\cdot 2^{-4s-3l} \cdot i^{t_k-l} \cdot \pi^{-4s-3l+2}\\
   &\quad\times \Gamma\left(s+\frac{3l-3}{2}\right)\Gamma\left(s+\frac{l+1}{2}\right)\Gamma\left(s+\frac{l-1}{2}\right)^2\\
   &= |\bbc|^{-s-\frac{3l}{2}+\frac{3}{2}}|\alphaS-\alphabS|^{-2s-l}\cdot 2^{-6} e^{\frac{\pi i}{2}(4s+t_k+2l+2)} \cdot E_\infty\left(s+\frac{\epsilon}{2},\tilde{\Pi}\times\tilde{\pi}\times\chi\right).
\end{align*}
}
Then
{\small
\begin{equation}\label{eq:Iinftys}
\begin{aligned}
   I_\infty(k,\Pi_\infty,\pi_\infty,\Lambda_\infty)
   &= (-1)^k\bbc^{\frac{2k+\epsilon+3l}{2}}\left(-(\alphaS-\alphabS)^2\right)^{\frac{2k+\epsilon+l}{2}}\cdot I_\infty\left(\frac{2k+\epsilon-1}{2}\right)\\
   &=|\bbc^2(\alphaS-\alphabS)|\cdot 2^{-6} i^{\epsilon}
   \cdot E_\infty\left(k
   +\frac{\epsilon}{2}+1,\Pi\times\pi\right).
\end{aligned}
\end{equation}
}
\eqref{eq:int-sc} follows from plugging \eqref{eq:Iinftys} into \eqref{eq:interp}.
\end{proof}

To see (iii) of Theorem~\ref{thm:intro}, since $\varphi_\ord$ is cuspidal holomorphic invariant under $\begin{bmatrix}\bid_2&\Sym_2(\bZ_v)\\0&\bid_2\end{bmatrix}$ for all finite places $v$, there exists $\bS=\begin{bmatrix}\bba&\frac{\bbb}{2}\\\frac{\bbb}{2}&\bbc\end{bmatrix}\in \Sym^*_2(\bZ)_{>0}$ such that its Fourier coefficient indexed by $\bS$ is nonzero. Then one can choose $\Lambda$ such that $\bes_{\bS,\Lambda}(\varphi_\ord)\neq 0$. If $\Pi_p$ does not belong to IIa or IVa in the classification in \cite{RSLocal}, (ii) of Proposition~\ref{prop:Zord} implies $\bes^\dagger_{\bS,\Lambda}(\varphi_\ord)\neq 0$. Let $f_0\in\pi$ be the $p$-ordinary modular form new at every finite place $v\neq p$. Then $\whi_1(f_0)\neq 0$. Writing  $\bbc=mp^t$ with $p\nmid m$, let $f_\ord=\begin{bmatrix}m\\&1\end{bmatrix}^{-1}_f \cdot f_0$, then $f_\ord$ is $p$-ordinary and 
\[
    \whi_\bbc(f_\ord)=\lambda_p(f_\ord)^t\cdot  \whi_m(f_\ord)=\lambda_p(f_\ord)^t\cdot \whi_1(f_0)\neq 0.
\]

\bibliographystyle{E:/Dropbox/Zheng/LatexFiles/00Bibfiles/halpha}
\bibliography{E:/Dropbox/Zheng/LatexFiles/00Bibfiles/BiStd}
\end{document}